\documentclass[a4paper]{article}
\usepackage{enumitem}
\usepackage{calc}
\usepackage{hyperref}
\usepackage{upgreek}

\usepackage{lmodern} 

\usepackage{graphicx} 

\usepackage{amsmath, accents}
\usepackage{amssymb,tikz}
\usepackage{pict2e}
\usepackage{amsthm}
\usepackage{amscd}
\usepackage{mathrsfs}
\usepackage{todonotes}

\usetikzlibrary{calc} 

\usepackage{yfonts}

\usepackage{marvosym}

\newtheorem{theorem}			     {Theorem} [section]
	
\newtheorem{corollary}	  [theorem]	 {Corollary}	
\newtheorem{lemma}	      [theorem]  {Lemma}		
\theoremstyle{definition}

\newtheorem{remark} {Remark}

\newtheorem{assumptions}[theorem]{Assumptions}

\newcommand{\C}{\mathbb{C}}

\newcommand{\N}{\mathbb{N}}

\newcommand{\im}{\mathrm{Im}\,}
\newcommand{\re}{\mathrm{Re}\,}

\newcommand{\Ai}{{\rm Ai}}

\usepackage{tikz}
\usetikzlibrary{arrows}
\usetikzlibrary{decorations.pathmorphing}
\usetikzlibrary{decorations.markings}
\usetikzlibrary{patterns}
\usetikzlibrary{automata}
\usetikzlibrary{positioning}
\usepackage{tikz-cd}

\let\oldbibliography\thebibliography
\renewcommand{\thebibliography}[1]{\oldbibliography{#1}
\setlength{\itemsep}{-3pt}}

\usepackage{pgfplots}

\numberwithin{equation}{section}

\def\ds{\displaystyle}
\def\bigO{{\cal O}}
\begin{document}
\title{The hard-to-soft edge transition: exponential moments, \\ central limit theorems and rigidity}
\author{Christophe Charlier and Jonatan Lenells}
\date{Department of Mathematics, KTH Royal Institute of Technology, \\
100 44 Stockholm, Sweden. \\ \smallskip
E-mail: cchar@kth.se, jlenells@kth.se}

\maketitle

\begin{abstract}
The local eigenvalue statistics of large random matrices near a hard edge transitioning into a soft edge are described by the Bessel process associated with a large parameter $\alpha$. For this point process, we obtain 1) exponential moment asymptotics, up to and including the constant term, 2) asymptotics for the expectation and variance of the counting function, 3) several central limit theorems and 4) a global rigidity upper bound. 
\end{abstract}

\bigskip




\section{Introduction and statement of results}

The hard-to-soft edge transition is a universal phenomenon which arises in a wide class of unitary invariant random matrix ensembles \cite{BorodinForrester}, but to make the explanation more concrete we are going to focus on Wishart matrices. Let $X$ denote an $n \times (n+\alpha)$ matrix whose elements are independent standard complex Gaussian random variables, and consider the matrix $XX^{*}$, where $X^*$ is the conjugate transpose of $X$. Such matrices were first studied in 1928 by Wishart \cite{Wishart}, and since then have found applications in various areas, such as numerical analysis \cite{Edelman}, information theory \cite{Telatar} and finance \cite{BaiSil}. Of particular interest is the local statistics of the smallest eigenvalues of $XX^{*}$ in the large dimensional limit. There are two distinct regimes that have been observed \cite{MarPas}. In the limit $n \to + \infty$ while $\alpha$ is kept fixed, the smallest eigenvalues accumulate near the ``hard edge" $0$, and after rescaling, a limiting point process arises \cite{For1,ForNag}, known as the Bessel point process (which depends on $\alpha$). On the other hand, if $\alpha/n$ tends to a positive constant as $n \to + \infty$, then the smallest eigenvalues are pushed away from $0$ and instead are located near a ``soft edge". In this case, after rescaling, they give rise to the Airy point process \cite{For1,ForNag} (which is independent of $\alpha$). The hard-to-soft edge transition corresponds to the regime where $\alpha/n \to 0$ but $\alpha \to + \infty$, and is described by the Bessel process associated with a large parameter $\alpha$. 

\medskip The above example of Wishart matrices only makes sense for $\alpha \in \mathbb{N}_{\geq 0}:=\{0,1,2,\ldots\}$, but in fact the Bessel process can be defined for all values $\alpha \in (-1,+\infty)$. For non-integer values of $\alpha$, it appears for example at the hard edge of the Laguerre Unitary Ensemble \cite{Vanlessen}.

\medskip It is known, see \cite[Theorem 1]{BorodinForrester} or \eqref{BoroForr scaling} below, that the gap probabilities of the Bessel process, when properly rescaled, converge as $\alpha \to +\infty$ to the gap probabilities of the Airy process. In this paper, we establish various other properties of the large $\alpha$ Bessel process. 

\medskip The Bessel point process is a determinantal point process on $[0,+\infty)$ whose kernel is given by
\begin{equation}\label{Bessel kernel}
K_{\alpha}^{\mathrm{Be}}(x,y) = \frac{J_{\alpha}(\sqrt{x})\sqrt{y}J_{\alpha}^{\prime}(\sqrt{y})-\sqrt{x}J_{\alpha}^{\prime}(\sqrt{x})J_{\alpha}(\sqrt{y})}{2(x-y)}, \qquad \alpha > -1,
\end{equation}
where $J_\alpha$ is the Bessel function of the first kind of order $\alpha$. As $\alpha$ increases, the Bessel process increasingly favors configurations whose points are further away from $0$. We refer to \cite{Soshnikov,Borodin,Johansson} for several surveys on determinantal point processes.


\medskip We emphasize that the Bessel process studied in this paper arises near hard edges of unitary invariant random matrix ensembles. The generalization of this process to the so-called $\beta$-ensembles has been introduced in \cite{RamirezRider} and is not determinantal. Its transition to the $\beta$ soft edge has been studied in \cite{RamirezRider, RamirezRider2, DLV2020}. We also mention that other types of hard-to-soft edge transitions than the one considered here have been studied in e.g. \cite{ClaeysKuijlaars, KuijMartFinWie, XDZ2014}.

\medskip Let us introduce the parameters
\begin{equation}\label{def of m vecu and vecx}
m \in \N_{>0}, \quad \vec{u}=(u_1,\dots,u_m) \in \mathbb{R}^{m} \quad \mbox{ and } \quad \vec{x} = (x_{1},\dots,x_{m}) \in \mathbb{R}_{\mathrm{ord}}^{+,m},
\end{equation}
where $\mathbb{R}_{\mathrm{ord}}^{+,m} := \{\vec{x} = (x_{1},\dots,x_{m}): 0 < x_1 < x_2 < \dots < x_m < +\infty\}$. Our main interest in this work lies in the generating function of the Bessel process, which can be written as the following exponential moments
\begin{equation}\label{def of E alpha}
E_{\alpha}(\vec{x},\vec{u}) := \mathbb{E}\Bigg[ \prod_{j=1}^{m} e^{u_{j}N_{\alpha}(x_{j})} \Bigg],
\end{equation}
where $N_{\alpha}(x)$ denotes the random variable that counts the number of points $\leq x$ in the Bessel process. It is known that $E_{\alpha}(\vec{x},\vec{u})$ can be naturally expressed in terms of the solution to a system of $m$ coupled Painlev\'{e} V equations \cite{TraWidLUE, ChDoe}. The goal of this paper is to obtain precise \textit{exponential moment asymptotics} for $E_{\alpha}(r\vec{x},\vec{u})$ as $r \to + \infty$ in the critical regime where $\alpha \to + \infty$. Exponential moment asymptotics have attracted considerable attention recently, partly due to their relevance for the global rigidity of the associated point process, see Section \ref{subsection:rigidity}. Such asymptotics have been obtained for various point processes, see \cite{BW1983, BB1995, BDIK2015, CharlierSine} for the sine process, \cite{BIP2019,CharlierBessel} for the Bessel process with $\alpha$ fixed (or bounded), \cite{BothnerBuckingham,ChCl3} for the Airy process, \cite{ChCl4} for the Meijer-$G$ and Wright's generalized Bessel process, and \cite{DXZ2020} for the Pearcey process. In this work, we obtain asymptotics for $E_{\alpha}(r\vec{x},\vec{u})$ as $r \to + \infty$ uniformly in $\alpha$ and $\vec{x}$ in a way that allows us to discuss the transition to the Airy exponential moment asymptotics, and to obtain a precise matching, up to and including the constant term, with the known exponential moment asymptotics from \cite{BIP2019,CharlierBessel} for the Bessel process with $\alpha$ bounded. More precisely, our main result is stated in Theorem \ref{thm:s1 neq 0} and describes the asymptotics of $E_{\alpha}(r\vec{x},\vec{u})$ in the following three regimes:
\begin{enumerate}
\item \label{item 1} $r \to + \infty$, $\alpha = a\sqrt{r} \to + \infty$ with $a,x_{1},\ldots,x_{m} > 0$ fixed (or mildly varying),
\item \label{item 2} $r \to + \infty$, $\alpha = a\sqrt{r} \to + \infty$, $a \to 0$ with $x_{1},\ldots,x_{m} > 0$ fixed (or mildly varying),
\item \label{item 3} $r \to + \infty$, $\alpha = a\sqrt{r} \to + \infty$, $a$ fixed, and the points $x_{j}$ converging sufficiently slowly to $a^{2}$.
\end{enumerate}
The second regime, in which $a \to 0$, requires a different (and more delicate) proof than the other two, and allows for a matching with the results of \cite{BIP2019,CharlierBessel} for the Bessel process with $\alpha$ bounded, see Corollary \ref{coro:conv to Bessel fixed alpha}. The third regime is needed to discuss the transition to the Airy exponential moment asymptotics obtained in \cite{BothnerBuckingham,ChCl3}, see Corollary \ref{coro:conv to Airy}. 

\medskip Theorem \ref{thm:s1 neq 0} encodes significant information about the hard-to-soft edge transition: we can deduce from it large $r$ asymptotics for the expectation and the variance of $N_{\alpha}(rx)$, see Corollary \ref{coro: s1>0 consequence of thm}, several central limit theorems (CLTs), see Corollaries \ref{coro: CLTs} and \ref{coro: CLT 2}, and an upper bound for the global rigidity of the process, see Theorem \ref{thm:rigidity Bessel}. 

\begin{theorem}[Exponential moment asymptotics for the large $\alpha$ Bessel process]\label{thm:s1 neq 0} 
\hspace{3cm}
\begin{enumerate}
\item\label{item 1 in thm} Let $m \in \mathbb{N}_{>0}$,  and let
\begin{align}\label{vec u and vec x in thm}
a \in (0,+\infty), \quad \vec{u}=(u_1,\dots,u_m) \in \mathbb{R}^{m} \quad \mbox{ and } \quad \vec{x} = (x_{1},\dots,x_{m}) \in \mathbb{R}_{\mathrm{ord}}^{+,m}, 
\end{align}
be such that
\begin{align*}
0 < x_{1} < \cdots < x_{n-1} < a^{2} < x_{n} < \cdots < x_{m},
\end{align*}
with $n \in \{1,\dots,m+1\}$, $x_{0}:=0$, $x_{m+1}:=+\infty$, and for $r>0$, define $\alpha = a \sqrt{r}$. As $r \to + \infty$, we have
\begin{multline}\label{F asymptotics thm s1 neq 0}
E_{\alpha}(r\vec{x},\vec{u}) = \exp \Bigg( \sum_{j=n}^{m}  u_{j} \mu_{\alpha}(rx_{j}) + \sum_{j=n}^{m} \frac{u_{j}^{2}}{2} \sigma_{\alpha}^{2}(rx_{j}) + \sum_{n \leq j < k \leq m} u_{j} u_{k} \Sigma_{a}(x_{k},x_{j})  \\ + \sum_{j=n}^{m} \log\big(G(1+\tfrac{u_{j}}{2\pi i})G(1-\tfrac{u_{j}}{2\pi i})\big) + \bigO \bigg( \frac{\log r}{\sqrt{r}} \bigg) \Bigg),
\end{multline}
where $G$ is Barnes' $G$-function, and $\mu_{\alpha}$, $\sigma_{\alpha}^{2}$ and $\Sigma_{a}$ are given by
\begin{align}
& \mu_{\alpha}(rx) = \frac{1}{\pi}\int_{\alpha^{2}}^{rx} \frac{\sqrt{u-\alpha^{2}}}{2u}du =\frac{\sqrt{r}}{\pi} \int_{a^{2}}^{x}\frac{\sqrt{u-a^{2}}}{2u}du, \label{def of mu} \\
& \sigma_{\alpha}^{2}(rx) = \frac{\log \big(4 (rx-\alpha^{2})^{3/2} (rx)^{-1}\big)}{2 \pi^{2}} = \frac{\log r}{4\pi^{2}} + \frac{\log \big(4 (x-a^{2})^{3/2} x^{-1} \big)}{2 \pi^{2}} , \label{def of sigma} \\
& \Sigma_{a}(x_{k},x_{j}) = \frac{1}{2\pi^{2}} \log\left(\frac{\sqrt{x_{j}-a^{2}}+\sqrt{x_{k}-a^{2}}}{|\sqrt{x_{j}-a^{2}}-\sqrt{x_{k}-a^{2}}|}\right). \label{def of Sigma}
\end{align}
Furthermore, the asymptotics \eqref{F asymptotics thm s1 neq 0} are uniform for $(x_{1},\ldots,x_{n-1},a^{2},x_{n},\ldots,x_{m})$ in compact subsets of $\mathbb{R}_{\mathrm{ord}}^{+,m+1}$, and uniform for $u_{1},\ldots,u_{m}$ in compact subsets of $\mathbb{R}$.
\item\label{item 2 in thm} Let $m \in \mathbb{N}_{>0}$, $\vec{u} \in \mathbb{R}^{m}$ and $\vec{x} \in \mathbb{R}_{\mathrm{ord}}^{+,m}$, and for $r>0$ and $a\in (0,x_{1})$, define $\alpha = a \sqrt{r}$. The asymptotics of $E_{\alpha}(r\vec{x},\vec{u})$ as $r \to + \infty$ and simultaneously $a \to 0$, $a\sqrt{r}\to + \infty$ are also given by the right-hand side of \eqref{F asymptotics thm s1 neq 0} (necessarily with $n=1$). Furthermore, these asymptotics are uniform for $u_{1},\ldots,u_{m}$ in compact subsets of $\mathbb{R}$, and uniform for $\vec{x}$ in compact subsets of $\mathbb{R}_{\mathrm{ord}}^{+,m}$.
\item\label{item 3 in thm} Let $m \in \mathbb{N}_{>0}$ and $a \in (0,+\infty)$ be fixed, let $\vec{u} \in \mathbb{R}^{m}$ and $\vec{x} \in \mathbb{R}_{\mathrm{ord}}^{+,m}$, and for $r>0$, define $\alpha = a \sqrt{r}$. Assume that $0 < a^{2} < x_{1} < \cdots < x_{m}$ and that the points $\{x_j\}_1^m$ tend to $a^2$ at the same sufficiently slow rate in the sense that the following hold:
\begin{subequations}\label{Airy regime in the main first thm}
\begin{align}
& |x_{m}-a^{2}| \to 0 \mbox{ as } r \to +\infty, 	
	\\ \label{Airy regime in the main first thm b}
& \frac{x_j - a^2}{x_m - a^2}  \text{ stays in a bounded subset of $(0,+\infty)$ for each $j = 1, \dots, m-1$},
	\\
& \frac{\log r}{|x_{m}-a^{2}|^{4}\sqrt{r}} \to 0 \mbox{ as } r \to +\infty.
\end{align}
\end{subequations}
In this regime, the asymptotics of $E_{\alpha}(r\vec{x},\vec{u})$ as $r \to + \infty$ are given by the right-hand side of \eqref{F asymptotics thm s1 neq 0} but with the error term $\bigO(\log r/\sqrt{r})$ replaced by
\begin{align}\label{error term Airy}
\bigO \bigg( \frac{\log r}{|x_{m}-a^{2}|^{4}\sqrt{r}} \bigg).
\end{align}
Furthermore, these asymptotics are also uniform for $u_{1},\ldots,u_{m}$ in compact subsets of $\mathbb{R}$.
\end{enumerate}
The asymptotic formulas in each of the above three regimes can be differentiated any number of times with respect to $u_{1},\ldots,u_{m}$ at the cost of increasing slightly the error term as follows. Let $\widetilde{E}_{\alpha}(r\vec{x},\vec{u})$ be the right-hand side of \eqref{F asymptotics thm s1 neq 0} without the error term and let $\mathcal{E}=\log E_{\alpha}(r\vec{x},\vec{u})- \log \widetilde{E}_{\alpha}(r\vec{x},\vec{u})$ be the error term. If $k_{1},\ldots,k_{m}\in \mathbb{N}_{\geq 0}$, $k=k_{1}+\ldots+k_{m}\geq 1$ and $\partial_{u}^{k}=\partial_{u_{1}}^{k_{1}}\ldots \partial_{u_{m}}^{k_{m}}$, then
\begin{align}
& \partial_{u}^{k} \mathcal{E}=\bigO \bigg( \frac{(\log r)^{k}}{\sqrt{r}} \bigg) & & \mbox{for regimes \ref{item 1 in thm} and \ref{item 2 in thm}}, \label{der of error in regimes 1 and 2} \\
& \partial_{u}^{k} \mathcal{E}=\bigO \bigg( \frac{(\log r)^{k}}{|x_{m}-a^{2}|^{4}\sqrt{r}} \bigg) & & \mbox{for regime \ref{item 3 in thm}}. \label{der of error in regime 3}
\end{align}
\end{theorem}
\begin{remark}
If $a^{2}>x_{m}$, i.e. $n=m+1$, then each sum in \eqref{F asymptotics thm s1 neq 0} should be interpreted as $0$.
\end{remark}
As stated in part \ref{item 2 in thm} of Theorem \ref{thm:s1 neq 0}, the error term in \eqref{F asymptotics thm s1 neq 0} does not get worse as $a \to 0$. It is however important in our proof that $\alpha = a\sqrt{r} \to + \infty$. The asymptotics of $E_{\alpha}(r\vec{x},\vec{u})$ as $r \to + \infty$ with $\alpha$ bounded are not covered by Theorem \ref{thm:s1 neq 0}, but they have been obtained in \cite[eq (1.35)]{BIP2019} for $m=1$ and in \cite[Theorem 1.1]{CharlierBessel} for $m \geq 2$. These asymptotics are given by the right-hand side of \eqref{F asymptotics thm s1 neq 0} with $n=1$ and with $\mu_{\alpha}$, $\sigma_{\alpha}^{2}$ and $\Sigma_{a}$ replaced by 
\begin{align}\label{def of mu sigma cov alpha bounded}
& \tilde{\mu}_{\alpha}(rx) = \frac{\sqrt{rx}}{\pi}-\frac{\alpha}{2}, & & \tilde{\sigma}_{\alpha}^{2}(rx) = \frac{\log (4\sqrt{rx})}{2\pi^{2}}, & & \widetilde{\Sigma}(x_{k},x_{j}) = \frac{1}{2\pi^{2}} \log\left(\frac{\sqrt{x_{j}}+\sqrt{x_{k}}}{|\sqrt{x_{j}}-\sqrt{x_{k}}|}\right),
\end{align}
respectively. Corollary \ref{coro:conv to Bessel fixed alpha} below shows that Theorem \ref{thm:s1 neq 0} matches explicitly with \cite[Theorem 1.1]{CharlierBessel} in the regime $a \to 0$, $\alpha=a\sqrt{r}\to + \infty$, up to an error $\bigO \big( \frac{\log r}{\sqrt{r}} + a^{2}\sqrt{r} \big)$. In other words, Theorem \ref{thm:s1 neq 0} and \cite[Theorem 1.1]{CharlierBessel} taken together describe the asymptotics of $E_{a\sqrt{r}}(r\vec{x},\vec{u})$ as $r \to + \infty$, up to and including the constant, uniformly for $a \in [\frac{\tilde{\alpha}}{\sqrt{r}},a_{0}]$ for any fixed $a_{0} \in (0,x_{1})$ and fixed $\tilde{\alpha} \in (-1,+\infty)$. 
\begin{corollary}[Matching with the Bessel exponential moment asymptotics with $\alpha$ bounded]\label{coro:conv to Bessel fixed alpha}
Let $m \in \mathbb{N}_{>0}$, $\vec{u} \in \mathbb{R}^{m}$ and $\vec{x}\in \mathbb{R}_{\mathrm{ord}}^{+,m}$, and for $r>0$ and $a>0$, define $\alpha = a \sqrt{r}$. The asymptotics of $E_{\alpha}(r\vec{x},\vec{u})$ as $r \to + \infty$ and simultaneously $a \to 0$, $a\sqrt{r}\to + \infty$ can be written as
\begin{multline}\label{F asymptotics thm s1 neq 0 Bessel alpha bounded}
E_{\alpha}(r\vec{x},\vec{u}) = \exp \Bigg( \sum_{j=1}^{m}  u_{j} \tilde{\mu}_{\alpha}(rx_{j}) + \sum_{j=1}^{m} \frac{u_{j}^{2}}{2} \tilde{\sigma}_{\alpha}^{2}(rx_{j}) + \sum_{1 \leq j < k \leq m} u_{j} u_{k} \widetilde{\Sigma}(x_{k},x_{j})  \\ + \sum_{j=1}^{m} \log G(1+\tfrac{u_{j}}{2\pi i})G(1-\tfrac{u_{j}}{2\pi i}) + \bigO \bigg( \frac{\log r}{\sqrt{r}} \bigg) + \bigO \big( a^{2}\sqrt{r} \big) \Bigg),
\end{multline}
where $\tilde{\mu}_{\alpha}$, $\tilde{\sigma}_{\alpha}^{2}$, $\widetilde{\Sigma}$ are defined in \eqref{def of mu sigma cov alpha bounded}. Furthermore, these asymptotics are uniform for $u_{1},\ldots,u_{m}$ in compact subsets of $\mathbb{R}$, and uniform for $\vec{x}$ in compact subsets of $\mathbb{R}_{\mathrm{ord}}^{+,m}$.
\end{corollary}
\begin{proof}
Since the error term in Theorem \ref{thm:s1 neq 0} is uniform as $a \to 0$, the claim follows after letting $a \to 0$ in \eqref{def of mu}--\eqref{def of Sigma} with $n = 1$, and noting that
\begin{align}\label{asymptotics as a is small mean var cov}
& \hspace{-0.2cm} \mu_{\alpha}(rx) = \tilde{\mu}_{\alpha}(rx) + \bigO(a^{2}\sqrt{r}), \; \sigma_{\alpha}^{2}(rx) = \tilde{\sigma}_{\alpha}^{2}(rx) + \bigO( a^{2} ), \; \Sigma_{a}(x_{k},x_{j}) = \widetilde{\Sigma}(x_{k},x_{j}) + \bigO( a^{2} )
\end{align}
uniformly for $\vec{x}$ in compact subsets of $\mathbb{R}_{\mathrm{ord}}^{+,m}$ and for $r \geq 1$.
\end{proof}
The Airy process is a determinantal point process on $\mathbb{R}$ whose kernel is given by
\begin{align*}
K_{\Ai}(x, y) = \frac{ \Ai(x) \Ai'(y) - \Ai'(x) \Ai(y) }{ x - y }, \qquad x,y \in \mathbb{R},
\end{align*}
where $\Ai$ denotes the Airy function. It has been proved by Borodin and Forrester in \cite[Theorem 1]{BorodinForrester} that 
\begin{align}\label{BoroForr scaling}
\mathbb{P}_{\mathrm{Bessel}}\Big[ \mbox{there is a gap on }[0,\alpha^{2}+2^{\frac{2}{3}}\alpha^{\frac{4}{3}}y+\bigO(\alpha)] \Big] \to \mathbb{P}_{\mathrm{Airy}}\Big[ \mbox{there is a gap on }[-y,+\infty] \Big]
\end{align}
as $\alpha \to \infty$, for any $y \in \mathbb{R}$. Here ``a gap" on an interval $A \subset \mathbb{R}$ means that no points in the process fall in $A$. It is important to note that $y$ is fixed in \eqref{BoroForr scaling}.
In Corollary \ref{coro:conv to Airy} below, we show an analogue of \eqref{BoroForr scaling} at the level of the exponential moment asymptotics, i.e. we consider the double scaling limit where both $\alpha \to +\infty$ and simultaneously $y \to + \infty$. Let $N_{\mathrm{Ai}}(y)$ denote the random variable that counts the number of points in $[-y,+\infty)$ in the Airy process, $y \in \mathbb{R}$, and consider the exponential moment 
\begin{align*}
E_{\mathrm{Ai}}(\vec y,\vec u):=\mathbb{E}\Bigg[\prod_{j = 1}^m e^{u_j N_{\mathrm{Ai}}(y_j)}\Bigg], \qquad \vec{y} = (y_{1},\ldots,y_{m}) \in \mathbb{R}^{m}, \; \vec{u} = (u_{1},\ldots,u_{m}) \in \mathbb{R}^{m}.
\end{align*} 
The asymptotics for $E_{\mathrm{Ai}}(r_{\mathrm{Ai}}\vec y,\vec u)$ as $r_{\mathrm{Ai}} \to + \infty$ have been obtained in \cite{BothnerBuckingham} for $m=1$ and in \cite[Theorem 1.1]{ChCl3} for $m \geq 2$, and are as follows: 
\begin{multline}\label{F asymptotics thm s1 neq 0 Airy outside of thm}
E_{\mathrm{Ai}}(r_{\mathrm{Ai}}\vec y,\vec u) = \exp \Bigg( \sum_{j=1}^{m}  u_{j} \mu_{\mathrm{Ai}}(r_{\mathrm{Ai}}y_{j}) + \sum_{j=1}^{m} \frac{u_{j}^{2}}{2} \sigma_{\mathrm{Ai}}^{2}(r_{\mathrm{Ai}}y_{j}) + \sum_{1 \leq j < k \leq m} u_{j} u_{k} \Sigma_{\mathrm{Ai}}(y_{k},y_{j})  \\ + \sum_{j=1}^{m} \log G(1+\tfrac{u_{j}}{2\pi i})G(1-\tfrac{u_{j}}{2\pi i}) + \bigO \bigg( \frac{\log r_{\mathrm{Ai}}}{r_{\mathrm{Ai}}^{3/2}} \bigg) \Bigg),
\end{multline}
where
\begin{align}\label{def of mean var cov Airy}
& \mu_{\mathrm{Ai}}(y) = \frac{2}{3\pi}y^{3/2}, & & \sigma_{\mathrm{Ai}}^{2}(y) = \frac{3}{4\pi^2}\log (4y), & & \Sigma_{\mathrm{Ai}}(y_{k},y_{j}) = \frac{1}{2\pi^{2}} \log\left(\frac{\sqrt{y_{j}}+\sqrt{y_{k}}}{|\sqrt{y_{j}}-\sqrt{y_{k}}|}\right).
\end{align}
Furthermore, the error term in \eqref{F asymptotics thm s1 neq 0 Airy outside of thm} is uniform for $u_{1},\ldots,u_{m}$ in compact subsets of $\mathbb{R}$. 
The result \eqref{BoroForr scaling} suggests that the asymptotics of $E_{\alpha}(r\vec{x},\vec{u})$ as $r \to + \infty$, $\alpha \to + \infty$ and simultaneously 
\begin{align*}
& rx_{j} = \alpha^{2} + 2^{\frac{2}{3}} \alpha^{\frac{4}{3}} r_{\mathrm{Ai}}y_{j}, \qquad j=1,\ldots,m, & & 0<y_{1}< \cdots<y_{m},
\end{align*}
with $r_{\mathrm{Ai}} \to + \infty$ at a sufficiently slow speed should somehow be related to the asymptotics of $E_{\mathrm{Ai}}(r_{\mathrm{Ai}}\vec y,\vec u)$. Corollary \ref{coro:conv to Airy} below confirms this expectation. 
\begin{corollary}[Matching with the Airy exponential moment asymptotics]\label{coro:conv to Airy}
Let $m \in \mathbb{N}_{>0}$, $a \in (0,+\infty)$ and $0<y_{1}< \cdots<y_{m}$ be fixed. For $r>0$, define $\alpha = a \sqrt{r}$, and let $\vec{u} \in \mathbb{R}^{m}$ and $\vec{x} \in \mathbb{R}_{\mathrm{ord}}^{+,m}$ be such that
\begin{align}\label{scaling inside coro Airy 1}
& x_{j} = \frac{\alpha^{2}}{r} + 2^{\frac{2}{3}}  \frac{\alpha^{\frac{4}{3}}r_{\mathrm{Ai}}}{r}y_{j} = a^{2} + 2^{\frac{2}{3}}  \frac{a^{\frac{4}{3}}r_{\mathrm{Ai}}}{r^{\frac{1}{3}}}y_{j}, \qquad j=1,\ldots,m, \qquad  r_{\mathrm{Ai}}>0.
\end{align}
As $r \to + \infty$ and simultaneously
\begin{align}\label{scaling inside coro Airy 2}
M_r r^{\frac{5}{24}}(\log r)^{\frac{1}{4}} \leq r_{\Ai} \leq \frac{1}{M_r} r^{\frac{1}{3}}
\end{align}
where $M_r > 0$ tends to infinity at an arbitrarily slow rate as $r \to +\infty$, we have
\begin{multline}\label{F asymptotics thm s1 neq 0 Airy}
E_{\alpha}(r\vec{x},\vec{u}) = \exp \Bigg( \sum_{j=1}^{m}  u_{j} \mu_{\mathrm{Ai}}(r_{\mathrm{Ai}}y_{j})\bigg( 1+\bigO\Big(\frac{r_{\mathrm{Ai}}}{r^{\frac{1}{3}}}\Big) \bigg)  + \sum_{j=1}^{m} \frac{u_{j}^{2}}{2} \bigg( \sigma_{\mathrm{Ai}}^{2}(r_{\mathrm{Ai}}y_{j}) +\bigO\Big(\frac{r_{\mathrm{Ai}}}{r^{\frac{1}{3}}}\Big) \bigg)   \\ + \sum_{1 \leq j < k \leq m} u_{j} u_{k} \Sigma_{\mathrm{Ai}}(y_{k},y_{j}) + \sum_{j=1}^{m} \log G(1+\tfrac{u_{j}}{2\pi i})G(1-\tfrac{u_{j}}{2\pi i}) + \bigO \bigg( \frac{r^{\frac{5}{6}}\log r}{r_{\mathrm{Ai}}^{4}} \bigg) \Bigg).
\end{multline}
Furthermore, the error term is uniform for $u_{1},\ldots,u_{m}$ in compact subsets of $\mathbb{R}$.
\end{corollary}
\begin{proof}
Substituting $rx=\alpha^{2}+2^{\frac{2}{3}}\alpha^{\frac{4}{3}}r_{\mathrm{Ai}}y$ in \eqref{def of mu}--\eqref{def of Sigma} and letting $\frac{r_{\mathrm{Ai}}}{\alpha^{\frac{2}{3}}} \to 0$, we get
\begin{align}\label{asymp mean var cov airy}
& \hspace{-0.2cm} \mu_{\alpha}(rx) = \mu_{\mathrm{Ai}}(r_{\mathrm{Ai}}y)\bigg( 1+\bigO\Big(\frac{r_{\mathrm{Ai}}}{\alpha^{\frac{2}{3}}}\Big) \bigg), \; \sigma_{\alpha}^{2}(rx) = \sigma_{\mathrm{Ai}}^{2}(r_{\mathrm{Ai}}y) +\bigO\Big(\frac{r_{\mathrm{Ai}}}{\alpha^{\frac{2}{3}}}\Big), \; \Sigma_{a}(x_{k},x_{j}) = \Sigma_{\mathrm{Ai}}(y_{k},y_{j}).
\end{align}
The claim now follows directly from part \ref{item 3 in thm} of Theorem \ref{thm:s1 neq 0}.
\end{proof}
\subsection{Large $r$ asymptotics of $\mathbb{E}[N_{\alpha}(rx_{1})]$, $\mathrm{Var}[N_{\alpha}(rx_{1})]$ and $\mathrm{Cov}[ N_{\alpha}(rx_{1}),N_{\alpha}(rx_{2}) ]$}
It is directly seen from \eqref{def of E alpha} with $m=1$ and $m=2$ that
\begin{align}
& \partial_u \log E_{\alpha}(x_{1},u)|_{u=0} = \mathbb{E}[N_{\alpha}(x_{1})], \quad \partial_u^2 \log E_{\alpha}(x_{1},u)|_{u=0} = \mathrm{Var}[N_{\alpha}(x_{1})],
 \label{asymp u to 0 1}
 	\\
& \partial_u^2 \log\bigg( \frac{E_{\alpha}\big((x_{1},x_{2}),(u,u)\big)}{E_{\alpha}(x_{1},u)E_{\alpha}(x_{2},u)}\bigg)\bigg|_{u=0} = 2\, \mbox{Cov}(N_{\alpha}(x_{1}),N_{\alpha}(x_{2})). \label{asymp u to 0 2}
\end{align}
Recall that the asymptotics of $E_{\alpha}(rx_{1},u)$ and $E_{\alpha}\big((rx_{1},rx_{2}),(u,u)\big)$ as $r \to + \infty, \alpha \to + \infty$, which are given by Theorem \ref{thm:s1 neq 0} with $m=1$ and $m=2$, can be differentiated at the cost of increasing the error term as in \eqref{der of error in regimes 1 and 2}--\eqref{der of error in regime 3}. Hence, from \eqref{asymp u to 0 1}-\eqref{asymp u to 0 2}, we obtain the following asymptotics for $\mathbb{E}[N_{\alpha}(rx_{1})]$, $\mathrm{Var}[N_{\alpha}(rx_{1})]$ and $\mathrm{Cov}[ N_{\alpha}(rx_{1}),N_{\alpha}(rx_{2}) ]$ as $r \to + \infty, \alpha \to + \infty$.

\begin{corollary}\label{coro: s1>0 consequence of thm}
Let $0<x_{1}<x_{2}$ be fixed, let $a \in (0,x_{1})$, and for $r>0$, define $\alpha = a\sqrt{r}$. As $r \to + \infty$, we have
\begin{align}
& \mathbb{E}[N_{\alpha}(rx_{1})] = \mu_{\alpha}(rx_{1}) + \bigO \bigg( \frac{\log r}{\sqrt{r}} \bigg) = \frac{\sqrt{r}}{\pi} \int_{a^{2}}^{x_{1}}\frac{\sqrt{u-a^{2}}}{2u}du + \bigO \bigg( \frac{\log r}{\sqrt{r}} \bigg), \label{expected and variance asymp 1} \\
& \mathrm{Var}[N_{\alpha}(rx_{1})] = \sigma_{\alpha}^{2}(rx_{1}) + \frac{1 + \gamma_{\mathrm{E}}}{2 \pi^{2}} + \bigO \bigg( \frac{\log^{2} r}{\sqrt{r}} \bigg) \nonumber \\
& \hspace{2cm} = \frac{\log r}{4 \pi^{2}} + \frac{1 + \log \big(4 (x_{1}-a^{2})^{3/2} x_{1}^{-1} \big) + \gamma_{\mathrm{E}}}{2 \pi^{2}} + \bigO \bigg( \frac{\log^{2} r}{\sqrt{r}} \bigg), \label{expected and variance asymp 2} \\
& \mathrm{Cov}[ N_{\alpha}(rx_{1}),N_{\alpha}(rx_{2}) ] = \Sigma_{a}(x_{2},x_{1}) + \bigO \bigg( \frac{\log^{2} r}{\sqrt{r}} \bigg) = \frac{1}{2\pi^{2}} \log \left(\frac{\sqrt{x_{2}-a^{2}}+\sqrt{x_{1}-a^{2}}}{\sqrt{x_{2}-a^{2}}-\sqrt{x_{1}-a^{2}}}\right) + \bigO \bigg( \frac{\log^{2} r}{\sqrt{r}} \bigg), \nonumber
\end{align}
where $\gamma_{\mathrm{E}} \approx 0.5772$ is Euler's gamma constant. Furthermore, these asymptotics are uniform for $a$ in compact subsets of $[0,x_{1})$.
\end{corollary}
\begin{remark}\label{remark: conjecture}
The asymptotics \eqref{expected and variance asymp 2}, without the error term, have been obtained independently in the recent work \cite{SDoussMajSchehr}. The main result of \cite{SDoussMajSchehr} is a general asymptotic formula for the variance of the counting function of an arbitrary \textit{non-interacting fermion process}. Non-interacting fermion processes are determinantal point processes whose kernel can be written in the form 
\vspace{-0.2cm}\begin{align*}
K_{\mu}(x,y) = \sum_{k=1}^{n} \overline{\psi_{k}(x)}\psi_{k}(y), \qquad n = \#\{k: \epsilon_{k} \leq \mu \}, \quad \mu \in \mathbb{R},
\end{align*}\\[-0.4cm]
where $\epsilon_{1} \leq \epsilon_{2} \leq \epsilon_{3} \leq \cdots$ denote the eigenvalues of a Schr\"{o}dinger operator $\mathcal{H} = - \partial_{x}^{2}+V$ associated to a certain potential $V: \mathbb{R}\to \mathbb{R}\cup \{+\infty\}$, and $\psi_{k}$ denotes the eigenfunction associated to $\epsilon_{k}$. It has been observed, see \cite[Section 6.2]{LaCroix}, that the Bessel kernel, after a proper rescaling, arises as the limit $\mu \to + \infty$ of the correlation kernel of the non-interacting fermion process associated to the potential
\begin{align}\label{inverse square well potential}
V(x) = \begin{cases}
\frac{\alpha^{2}-1/4}{2x^{2}}, & \mbox{if } x >0, \\
+\infty, & \mbox{if } x \leq 0.
\end{cases}
\end{align}
The general formula of \cite{SDoussMajSchehr} specialized to the potential \eqref{inverse square well potential} agrees with \eqref{expected and variance asymp 2} up to and including the constant; this consistency has been verified in detail in \cite[Section V (ii)]{SDoussMajSchehr}. 
\end{remark}
The large $r$ asymptotics of $\mathbb{E}[N_{\alpha}(rx_{1})]$, $\mathrm{Var}[N_{\alpha}(rx_{1})]$ and $\mathrm{Cov}[ N_{\alpha}(rx_{1}),N_{\alpha}(rx_{2}) ]$ with $\alpha$ bounded have been obtained in \cite{SoshnikovSineAiryBessel} (for the leading terms) and in \cite[Corollary 1.2]{CharlierBessel} (for the subleading terms). Using \eqref{asymptotics as a is small mean var cov}, it is straightforward to verify that if $a \to 0$ sufficiently fast, the asymptotics of Corollary \ref{coro: s1>0 consequence of thm} agree with \cite[Corollary 1.2]{CharlierBessel}, up to and including the constant term. More precisely, we have the following.
\begin{corollary}[Matching with the Bessel asymptotics with $\alpha$ bounded]\label{coro: s1>0 consequence of thm Bessel bounded} 
Let $0<x_{1}<x_{2}$ be fixed, and for $a>0$, $r>0$, define $\alpha = a \sqrt{r}$. As $r \to + \infty$, $\alpha \to + \infty$ and $a^{2}\sqrt{r} \to 0$, we have 
\begin{align*}
& \mathbb{E}[N_{\alpha}(rx_{1})] = \tilde{\mu}_{\alpha}(rx_{1}) + \bigO \bigg( \frac{\log r}{\sqrt{r}} + a^{2}\sqrt{r} \bigg) , \\
& \mathrm{Var}[N_{\alpha}(rx_{1})] = \tilde{\sigma}_{\alpha}^{2}(rx_{1}) + \frac{1 + \gamma_{\mathrm{E}}}{2 \pi^{2}} + \bigO \bigg( \frac{\log^{2} r}{\sqrt{r}} + a^{2} \bigg), \nonumber \\
& \mathrm{Cov}[ N_{\alpha}(rx_{1}),N_{\alpha}(rx_{2}) ] = \widetilde{\Sigma}(x_{2},x_{1}) + \bigO \bigg( \frac{\log^{2} r}{\sqrt{r}} + a^{2} \bigg). \nonumber
\end{align*}
\end{corollary}
The following asymptotics related to the Airy process have been obtained in \cite{SoshnikovSineAiryBessel} (for the leading terms) and in \cite{ChCl3} (for the subleading terms), and are valid as $r_{\mathrm{Ai}} \to + \infty$ with fixed $0 <y_{1}<y_{2}$:
\begin{align*}
& \mathbb{E}[N_{\mathrm{Ai}}(r_{\mathrm{Ai}}y_{1})] = \mu_{\mathrm{Ai}}(r_{\mathrm{Ai}}y_{1}) + o(1), \\
& \mathrm{Var}[N_{\mathrm{Ai}}(r_{\mathrm{Ai}}y_{1})] = \sigma_{\mathrm{Ai}}^{2}(r_{\mathrm{Ai}}y_{1}) + \frac{1+\gamma_{E}}{2\pi^{2}} + o(1), \\
& \mathrm{Cov}[ N_{\mathrm{Ai}}(r_{\mathrm{Ai}}y_{1}),N_{\mathrm{Ai}}(r_{\mathrm{Ai}}y_{2}) ] = \Sigma_{\mathrm{Ai}}(y_{1},y_{2}) + o(1).
\end{align*}
Using the expansions given in the proof of Corollary \ref{coro:conv to Airy}, we establish the following asymptotics (the result follows in the same way as Corollary \ref{coro: s1>0 consequence of thm}).

\begin{corollary}[Matching with the Airy asymptotics]
Let $0<y_{1}<y_{2}$ and $a \in (0,+\infty)$ be fixed. For $r>0$, define $\alpha = a \sqrt{r}$, and let $x_{1},x_{2}$ be such that
\begin{align*}
x_{j} = \frac{\alpha^{2}}{r}+2^{\frac{2}{3}} \frac{\alpha^{\frac{4}{3}}r_{\mathrm{Ai}}}{r}y_{j} = a^{2}+2^{\frac{2}{3}} \frac{a^{\frac{4}{3}}r_{\mathrm{Ai}}}{r^{\frac{1}{3}}}y_{j}, \qquad j=1,2.
\end{align*}
As $r_{\mathrm{Ai}} \to +\infty$ and $r \to + \infty$ such that \eqref{scaling inside coro Airy 2} holds, we have
\begin{align*}
& \mathbb{E}[N_{\alpha}(rx_{1})] = \mu_{\mathrm{Ai}}(r_{\mathrm{Ai}}y_{1})\bigg( 1+\bigO\Big(\frac{r_{\mathrm{Ai}}}{r^{\frac{1}{3}}}\Big) \bigg) + \bigO \bigg( \frac{r^{\frac{5}{6}}\log r}{r_{\mathrm{Ai}}^{4}} \bigg), \\
& \mathrm{Var}[N_{\alpha}(rx_{1})] = \sigma_{\mathrm{Ai}}^{2}(r_{\mathrm{Ai}}y_{1}) + \frac{1 + \gamma_{\mathrm{E}}}{2 \pi^{2}} + \bigO \bigg(\frac{r_{\mathrm{Ai}}}{r^{\frac{1}{3}}} +  \frac{r^{\frac{5}{6}}\log^{2} r}{r_{\mathrm{Ai}}^{4}} \bigg), \nonumber \\
& \mathrm{Cov}[ N_{\alpha}(rx_{1}),N_{\alpha}(rx_{2}) ] = \Sigma_{\mathrm{Ai}}(y_{1}, y_{2}) + \bigO \bigg(  \frac{r^{\frac{5}{6}}\log^{2} r}{r_{\mathrm{Ai}}^{4}} \bigg). \nonumber
\end{align*}
\end{corollary}

\subsection{Central limit theorems}
We first obtain a CLT for the counting function of the large $\alpha$ Bessel process.
\begin{corollary}\label{coro: CLTs}
Let $\alpha = a \sqrt{r}$, $a>0$, $r>0$, and let $m \in \mathbb{N}_{>0}$ and $\vec{x} \in \mathbb{R}_{\mathrm{ord}}^{+,m}$ be such that $0<a<x_{1}< \cdots<x_{m}$. Consider the random variables $\mathcal{N}_{j}^{(r)}$ defined by
\begin{align*}
\mathcal{N}_{j}^{(r)} = \frac{N_{\alpha}(rx_{j})-\mu_{\alpha}(rx_{j})}{\sqrt{\sigma_{\alpha}^{2}(rx_{j})}}, \qquad j=1,\ldots,m.
\end{align*}
\begin{enumerate}
\item As $r \to +\infty$ and $\alpha \to + \infty$, we have
\begin{align}\label{convergence in distribution 1}
\big( \mathcal{N}_{1}^{(r)},\mathcal{N}_{2}^{(r)},\ldots,\mathcal{N}_{m}^{(r)}\big) \quad \overset{d}{\longrightarrow} \quad \mathcal{N}(\vec{0},I_{m}),
\end{align}
uniformly for $(a,x_{1},\ldots,x_{m})$ in compact subsets of $\mathbb{R}_{\mathrm{ord}}^{+,m+1}$, where ``$\overset{d}{\longrightarrow}$" means convergence in distribution, $I_{m}$ is the $m \times m$ identity matrix, and $\mathcal{N}(\vec{0},I_{m})$ is a multivariate normal random variable of mean $\vec{0}=(0,\ldots,0)$ and covariance matrix $I_{m}$.
\item The convergence in distribution \eqref{convergence in distribution 1} still holds in the regime where $r \to +\infty$ and $\alpha=a\sqrt{r} \to + \infty$ such that $a \to 0$ and $\vec{x}$ lies in a compact subset of $\mathbb{R}_{\mathrm{ord}}^{+,m}$.
\item The convergence in distribution \eqref{convergence in distribution 1} still holds in the regime where $r \to +\infty$, $\alpha=a\sqrt{r} \to + \infty$, $a \in (0,+\infty)$ is fixed and $\vec{x} \in \mathbb{R}_{\mathrm{ord}}^{+,m}$ satisfies $a^{2} < x_{1} < \cdots < x_{m}$ and \eqref{Airy regime in the main first thm}.
\end{enumerate}

\end{corollary}
\begin{proof}
Recall that the asymptotics of part \ref{item 1 in thm} of Theorem \ref{thm:s1 neq 0} are valid uniformly for $u_{1},\ldots,u_{m}$ in compact subsets of $\mathbb{R}$, and uniformly for $\vec{x}$ in compact subsets of $\mathbb{R}_{\mathrm{ord}}^{+,m}$. Therefore, it follows from \eqref{F asymptotics thm s1 neq 0} with $n=1$ and $u_{j}=t_{j}/\sqrt{\sigma_{\alpha}^{2}(rx_{j})}$ that
\begin{align*}
\mathbb{E}_{\alpha} \Bigg[ \exp\bigg(\sum_{j=1}^{m} t_{j}\mathcal{N}_{j}^{(r)} \bigg) \Bigg] = \exp \Bigg( \sum_{j=1}^{m} \frac{t_{j}^{2}}{2} + o(1) \Bigg), \qquad \mbox{as } r \to + \infty, \; \alpha \to + \infty,
\end{align*}
uniformly for $\vec{x}$ in compact subsets of $\mathbb{R}_{\mathrm{ord}}^{+,m}$, which implies the convergence in distribution \eqref{convergence in distribution 1} stated in part $1$ of the claim. The regimes considered in parts $2$ and $3$ of the claim follow similarly from parts \ref{item 2 in thm} and \ref{item 3 in thm} of Theorem \ref{thm:s1 neq 0}.
\end{proof}
The Bessel process is locally finite, has almost surely all points distinct and possesses almost surely a smallest point. Let $0<\xi_{\alpha,1}<\xi_{\alpha,2}< \cdots$ be the points in the Bessel point process. Our next result is a CLT for the fluctuations of the points around their classical locations. We will use the notation $[x] = \lfloor x+\frac{1}{2}\rfloor$, i.e. $[x]$ is the closest integer to $x$. Also, since $\mu_{\alpha}:[\alpha^{2},+\infty) \to \mathbb{R}$, defined by \eqref{def of mu}, is strictly increasing, it possesses an inverse which we denote by $\mu_{\alpha}^{-1}$.
\begin{corollary}\label{coro: CLT 2}
Let $\alpha = a \sqrt{r}$, $a>0$, $r>0$, let $m \in \mathbb{N}_{>0}$ and $\vec{x} \in \mathbb{R}_{\mathrm{ord}}^{+,m}$ be such that $0<a<x_{1}< \cdots<x_{m}$, and let $k_{j}=[\mu_{\alpha}(rx_{j})]$, $j=1,\ldots,m$. Consider the random variables $Y_{j}^{(r)}$ defined by
\begin{align*}
Y_{j}^{(r)} = \frac{\mu_{\alpha}(\xi_{\alpha,k_{j}})-k_{j}}{\sqrt{\sigma_{\alpha}^{2}\circ \mu_{\alpha}^{-1}(k_{j})}}, \qquad j=1,\ldots,m.
\end{align*}
\begin{enumerate}
\item As $r \to +\infty$, $\alpha \to + \infty$, such that $(a,x_{1},\ldots,x_{m})$ lies in a compact subset of $\mathbb{R}_{\mathrm{ord}}^{+,m+1}$, we have
\begin{align}\label{convergence in distribution 2}
\big( Y_{1}^{(r)},Y_{2}^{(r)},\ldots,Y_{m}^{(r)}\big) \quad \overset{d}{\longrightarrow} \quad \mathcal{N}(\vec{0},I_{m}).
\end{align}
\item The convergence in distribution \eqref{convergence in distribution 2} still holds in the regime where $r \to +\infty$ and $\alpha=a\sqrt{r} \to + \infty$ such that $a \to 0$ and $\vec{x}$ lies in a compact subset of $\mathbb{R}_{\mathrm{ord}}^{+,m}$.
\item The convergence in distribution \eqref{convergence in distribution 2} still holds in the regime where $r \to +\infty$, $\alpha=a\sqrt{r} \to + \infty$, $a \in (0,+\infty)$ is fixed and $\vec{x}$ satisfies \eqref{Airy regime in the main first thm}.
\end{enumerate}
\end{corollary}
\begin{proof}
The proof is inspired by (but different from) the proof of \cite[Theorem 1.2]{Gustavsson}. Given $y_{1},\ldots,y_{m} \in \mathbb{R}$, we have
\begin{align}
& \mathbb{P}\big[ Y_{j}^{(r)}\leq y_{j} \mbox{ for all } j=1, \ldots,m \big] = \mathbb{P}\Big[\xi_{\alpha,k_{j}} \leq \mu_{\alpha}^{-1}\Big(k_{j} + y_{j}\sqrt{\sigma_{\alpha}^{2}\circ \mu_{\alpha}^{-1}(k_{j})}\Big) \mbox{ for all } j=1, \ldots,m \Big], \nonumber \\
& = \mathbb{P}\Big[N_{\alpha}\Big(\mu_{\alpha}^{-1}\Big(k_{j} + y_{j}\sqrt{\sigma_{\alpha}^{2}\circ \mu_{\alpha}^{-1}(k_{j})}\Big)\Big) \geq k_{j} \mbox{ for all } j=1, \ldots,m \Big]. \label{prob1}
\end{align}
Let us define 
\begin{align*}
\tilde{x}_{j} := \frac{1}{r} \, \mu_{\alpha}^{-1}\Big(k_{j} + y_{j}\sqrt{\sigma_{\alpha}^{2}\circ \mu_{\alpha}^{-1}(k_{j})}\Big),  \qquad j=1,\ldots,m.
\end{align*}
As $r \to +\infty$, $\alpha \to + \infty$ such that $(a,x_{1},\ldots,x_{m})$ lies in a compact subset of $\mathbb{R}_{\mathrm{ord}}^{+,m+1}$, we verify from \eqref{def of mu} and \eqref{def of sigma} that 
\begin{align*}
k_{j} = [\mu_{\alpha}(rx_{j})] = \bigO(\sqrt{r}), \qquad \tilde{x}_{j}=x_{j}\Big(1+\bigO\Big(\sqrt{\tfrac{\log r}{r}}\Big)\Big), 
\end{align*}
and in particular $(a,\tilde{x}_{1},\ldots,\tilde{x}_{m})$ lies also in a compact subset of $\mathbb{R}_{\mathrm{ord}}^{+,m+1}$ for all sufficiently large $r$. 
Now, we rewrite \eqref{prob1} as
\begin{align}
& \mathbb{P}\big[ Y_{j}^{(r)} \leq y_{j} \mbox{ for all } j=1, \ldots,m \big] \nonumber \\
& = \mathbb{P}\bigg[ \frac{N_{\alpha}(r\tilde{x}_{j})-\mu_{\alpha}(r\tilde{x}_{j})}{\sqrt{\sigma_{\alpha}^{2} (r\tilde{x}_{j})}} \geq \frac{k_{j}-\mu_{\alpha}(r\tilde{x}_{j})}{\sqrt{\sigma_{\alpha}^{2} (r\tilde{x}_{j})}} \mbox{ for all } j=1, \ldots,m \bigg] \nonumber \\
& = \mathbb{P}\bigg[ \frac{N_{\alpha}(r\tilde{x}_{j})-\mu_{\alpha}(r\tilde{x}_{j})}{\sqrt{\sigma_{\alpha}^{2} (r\tilde{x}_{j})}} \geq -y_{j}\frac{\sqrt{\sigma_{\alpha}^{2} \circ \mu_{\alpha}^{-1}(k_{j})}}{\sqrt{\sigma_{\alpha}^{2} (r\tilde{x}_{j})}} \mbox{ for all } j=1, \ldots,m \bigg] \nonumber \\
& = \mathbb{P}\bigg[ \frac{\mu_{\alpha}(r\tilde{x}_{j})-N_{\alpha}(r\tilde{x}_{j})}{\sqrt{\sigma_{\alpha}^{2}(r\tilde{x}_{j})}} \leq y_{j}(1+o(1)) \mbox{ for all } j=1, \ldots,m \bigg], \label{Prob Y leq yj}
\end{align}
where $o(1)$ in the last equality is understood as $r \to + \infty$, $\alpha \to + \infty$ such that $(a,\tilde{x}_{1},\ldots,\tilde{x}_{m})$ lies in a compact subset of $\mathbb{R}_{\mathrm{ord}}^{+,m+1}$. 
Part $1$ of the claim now follows from \eqref{Prob Y leq yj} and part 1 of Corollary \ref{coro: CLTs}, because if $\mathcal{N}$ is a multivariate normal random variable of mean $\vec{0}$ and covariance matrix $I_{m}$, then so is $-\mathcal{N}$. Part $2$ follows similarly from \eqref{Prob Y leq yj} and part $2$ of Corollary \ref{coro: CLTs}. For part $3$ of the claim, we note that
\begin{align*}
k_{j} =  \bigO(\sqrt{r}|x_{m}-a^{2}|^{3/2}), \qquad \tilde{x}_{j}=x_{j}\Big(1+\bigO\Big(\tfrac{\sqrt{\log (r|x_{m}-a^{2}|^{3/2})}}{\sqrt{r}|x_{m}-a^{2}|^{3/2}}\Big)\Big).
\end{align*}
Since $\tfrac{\sqrt{\log (r|x_{m}-a^{2}|^{3/2})}}{\sqrt{r}|x_{m}-a^{2}|^{3/2}} = \bigO\big( \frac{\log r}{|\tilde{x}_{m}-a^{2}|^{4}\sqrt{r}} \big)$ as $r \to + \infty$, $\alpha \to + \infty$ with $a \in (0,+\infty)$ fixed and $\vec{x}$ satisfying \eqref{Airy regime in the main first thm}, the claim follows from \eqref{Prob Y leq yj} and part $3$ of Corollary \ref{coro: CLTs}.
\end{proof}

\subsection{Rigidity}\label{subsection:rigidity}
There exist several notions of \textit{rigidity} in the literature. For example, the Bessel process is said to be \textit{rigid} in \cite{BufetovRigidityAiryBessel, BufetovConditional, MolagStevens} because $N_{\alpha}(x)$ is almost surely determined by the points on $(x,+\infty)$.
In this work, the \textit{rigidity} of a point process refers to the study of the maximal deviation of the points with respect to their classical locations. This notion of rigidity has been widely studied in recent years, see e.g. \cite{ErdosYauYin,ArguinBeliusBourgade} for important early works, \cite{HolcombPaquette} for the sine process, \cite{ChCl4} for the Airy and Bessel point processes, and \cite{CharlierPearcey} for the Pearcey process. A global rigidity upper bound for the Bessel process with $\alpha$ bounded has been established in \cite{ChCl4}. In this work, we contribute in this direction by providing a global rigidity upper bound for the large $\alpha$ Bessel process. Our techniques are inspired by previous works and rely on the first exponential moment asymptotics, which are given by \eqref{F asymptotics thm s1 neq 0} with $m=1$. There already exist several general rigidity theorems that are available in the literature, see \cite{CFLW, ChCl4}, but none of them can be directly applied to our case. The reason is that here the Bessel process varies as $r$ increases (recall that $\alpha \to + \infty$ as $r \to + \infty$). 

\medskip In fact, we will go slightly further and prove a general rigidity result which can be applied to varying point processes that possess a smallest point almost surely and for which the first exponential moment asymptotics are known. This result, which is stated in Theorem \ref{thm:rigidity} below, generalizes \cite[Theorem 1.2]{ChCl4} and allows us to obtain a global rigidity upper bound for the large $\alpha$ Bessel process. It can also be applied to other varying point processes, such as the conditional Airy and Bessel point processes considered in \cite[Theorem 1.2]{ChCl3} and \cite[Theorem 1.4]{CharlierBessel}, but we do not pursue this matter here.


\medskip Assume that we have a family of point processes $\{X_{r}\}_{r \geq 0}$, each of them with an almost surely smallest point. Let $\mathrm{N}_{r}(x)$ be the random variable that counts the number of points of $X_{r}$ that are $\leq x$. Our general rigidity result will be valid under the following assumptions on $\{X_{r}\}_{r \geq 0}$. 

\begin{assumptions}\label{assumptions}
There exist constants $\mathfrak{a} >0$, $M > \sqrt{2/\mathfrak{a}}$, $\mathrm{C}>0$, $r_0>0$, $\{\eta_{r,2},\eta_{r,1},\delta_{r}\}_{r \geq r_{0}}$ with $0 < \eta_{r,1} < \eta_{r,2} < + \infty$, $\delta_{r} \in (0,\frac{\eta_{r,2}-\eta_{r,1}}{5})$ and functions $\{\upmu_{r},\upsigma_{r}:[\eta_{r,1} r,\eta_{r,2}r]\to [0,+\infty)\}_{r \geq r_{0}}$ such that the following statements hold:
\begin{enumerate}[label=(\arabic*)]
\item We have
\begin{equation}\label{expmomentbound}
\mathbb{E} \big[e^{\gamma \mathrm{N}_{r}(xr)}\big]\leq \mathrm{C} \, e^{\gamma \upmu_{r}(xr)+\frac{\gamma^{2}}{2}\upsigma_{r}^2(xr)},
\end{equation}
for all $\gamma\in[-M,M]$, all $x \in (\eta_{r,1}+\frac{\delta_{r}}{2},\eta_{r,2}-\frac{\delta_{r}}{2})$, and  all $r>r_0$.
\item For each $r>r_{0}$, the functions $x \mapsto \upmu_{r}(x)$ and $x \mapsto \upsigma_{r}(x)$ are strictly increasing and differentiable, and satisfy 
\begin{align*}
\lim_{r\to + \infty} \upmu_{r}((\eta_{r,1}+\delta_{r}) r) = + \infty \qquad \mbox{ and } \qquad \lim_{r\to + \infty} \upsigma_{r}((\eta_{r,1}+\delta_{r}) r) = + \infty.
\end{align*}
Moreover, $x\mapsto x\upmu_{r}'(x)$ is non-decreasing and
\begin{align}\label{some limits are + inf in Assumptions}
\lim_{r\to+\infty} \inf_{x,y \in (\eta_{r,1}+\frac{\delta_{r}}{2},\eta_{r,2}-\frac{\delta_{r}}{2})} \frac{\delta_{r}r\upmu_{r}'(x r)}{\upsigma_{r}^2(y r)}=+\infty.
\end{align}
\item For each $r>r_{0}$, the function $\upsigma_{r}^2\circ\upmu_{r}^{-1}:[\upmu_{r}(\eta_{r,1}r),\upmu_{r}(\eta_{r,2}r)]\to [0,+\infty)$ is strictly concave, strictly increasing, and
\begin{align}\label{afrak def}
& (\upsigma_{r}^2\circ\upmu_{r}^{-1})(k)  = \mathfrak{a} \log k +\bigO(1) \qquad \mbox{as} \quad r\to +\infty,
\end{align}
uniformly for $k \in [\upmu_{r}((\eta_{r,1}+\frac{\delta_{r}}{2})r),\upmu_{r}((\eta_{r,2}-\frac{\delta_{r}}{2})r)]$. 
\item The quantities $\eta_{r,1}+\delta_{r}$ and $\eta_{r,2}-\eta_{r,1}-4\delta_{r}$ remain bounded away from $0$ as $r \to + \infty$, and $\eta_{r,2}$ remains bounded away from $+\infty$ as $r \to + \infty$.
\end{enumerate}
\end{assumptions}

\begin{theorem}[Rigidity]\label{thm:rigidity}
Suppose that $\{X_{r}\}_{r \geq 0}$ is a family of locally finite point processes on the real line with almost surely a smallest point and such that Assumptions \ref{assumptions} hold. Let $\mathfrak{a}, \eta_{r,1},\eta_{r,2},\delta_{r}$ be the constants appearing in Assumptions \ref{assumptions}, and let $\upxi_{r,k}$ denote the $k$-th smallest point of $X_{r}$, $k \geq 1$, $r \geq 0$. Let $\lambda >1$. There exist constants $c>0$ and $r_{0}>0$ such that for any small enough $\epsilon>0$ and for all $r \geq r_{0}$, 
\begin{align}\label{estimate of main thm}
\mathbb P\left(\max_{k \in I_{r}\cap \mathbb{N}_{\geq 0}} \frac{|\upmu_{r}(\upxi_{r,k}) - k|}{(\upsigma_{r}^2\circ\upmu_{r}^{-1})(k)} \leq  \sqrt{\frac{2}{\mathfrak{a}} (1+ \epsilon )} \right)\geq 1-\frac{c\, \upmu_{r}((\eta_{r,1}+\delta_{r}) r)^{-\frac{\epsilon}{\lambda}}}{\epsilon},
\end{align}
where $I_{r} = (\upmu_{r}((\eta_{r,1}+2\delta_{r}) r),\upmu_{r}((\eta_{r,2}-2\delta_{r}) r))$.
In particular, for any $\epsilon>0$,
\begin{align}\label{upper bound sharp general}
\lim_{r \to + \infty}\mathbb P\left(\max_{k \in I_{r}\cap \mathbb{N}_{\geq 0}} \frac{|\upmu_{r}(\upxi_{r,k}) - k|}{(\upsigma_{r}^2\circ\upmu_{r}^{-1})(k)} \leq  \sqrt{\frac{2}{\mathfrak{a}}}+\epsilon \right)=1.
\end{align}
\end{theorem}
\begin{remark}
We expect \eqref{estimate of main thm} to hold also for $\lambda=1$, but this sharper case requires a more delicate analysis and we do not pursue this here.
\end{remark}
\begin{remark}\label{remark:big lebesgue measure}
Assumptions \ref{assumptions} imply that the Lebesgue measure of $I_{r}$ tends to $+ \infty$ as $r \to + \infty$. Indeed, using parts (2) and (4) of Assumptions \ref{assumptions}, one has
\begin{align*}
|I_{r}| & \geq (\eta_{r,2}-\eta_{r,1}-4\delta_{r}) \inf_{} r\upmu_{r}'(x r) \geq  \tfrac{(\eta_{r,2}-\eta_{r,1}-4\delta_{r})(\eta_{r,1}+2\delta_{r})}{\eta_{r,2}-2\delta_{r}} r\upmu_{r}'((\eta_{r,1}+2\delta_{r})r)  \geq c_{1}r\upmu_{r}'((\eta_{r,1}+2\delta_{r})r),
\end{align*}
where the infimum is taken over $x \in (\eta_{r,1}+2\delta_{r},\eta_{r,2}-2\delta_{r})$ and $c_{1}>0$ is independent of $r$. It follows from part (2) of Assumptions \ref{assumptions} that the right-hand side tends to $+\infty$ as $r \to + \infty$. This means that the maximum in \eqref{upper bound sharp general} is taken over a large number of values of $k$ as $r$ gets large. 

The parameters $\eta_{r,2}$ and $\eta_{r,1}$ influence the size of $|I_{r}|$. Since the bound \eqref{upper bound sharp general} gives more information about $\{X_{r}\}_{r \geq 0}$ if $|I_{r}|$ is larger, it is important, in concrete examples, to choose $\eta_{r,2}$ as large as possible, and $\eta_{r,1}$ and $\delta_{r}$ as small as possible, such that \eqref{expmomentbound} holds.  
\end{remark}
In the next subsection, we use Theorem \ref{thm:rigidity} to obtain a global rigidity upper bound for the large $\alpha$ Bessel process.
\subsection{Rigidity of the large $\alpha$ Bessel point process}\label{section: rigidity of large alpha Bessel process}
Assume that $\alpha = \alpha(r) > 0$ is an increasing function of $r$ such that $\{a:=\alpha^{2}/r\}_{r\geq 1}$ is bounded and $\alpha \to + \infty$ as $r \to + \infty$. Let $X_{r}$ be the Bessel process associated with $\alpha$. It follows from parts $1$ and $2$ of Theorem \ref{thm:s1 neq 0} that the sequence $\{X_{r}\}_{r \geq 1}$ satisfies Assumptions \ref{assumptions} with $\eta_{r,1} = \alpha^{2}/r$, $\eta_{r,2} = K$, $r_{0}$ sufficiently large,
\begin{align*}
& \mathfrak{a} = \frac{1}{2\pi^{2}}, \qquad  \upmu_{r}(\xi) = \mu_{\alpha}(\xi), \qquad \upsigma_{r}^{2}(\xi) = \sigma_{\alpha}^{2}(\xi), \qquad \delta_{r} = \delta > 0, \quad \mbox{ and } \quad M> \sqrt{2/\mathfrak{a}},
\end{align*}
$\mathrm{C} = 2\sup_{u \in [-M,M]}G(1+\frac{u}{2\pi i})G(1-\frac{u}{2\pi i})$, where $\mu_{\alpha}$ and $\sigma_{\alpha}^{2}$ are defined in \eqref{def of mu}--\eqref{def of sigma}, where $M>\sqrt{2/\mathfrak{a}}$ and $K> \sup_{r \geq 1} \alpha^{2}/r$ can be chosen arbitrarily large but finite and where $\delta>0$ can be chosen arbitrarily small but fixed. Indeed, all statements in Assumptions \ref{assumptions} follow from straightforward computations except \eqref{afrak def}. To verify \eqref{afrak def}, suppose $k \in [\upmu_{r}((\eta_{r,1}+\frac{\delta}{2})r),\upmu_{r}((\eta_{r,2}-\frac{\delta}{2})r)]$. Since $\mu_{\alpha}$ is increasing, it follows that $\mu_\alpha^{-1}(k) \in [(\eta_{r,1}+\frac{\delta}{2})r,  (\eta_{r,2}-\frac{\delta}{2})r]$ and hence $\alpha^2 + cr \leq \mu_\alpha^{-1}(k) \leq Cr$ for some constants $c,C>0$. Since, by \eqref{def of sigma},
$$\sigma_{\alpha}^{2}(\xi) = \frac{\log \xi}{4\pi^2} + \bigO(1), \qquad \mbox{as }\alpha \to + \infty \mbox{ uniformly for } \xi \in [\alpha^2 + cr,Cr],$$
we conclude that
\begin{align}\label{upsigmaupmuinv}
(\upsigma_{r}^2\circ\upmu_{r}^{-1})(k) = (\sigma_{\alpha}^{2} \circ \mu_\alpha^{-1})(k)
= \frac{\log(\mu_\alpha^{-1}(k))}{4\pi^2} + \bigO(1)
= \frac{\log{r}}{4\pi^2} + \bigO(1), \qquad r \to +\infty.
\end{align}
Moreover, by \eqref{def of mu}, there exist constants $c_1, C_1 > 0$ such that $c_1\sqrt{\xi} \leq \mu_\alpha(\xi) \leq C_1 \sqrt{\xi}$ whenever $\xi \in [\alpha^2 + cr, Cr]$; thus $c_2 \sqrt{r} \leq k \leq C_2 \sqrt{r}$ for some constants $c_2, C_2 > 0$. This implies that $\log{r} = \log(k^2) + \bigO(1)$ and \eqref{afrak def} then follows from \eqref{upsigmaupmuinv}. Therefore, by specializing Theorem \ref{thm:rigidity} to the large $\alpha$ Bessel process, we obtain the following.

\begin{theorem}[Rigidity of the large $\alpha$ Bessel point process]\label{thm:rigidity Bessel} 
Assume that $\alpha$ is an increasing function of $r$ such that $\{a:=\alpha^{2}/r\}_{r\geq 1}$ is bounded and $\alpha \to + \infty$ as $r \to + \infty$, and let $\xi_{\alpha,1}<\xi_{\alpha,2}< \cdots$ be the points in the Bessel point process $X_{r}$. Fix $\lambda > 1$. For any $K > \sup_{r \geq 1} \alpha^{2}/r$, any $\delta \in (0,K)$ and any small enough $\epsilon>0$, there exist $r_{0}>0, c >0$ such that
\begin{align}\label{rigidity inside thm Bessel}
\mathbb P\left(\max_{k \in (\delta \sqrt{r},K\sqrt{r}) \cap \mathbb{N}_{\geq 0}} \frac{|\mu_{\alpha}(\xi_{\alpha,k}) - k|}{\log k} \leq  \frac{\sqrt{1+\epsilon}}{\pi} \right)\geq 1-\frac{c}{\epsilon}\sqrt{r}^{-\frac{\epsilon}{\lambda}},
\end{align}
for all $r \geq r_{0}$. In particular, for any $\epsilon > 0$,
\begin{align}\label{global rigidity upper bound for Bessel large alpha}
\lim_{r \to + \infty}\mathbb P\left(\max_{k \in (\delta \sqrt{r},K\sqrt{r})\cap \mathbb{N}_{\geq 0}} \frac{|\frac{1}{\pi}\int_{\alpha^{2}}^{\xi_{\alpha,k}} \frac{\sqrt{u-\alpha^{2}}}{2u}du - k|}{\log k} \leq  \frac{1}{\pi}+\epsilon \right)=1.
\end{align}
\end{theorem}
\begin{proof}
Applying Theorem \ref{thm:rigidity} to the large $\alpha$ Bessel process, we infer that there exist $c'$, $\epsilon_{0}'$ and $r_{0}'$ such that 
\begin{multline*}
\mathbb P\left(\max_{k \in I_{r} \cap \mathbb{N}_{\geq 0}} \frac{|\mu_{\alpha}(\xi_{\alpha,k}) - k|}{(\sigma_\alpha^2 \circ \mu_\alpha^{-1})(k)} \leq  2\pi \sqrt{1+\epsilon'} \right) \geq 1-\frac{c'\mu_{\alpha}(\alpha^{2}+\delta r)^{-\frac{\epsilon'}{\lambda}}}{\epsilon'} \\
= 1-\frac{c'}{\epsilon'}\bigg( \frac{\sqrt{r}}{\pi}\int_{a^{2}}^{a^{2}+\delta} \frac{\sqrt{u-a^{2}}}{2u}du  \bigg)^{-\frac{\epsilon'}{\lambda}} \geq 1-\frac{\hat{c}}{\epsilon'}\sqrt{r}^{-\frac{\epsilon'}{\lambda}},
\end{multline*}
for all $r \geq r_{0}'$ and $\epsilon' \in (0,\epsilon_{0}']$, where $\hat{c}:= c' \sup_{r \geq r_{0}', \epsilon' \in (0,\epsilon_{0}']} \big( \frac{1}{\pi}\int_{a^{2}}^{a^{2}+\delta} \frac{\sqrt{u-a^{2}}}{2u}du  \big)^{-\frac{\epsilon'}{\lambda}}< + \infty$ and
\begin{align*}
I_{r} = (\mu_{\alpha}(\alpha^{2}+2\delta r),\mu_{\alpha}(Kr-2\delta r)) =  \bigg( \frac{\sqrt{r}}{\pi}\int_{a^{2}}^{a^{2}+2\delta} \frac{\sqrt{u-a^{2}}}{2u}du \, , \, \frac{\sqrt{r}}{\pi}\int_{a^{2}}^{K-2\delta} \frac{\sqrt{u-a^{2}}}{2u}du \bigg).
\end{align*}
Note that 
\begin{align*}
I_{r} \supseteq (c_{\delta}\sqrt{r}, c_{K}\sqrt{r}), \quad \mbox{where } c_{\delta} := \sup_{r \geq r_{0}} \frac{1}{\pi}\int_{a^{2}}^{a^{2}+2\delta} \frac{\sqrt{u-a^{2}}}{2u}du, \; c_{K} := \inf_{r \geq r_{0}} \frac{1}{\pi}\int_{a^{2}}^{K-2\delta} \frac{\sqrt{u-a^{2}}}{2u}du.
\end{align*}
Let $\delta$ be sufficiently small and/or let $K$ be sufficiently large such that $c_{\delta}<c_{K}$. Then,
\begin{align*}
\mathbb P\left(\max_{k \in (c_{\delta}\sqrt{r}, c_{K}\sqrt{r}) \cap \mathbb{N}_{\geq 0}} \frac{|\mu_{\alpha}(\xi_{\alpha,k}) - k|}{(\sigma_\alpha^2 \circ \mu_\alpha^{-1})(k)} \leq  2\pi \sqrt{1+\epsilon'} \right) \geq 1-\frac{\hat{c}}{\epsilon'}\sqrt{r}^{-\frac{\epsilon'}{\lambda}},
\end{align*}
for all $r \geq r_{0}'$ and $\epsilon' \in (0,\epsilon_{0}']$. Using that (cf. \eqref{upsigmaupmuinv})
\begin{align*}
\max_{k \in (c_{\delta}\sqrt{r}, c_{K}\sqrt{r}) \cap \mathbb{N}_{\geq 0}} \bigg| (\sigma_\alpha^2 \circ \mu_\alpha^{-1})(k)\Big( \frac{\log k}{2\pi^2} \Big)^{-1} -1 \bigg| = \bigO\Big( \frac{1}{\log r}\Big) \qquad \mbox{as } r \to + \infty,
\end{align*}
we conclude that there exists a $c_0 >0$ such that
\begin{align*}
\mathbb P\left(\max_{k \in (c_{\delta}\sqrt{r}, c_{K}\sqrt{r}) \cap \mathbb{N}_{\geq 0}} \frac{|\mu_{\alpha}(\xi_{\alpha,k}) - k|}{\log k} \leq  \frac{\sqrt{1+\epsilon'}}{\pi(1 - \frac{c_0}{\log{r}})} \right)\geq 1-\frac{\hat{c}}{\epsilon'}\sqrt{r}^{-\frac{\epsilon'}{\lambda}},
\end{align*}
for all sufficiently large $r$ and all $\epsilon' \in (0,\epsilon_{0}']$. Hence for each $\epsilon'\in (0,\epsilon_{0}']$,
\begin{align*}
\mathbb P\left(\max_{k \in (c_{\delta}\sqrt{r}, c_{K}\sqrt{r}) \cap \mathbb{N}_{\geq 0}} \frac{|\mu_{\alpha}(\xi_{\alpha,k}) - k|}{\log k} \leq  \frac{\sqrt{1+\epsilon}}{\pi} \right)\geq 1-\frac{2\hat{c}}{\epsilon}\sqrt{r}^{-\frac{\epsilon'}{\lambda}}=1-\frac{2\,\hat{c}\,e^{\frac{c_{1}}{2\lambda}}}{\epsilon}\sqrt{r}^{-\frac{\epsilon}{\lambda}},
\end{align*}
for all sufficiently large $r$, where $\epsilon = \epsilon'+\frac{c_{1}}{\log r}$ and $c_{1}$ is sufficiently large but fixed. Since $c_{\delta}$ can be made arbitrarily small and $c_{K}$ can be made arbitrarily large by choosing $\delta$ sufficiently small and $K$ sufficiently large, the claim follows. 
\end{proof}
If we let $\alpha \to + \infty$ at a slow speed so that $\alpha^{2}/\sqrt{r} \to 0$ as $r \to + \infty$ in Theorem \ref{thm:rigidity Bessel}, we obtain the following corollary, which is consistent with the known rigidity result \cite[Theorem 1.6]{ChCl4} for the Bessel process with $\alpha$ fixed.
\begin{corollary}[Matching with the rigidity of the Bessel process with $\alpha$ bounded] \label{coro: rigi bessel alpha bounded} 
Assume that $\alpha$ is an increasing function of $r$ such that  $\alpha^{2}/\sqrt{r} \to 0$ and $\alpha \to + \infty$ as $r \to + \infty$, and let $\xi_{\alpha,1}<\xi_{\alpha,2}< \cdots$ be the points in the Bessel point process $X_{r}$. For any $K > \delta>0$ and any $\epsilon >0$, we have
\begin{align*}
\lim_{r \to + \infty}\mathbb P\left( \max_{k \in I_{r}\cap \mathbb{N}_{\geq 0}} \frac{|\frac{\sqrt{\xi_{\alpha,k}}}{\pi} - \frac{\alpha}{2} - k|}{\log k} \leq  \frac{1}{\pi}+\epsilon \right)=1, \quad \mbox{ where } \quad \mathcal{I}_{r} = \big( \delta \sqrt{r} \, , \, K \sqrt{r} \big).
\end{align*}
\end{corollary}
\begin{proof}
The claim follows easily from Theorem \ref{thm:rigidity Bessel} and the first expansion in \eqref{asymptotics as a is small mean var cov}. 
\end{proof}

\paragraph{Outline.} The proof of Theorem \ref{thm:s1 neq 0} is divided into several steps that are carried out in Sections \ref{section: model rh problem}-\ref{Section: integration s1 >0}. In Section \ref{section: model rh problem}, we present a differential identity from \cite{CharlierBessel} that relates $E_{\alpha}$ to a model Riemann-Hilbert (RH) problem whose solution is denoted by $\Phi$. In Section \ref{Section: Steepest descent with s1>0}, we obtain asymptotics for $\Phi$ in the three regimes of Theorem \ref{thm:s1 neq 0} via the Deift--Zhou \cite{DeiftZhou} steepest descent method. As mentioned in the introduction, the regime $\alpha/\sqrt{r} \to 0$ (part \ref{item 2 in thm} of Theorem \ref{thm:s1 neq 0}) requires a different analysis from the regimes \ref{item 1 in thm} and \ref{item 3 in thm} where $\alpha/\sqrt{r}$ remains bounded away from $0$. In Section \ref{Section: integration s1 >0}, we substitute the asymptotics of $\Phi$ into the differential identity to obtain large $r$ asymptotics for $E_{\alpha}(r \vec{x},\vec{u})$. The proof of Theorem \ref{thm:rigidity} is independent from the other proofs and is given in Section \ref{Section: rigidity}. The appendix collects three model RH problems that are important for the construction of local parametrices in the steepest descent analysis. Of particular interest is that we describe and employ the Bessel model problem in the regime where the order $\alpha$ of the Bessel functions tends to infinity. This appears to be the first time that the Bessel model problem is used in this way.

\section*{Acknowledgements}
We are grateful to Gr\'{e}gory Schehr for suggesting this problem to us, for interesting discussions and for pointing out a mistake in an early draft of this manuscript. CC acknowledges support from  the European Research Council, Grant Agreement No. 682537.
JL acknowledges support from the European Research Council, Grant Agreement No. 682537, the Swedish Research Council, Grant No. 2015-05430, and the Ruth and Nils-Erik Stenb\"ack Foundation.

\section{Background from \cite{ChDoe,CharlierBessel}}\label{section: model rh problem}
We will prove Theorem \ref{thm:s1 neq 0} via a differential identity that was established in \cite{CharlierBessel}. This identity relates $E_{\alpha}$ with the solution $\Phi$ to a model RH problem. We first recall the properties of $\Phi$ in Subsection \ref{subsection: model RH problem Phi}, and then present the differential identity in Subsection \ref{subsection: diff identity}.

\subsection{Model RH problem $\Phi$ from \cite{ChDoe}}\label{subsection: model RH problem Phi}
The model RH problem for $\Phi$ depends on parameters $\alpha$, $\vec{x}=(x_{1},\ldots,x_{m})$, $\vec{s}=(s_{1},\ldots,s_{m})$ satisfying
\begin{equation}\label{param for phi}
\alpha > -1, \qquad 0<x_{1}< \cdots <x_{m}<+\infty, \qquad s_{1},\dots,s_{m} \in [0,+\infty),
\end{equation}
and is as follows.

\begin{figure}[t]
    \begin{center}
\begin{tikzpicture}
\draw[fill] (0,0) circle (0.05);
\draw (0,0) -- (8,0);
\draw (0,0) -- (120:3);
\draw (0,0) -- (-120:3);
\draw (0,0) -- (-3,0);

\draw[fill] (3,0) circle (0.05);
\draw[fill] (6.5,0) circle (0.05);
\draw[fill] (8,0) circle (0.05);

\node at (0.15,-0.3) {$-x_{m}$};
\node at (3,-0.3) {$-x_{2}$};
\node at (6.5,-0.3) {$-x_{1}$};
\node at (8.4,-0.3) {$0=x_{0}$};
\node at (-3,-0.3) {$-\infty=-x_{m+1}$};

\draw[black,arrows={-Triangle[length=0.18cm,width=0.12cm]}]
(-120:1.5) --  ++(60:0.001);
\draw[black,arrows={-Triangle[length=0.18cm,width=0.12cm]}]
(120:1.3) --  ++(-60:0.001);
\draw[black,arrows={-Triangle[length=0.18cm,width=0.12cm]}]
(180:1.5) --  ++(0:0.001);

\draw[black,arrows={-Triangle[length=0.18cm,width=0.12cm]}]
(0:1.5) --  ++(0:0.001);
\draw[black,arrows={-Triangle[length=0.18cm,width=0.12cm]}]
(0:5) --  ++(0:0.001);
\draw[black,arrows={-Triangle[length=0.18cm,width=0.12cm]}]
(0:7.3) --  ++(0:0.001);

\end{tikzpicture}
    \caption{\label{figPhi}The contour $\Sigma_{\Phi}$ with $m=3$.}
\end{center}
\end{figure}
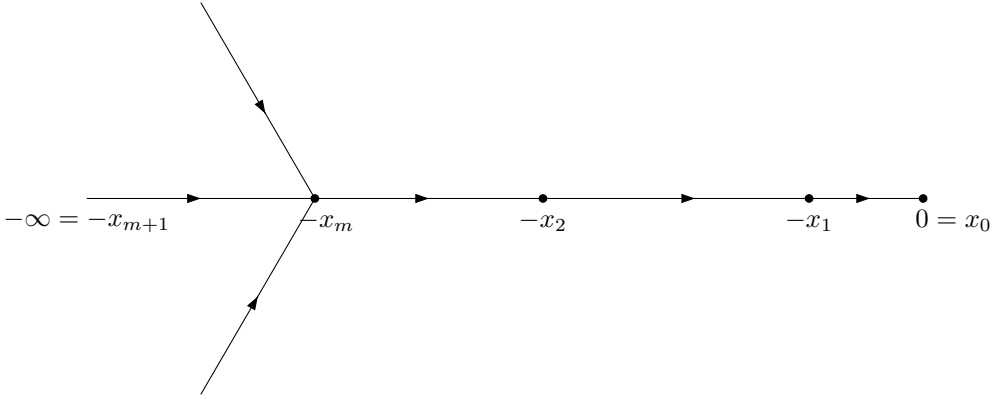

\vspace{-0.2cm}\subsubsection*{RH problem for $\Phi (\cdot) = \Phi(\cdot;\vec{x},\vec{s},\alpha)$}
\begin{itemize}
\item[(a)] $\Phi : \mathbb{C}\setminus \Sigma_{\Phi} \to \mathbb{C}^{2\times 2}$ is analytic, where the contour is given by
\begin{align*}
\Sigma_{\Phi} = (-\infty,0] \cup (-x_{m}+ [0,e^{\frac{2\pi i}{3}}\infty)) \cup (-x_{m}+ [0,e^{-\frac{2\pi i}{3}}\infty)),
\end{align*}
and is oriented as shown in Figure \ref{figPhi}.
\item[(b)] Let $\Phi_{+}$ (resp. $\Phi_{-}$) denote the left (resp. right) boundary values of $\Phi$. Here, ``left" and ``right" are to be understood with respect to the orientation of $\Sigma_{\Phi}$. $\Phi_{+}$ and $\Phi_{-}$ are continuous on $\Sigma_{\Phi}\setminus \{0,-x_{1},\dots,-x_{m}\}$ and are related by:
\begin{align}
& \Phi_{+}(z) = \Phi_{-}(z) \begin{pmatrix}
1 & 0 \\ e^{\pi i \alpha} & 1
\end{pmatrix}, & & z \in -x_{m}+ (0,e^{\frac{2\pi i}{3}}\infty), \label{jumps Phi upper} \\
& \Phi_{+}(z) = \Phi_{-}(z) \begin{pmatrix}
0 & 1 \\ -1 & 0
\end{pmatrix}, & & z \in (-\infty,-x_{m}), \\
& \Phi_{+}(z) = \Phi_{-}(z) \begin{pmatrix}
1 & 0 \\ e^{-\pi i \alpha} & 1
\end{pmatrix}, & & z \in -x_{m}+ (0,e^{-\frac{2\pi i}{3}}\infty), \label{jumps Phi lower} \\
& \Phi_{+}(z) = \Phi_{-}(z) \begin{pmatrix}
e^{\pi i \alpha} & s_{j} \\ 0 & e^{-\pi i \alpha}
\end{pmatrix}, & & z \in (-x_{j},-x_{j-1}), \label{last jumps of Phi}
\end{align}
where $j =1,\dots,m$.
\item[(c)] As $z \to \infty$, we have 
\begin{equation}\label{Phi inf}
\Phi(z) = \Big(I + \Phi_{1}z^{-1}+\bigO(z^{-2})\Big)z^{-\frac{\sigma_{3}}{4}}Ne^{\sqrt{z}\sigma_{3}},
\end{equation}
where the principal branch is chosen for each root, the matrix $\Phi_{1} = \Phi_{1}(\vec{x},\vec{s},\alpha)$ is independent of $z$ and traceless, and 
\begin{align*}
\sigma_{3} = \begin{pmatrix}
1 & 0 \\ 0 & -1
\end{pmatrix}, \qquad N = \frac{1}{\sqrt{2}}\begin{pmatrix}
1 & i \\ i & 1
\end{pmatrix}.
\end{align*}

As $z$ tends to $-x_{j}$, $j \in \{1,\dots,m\}$, we have
\begin{equation}\label{def of Gj}
\Phi(z) = G_{j}(z) \begin{pmatrix}
1 & \frac{s_{j+1}-s_{j}}{2\pi i} \log (z+x_{j}) \\ 0 & 1
\end{pmatrix} V_{j}(z) e^{\frac{\pi i \alpha}{2}\theta(z)\sigma_{3}}H_{-x_{m}}(z),
\end{equation}
where $G_{j}(z) = G_{j}(z;\vec{x},\vec{s},\alpha)$ is analytic in a neighborhood of $-x_{j}$ and satisfies $\det G_{j} \equiv 1$, and $\theta$, $V_{j}$, $H_{-x_{m}}$ are piecewise constant and given by
\begin{align}\label{def of theta}
& \theta(z) = \left\{  \begin{array}{l l}
+1, & \mbox{Im } z > 0, \\
-1, & \mbox{Im } z < 0,
\end{array} \right. & & V_{j}(z) = \left\{  \begin{array}{l l}
I, & \mbox{Im } z > 0, \\
\begin{pmatrix}
1 & -s_{j} \\ 0 & 1
\end{pmatrix}, & \mbox{Im } z < 0,
\end{array} \right.
\end{align}
and
\begin{equation}\label{def of H}
H_{-x_{m}}(z) = \left\{  \begin{array}{l l}

I, & \mbox{ for } -\frac{2\pi}{3}< \arg(z+x_{m})< \frac{2\pi}{3},\\

\begin{pmatrix}
1 & 0 \\
-e^{\pi i \alpha} & 1 \\
\end{pmatrix}, & \mbox{ for } \frac{2\pi}{3}< \arg(z+x_{m})< \pi, \\[0.3cm]

\begin{pmatrix}
1 & 0 \\
e^{-\pi i \alpha} & 1 \\
\end{pmatrix}, & \mbox{ for } -\pi< \arg(z+x_{m})< -\frac{2\pi}{3}.
\end{array} \right.
\end{equation} 

As $z$ tends to $0$, we have
\begin{equation}\label{def of G_0}
\Phi(z) = G_{0}(z)z^{\frac{\alpha}{2}\sigma_{3}} \begin{pmatrix}
1 & s_{1}h(z) \\ 0 & 1
\end{pmatrix},
\end{equation}
where $\det G_{0} \equiv 1$, $G_{0}(z) = G_{0}(z;\vec{x},\vec{s},\alpha)$ is analytic in a neighborhood of $0$, and 
\begin{equation}\label{def of h}
h(z) = \left\{ \begin{array}{l l}
\displaystyle \frac{1}{2 i \sin(\pi \alpha)}, & \alpha >-1, \; \alpha \notin \mathbb{N}_{\geq 0}, \\[0.35cm]
\displaystyle \frac{(-1)^{\alpha}}{2\pi i} \log z, & \alpha \in \mathbb{N}_{\geq 0}.
\end{array} \right.
\end{equation}
\end{itemize}
The solution $\Phi$ to the above RH problem exists and is unique, and satisfies $\det \Phi \equiv 1$, see \cite{ChDoe}.

\subsection{Differential identity from \cite{CharlierBessel}}\label{subsection: diff identity}
Recall that $E_{\alpha}(\vec{x},\vec{u})$ is the exponential moment \eqref{def of E alpha}, where $\vec{u}=(u_{1},\ldots,u_{m}) \in \mathbb{R}^{m}$. To relate $E_{\alpha}(\vec{x},\vec{u})$ with the solution $\Phi$ of the model RH problem of Subsection \ref{subsection: model RH problem Phi}, we define $\vec{s}=(s_{1},\ldots,s_{m})$ in terms of $\vec{u}$ as follows:
\begin{equation}\label{def tj thm s1 neq 0}
s_{j} = e^{u_{j}+u_{j+1}+\cdots+u_{m}}, \qquad j=1, \ldots, m.
\end{equation}
Since $u_{1},\ldots,u_{m} \in \mathbb{R}$, we have $s_{1},\ldots,s_{m} \in (0,+\infty)$, and in particular the $s_{j}$'s meet the requirement in \eqref{param for phi}. For convenience, we define
\begin{equation}\label{def of F alpha}
F_{\alpha}(\vec{x},\vec{s}) = E_{\alpha}(\vec{x},\vec{u}).
\end{equation}
Based on the facts that
\begin{itemize}
\item \vspace{-0.1cm} $F_{\alpha}(\vec{x},\vec{s})$ can be written as a Fredholm determinant with $m$ discontinuities and is an entire function of $s_{1},\ldots,s_{m}$ (see \cite[Theorem 2]{Soshnikov}), 
\item \vspace{-0.1cm} the kernel of this Fredholm determinant is \textit{integrable} in the sense of \cite{IIKS}, 
\item \vspace{-0.1cm} $\Phi$ admits a Lax pair (studied in \cite{ChDoe}),
\end{itemize} 
the following differential identity was obtained in \cite[eqs (2.34)--(2.37)]{CharlierBessel} for $k = 1, \dots, m$:
\vspace{-0.1cm}
\begin{equation}\label{DIFF identity final form general case}
\partial_{s_{k}} \log F_{\alpha}(r\vec{x},\vec{s})= K_{\infty} + \sum_{j=1}^{m}K_{-x_{j}} + K_{0}, \qquad \alpha \neq 0,
\end{equation}
where
\vspace{-0.2cm}\begin{align}
& K_{\infty} = -\frac{i}{2}\partial_{s_{k}} \Phi_{1,12}(r\vec{x},\vec{s},\alpha), \label{K inf} \\
& K_{-x_{j}} = \frac{s_{j+1}-s_{j}}{2\pi i} \Big( G_{j,11}\partial_{s_{k}} G_{j,21} - G_{j,21}\partial_{s_{k}} G_{j,11} \Big)(-rx_{j};r\vec{x},\vec{s},\alpha), \label{K -xj} \\
& K_{0} = \alpha \Big( G_{0,21} \partial_{s_{k}}G_{0,12} - G_{0,11} \partial_{s_{k}}G_{0,22} \Big)(0;r\vec{x},\vec{s},\alpha). \label{K 0}
\end{align}
We mention that the above differential identity is not valid for $\alpha=0$, but since our analysis deals with the case $\alpha \to + \infty$, this fact is not important for us.

\section{Large $r$ asymptotics for $\Phi$ with $\alpha \to + \infty$: regimes \ref{item 1 in thm} and \ref{item 3 in thm}}\label{Section: Steepest descent with s1>0}
In this section we perform an asymptotic analysis of $\Phi(rz;r\vec{x},\vec{s},\alpha)$ in the two regimes of the parameters that are relevant to prove parts \ref{item 1 in thm} and \ref{item 3 in thm} of Theorem \ref{thm:s1 neq 0}. In Subsections \ref{subsection: g function s1 neq 0}--\ref{subsection Small norm s1 neq 0}, we deal with the regime which is relevant for part $1$; in Subsection \ref{subsection: regime 3}, we indicate how to adapt this analysis to the regime in part \ref{item 3 in thm}.

\medskip Let $m \in \mathbb{N}_{>0}$, $a \in (0,+\infty)$, $\vec{u}=(u_1,\dots,u_m) \in \mathbb{R}^{m}$ and $\vec{x} = (x_{1},\dots,x_{m}) \in \mathbb{R}_{\mathrm{ord}}^{+,m}$, be such that
\begin{align*}
0 < x_{1} < \cdots < x_{n-1} < a^{2} < x_{n} < \cdots < x_{m},
\end{align*}
with $n \in \{1,\dots,m+1\}$, $x_{0}:=0$, $x_{m+1}:=+\infty$, and define $\vec{s}=(s_{1},\ldots,s_{m})$ in terms of $\vec{u}$ as in \eqref{def tj thm s1 neq 0}. For $r>0$, define $\alpha = a \sqrt{r}$. In this section, we first obtain asymptotics for $\Phi(rz;r\vec{x},\vec{s},\alpha)$ as $r \to + \infty$ uniformly for $(x_{1},\ldots,x_{n-1},a^{2},x_{n},\ldots,x_{m})$ in compact subsets of $\mathbb{R}_{\mathrm{ord}}^{+,m+1}$, uniformly for $u_{1},\ldots,u_{m}$ in compact subsets of $\mathbb{R}$, and uniformly for $z \in \mathbb{C}$. This analysis is based on the Deift--Zhou steepest descent method \cite{DeiftZhou} and is carried out in Subsections \ref{subsection: g function s1 neq 0}--\ref{subsection Small norm s1 neq 0}. 

\medskip In what follows, we assume without loss of generality that $a^{2}<x_{m}$, or in other words that $n \in \{1,2,\dots,m\}$. Indeed, if $a^{2}>x_{m}$, then it suffices to increase $m$ by $1$ and to define $x_{m}=a^{2}+1$ and $s_{m}=1$. 

\subsection{First transformation: $\Phi \to T$}\label{subsection: g function s1 neq 0}
As mentioned in the introduction, as $\alpha = a \sqrt{r}$ gets large, the Bessel process increasingly favors point configurations with fewer points near $0$. It turns out that a typical point configuration has its smallest point located near the soft edge $a^{2}$. Therefore, at least at a heuristic level, one expects the interval $(-\infty, -a^{2})$ to play a special role in the large $r$ analysis of $\Phi(rz;r\vec{x},\vec{s},a\sqrt{r})$.\footnote{The model RH problem for $\Phi$ is directly related to the reverse Bessel process for convenience. One needs to identify $(-\infty, -a^{2})$ in the reverse Bessel process with $(a^{2},+\infty)$ in the classical Bessel process.} The first transformation has multiple purposes: it normalizes the RH problem at $\infty$ and transforms the jumps that are not on $(-\infty,-a^{2})$ into exponentially decaying jumps. The $g$-function is the main ingredient of this transformation and is defined by
\vspace{-0.2cm}\begin{equation}\label{def of g s1 neq 0}
g(z)= f(z) + \theta(z)\frac{\pi i a}{2}, \qquad f(z) = \int_{-a^{2}}^{z} \frac{\sqrt{s+a^{2}}}{2s}ds,
\end{equation}
where the path does not intersect $(-\infty,-a^{2})\cup(0,+\infty)$, and $\theta$ is defined in \eqref{def of theta}.  We summarize the properties of the $g$-function in the following lemma.
\begin{lemma}\label{lemma: g function}The $g$-function is analytic in $\mathbb{C}\setminus (-\infty,0]$ and satisfies the following properties:
\begin{align}
& g_{+}(z)+g_{-}(z) = 0, & & \mbox{for } z \in (-\infty,-a^{2}), \label{g prop 1} \\
& g_{+}(z)-g_{-}(z) = \pi i a +f_{+}(z)-f_{-}(z), & & \mbox{for } z \in (-\infty,-a^{2}), \label{g prop 2} \\
& g_{+}(z)-g_{-}(z) = \pi i a, & & \mbox{for } z \in (-a^{2},0), \label{g prop 3} \\
& g(z) = \sqrt{z} - \frac{a^{2}}{2\sqrt{z}} + \bigO(z^{-3/2}), & & \mbox{as } z \to \infty, \label{g prop 5} \\
& g(z) = \frac{a}{2}\log z + g_{0}+\bigO(z), & & \mbox{as } z \to 0, \label{g prop 6} \\
& g(z) = \overline{g(\bar{z})}, & & \mbox{for } z \in \mathbb{C} \setminus (-\infty,0], \label{g prop 7}
\end{align}
where $g_{0} \in \mathbb{R}$. Furthermore, there exists $\epsilon > 0$ and an open set $\mathcal{V}\subset \mathbb{C}$ such that
\begin{align*}
\{ z \in \mathbb{C} : |z| \geq \tfrac{1}{\epsilon} \} \cup \{z \in \mathbb{C}: \re z < -a^{2} \mbox{ and } |\im z |\leq \epsilon\} \subset \mathcal{V},
\end{align*}
and such that
\begin{align}
& \re f(z) \geq 0, & & \mbox{for all } z \in \mathcal{V}, \label{re f} \\
& \re f(z) < 0, & &  \mbox{for all } z \in (-a^{2},0), \label{re f 2}
\end{align}
with equality in \eqref{re f} if and only if $z \in (-\infty,-a^{2})$.
\end{lemma}
\begin{proof}
The properties \eqref{g prop 1}--\eqref{g prop 7} and \eqref{re f 2} follow directly from \eqref{def of g s1 neq 0} and some straightforward computations. Let us prove \eqref{re f}. Since $\re f(z) / \re\sqrt{z}=1+\bigO(z^{-1})$ as $z \to \infty$, there exists $\epsilon_{1}>0$ such that $\re f(z) \geq 0$ for all $|z| \geq \frac{1}{\epsilon_{1}}$, with equality if and only if $z \in (-\infty,\frac{1}{\epsilon_{1}}]$. On the other hand, as $z \to -a^{2}$, we have $\re f(z) \sim c\, \re(z+a^{2})^{3/2}$ for a certain $c<0$. Hence, there exists $\epsilon_{2} > 0$ such that $\re f(z) \geq 0$ for all $z \in \{z:\re z \in [-a^{2}-\epsilon_{2},-a^{2}) \mbox{ and } |\im z| \leq \epsilon_{2}\}$ with equality if and only if $z \in [-a^{2}-\epsilon_{2},-a^{2})$. Also, we note from \eqref{def of g s1 neq 0} that $f(z) = \overline{f(\overline{z})}$ and that $\im f_{+}(x)$ is decreasing as $x \in (-\infty,-a^{2})$ increases. Therefore, by Cauchy-Riemann, there exists $\epsilon_{3}>0$ such that $\re f(z) >0$ for all $z$ satisfying $\re z \in [-\frac{1}{\epsilon_{1}},-a^{2}-\epsilon_{2}]$ and $|\im z| \in (0,\epsilon_{3}]$. Indeed, suppose such an $\epsilon_3$ does not exist. Then there is a sequence $\{z_n\}_1^\infty$ of zeros of $\re f$ such that $\re z_n \in [-\frac{1}{\epsilon_{1}},-a^{2}-\epsilon_{2}]$ and $\im z_n > 0$ for each $n$, and such that $\im z_1 > \im z_2 > \cdots$ and $\lim_{n\to +\infty}\im z_{n}=0$. For each $n$, there exists by the mean value theorem a point $\xi_n$ with the same real part as $z_n$ such that $0 < \im \xi_n < \im z_n$ and $\partial_y \re f(x+iy) = 0$ for $x+iy = \xi_n$. Moreover, since the interval $[-\frac{1}{\epsilon_{1}},-a^{2}-\epsilon_{2}]$ is compact, there is a subsequence $z_{n_k}$ such that $z_{n_k} \to x_*$ as $k \to +\infty$ where $x_* \in [-\frac{1}{\epsilon_{1}},-a^{2}-\epsilon_{2}]$. Since the restriction of $f$ to $[-\frac{1}{\epsilon_{1}},-a^{2}-\epsilon_{2}] \times [0,1]$ is $C^1$ and $\xi_{n_k} \to x_*$, we infer that $\partial_y \re f(x_*+iy)|_{y=0^+} = 0$, which is a contradiction.
This proves \eqref{re f} with $\epsilon = \min\{\epsilon_{1},\epsilon_{2},\epsilon_{3}\}$.
\end{proof}

The first transformation is defined by
\begin{equation}\label{def of T s1 neq 0}
T(z)= \begin{pmatrix}
1 & 0 \\
- \frac{i a^{2}}{2}\sqrt{r} & 1
\end{pmatrix} r^{\frac{\sigma_{3}}{4}}\Phi(rz;r\vec{x},\vec{s},\alpha)e^{-\sqrt{r}g(z)\sigma_{3}}.
\end{equation}
Using Lemma \ref{lemma: g function} together with the properties of $\Phi$ listed in Subsection \ref{subsection: model RH problem Phi}, it can be verified that $T$ satisfies the following RH problem.
\subsubsection*{RH problem for $T$}
\begin{itemize}
\item[(a)] $T : \mathbb{C}\setminus \Sigma_{\Phi} \to \mathbb{C}^{2\times 2}$ is analytic.
\item[(b)] The jumps for $T$ are given by
\begin{align}
& T_{+}(z) = T_{-}(z) \begin{pmatrix}
1 & 0 \\ e^{-2\sqrt{r}f(z)} & 1
\end{pmatrix}, & & z \in -x_{m}+ (0,e^{\pm \frac{2\pi i}{3}}\infty), \label{jumps for T 1} \\
& T_{+}(z) = T_{-}(z) \begin{pmatrix}
0 & 1 \\ -1 & 0
\end{pmatrix}, & & z \in (-\infty,-x_{m}), \\
& T_{+}(z) = T_{-}(z) \begin{pmatrix}
e^{-\sqrt{r}(f_{+}(z)-f_{-}(z))} & s_{j} \\ 0 & e^{\sqrt{r}(f_{+}(z)-f_{-}(z))}
\end{pmatrix}, & & z \in \mathcal{B}_{j}, \; j \in \{n,\ldots,m\}, \label{jumps for T to be factorized} \\
& T_{+}(z) = T_{-}(z) \begin{pmatrix}
1 & s_{j}e^{2\sqrt{r}f(z)} \\ 0 & 1
\end{pmatrix}, & & z \in \mathcal{C}_{j}, \; j \in \{1,\ldots,n\}, \label{jumps for T 4}
\end{align}
where
\begin{align*}
\mathcal{B}_{j}:= \begin{cases} 
(-x_{j},-x_{j-1}), & \mbox{if } j > n, \\
(-x_{n},-a^{2}), & \mbox{if } j=n,
\end{cases} \qquad \mathcal{C}_{j}:= \begin{cases} 
(-x_{j},-x_{j-1}), & \mbox{if } j <n, \\
(-a^{2},-x_{n-1}), & \mbox{if } j=n,
\end{cases}
\end{align*}
with $x_{0}:=0$ and $x_{m+1}=+\infty$.
\item[(c)] As $z \to \infty$, we have 
\begin{equation}
\label{eq:Tasympinf s1 neq 0}
T(z) = \left( I + \frac{T_{1}}{z} + \bigO\left(z^{-2}\right) \right) z^{-\frac{\sigma_3}{4}} N, \qquad (T_{1})_{1,2} = - \frac{i a^{2} \sqrt{r}}{2} + \frac{\Phi_{1,12}(r \vec{x},\vec{s},a\sqrt{r})}{\sqrt{r}}
\end{equation}
where the principal branch is chosen for the root, and the matrix $T_{1} = T_{1}(r \vec{x},\vec{s},a\sqrt{r})$ is independent of $z$ and traceless.

As $z \to -x_{j}$, $j \in \{1,\dots,m\}$, we have
\begin{equation}\label{behavior of T near -xj}
T(z) = \begin{pmatrix}
\bigO(1) & \bigO(\log (z+x_{j})) \\ \bigO(1) & \bigO(\log (z+x_{j}))
\end{pmatrix}.
\end{equation}

As $z \to -a^{2}$, we have $T(z) = \bigO(1)$.

As $z \to 0$, $T(z) = \bigO(1)$. More precisely, we have
\begin{align*}
T(z) & = \begin{pmatrix}
1 & 0 \\
- \frac{i a^{2}}{2}\sqrt{r} & 1
\end{pmatrix} r^{\frac{\sigma_{3}}{4}} G_{0}(rz;r\vec{x},\vec{s},\alpha)(rz)^{\frac{a \sqrt{r}}{2}\sigma_{3}} \begin{pmatrix}
1 & s_{1}h(rz) \\ 0 & 1
\end{pmatrix} z^{-\frac{a\sqrt{r}}{2}\sigma_{3}}e^{-\sqrt{r} g_{0} \sigma_{3}}(I+\bigO(z)).
\end{align*}
\end{itemize}
Using that $f_{+}(x)+f_{-}(x)=0$ for all $x \in (-\infty,-a^{2})$, the jumps in \eqref{jumps for T to be factorized} can be factorized as follows
\begin{multline}\label{factorization of the jump}
\begin{pmatrix}
e^{-2\sqrt{r}f_{+}(z)} & s_{j} \\ 0 & e^{-2\sqrt{r}f_{-}(z)}
\end{pmatrix} = \begin{pmatrix}
1 & 0 \\
s_{j}^{-1}e^{-2\sqrt{r}f_{-}(z)} & 1
\end{pmatrix} \begin{pmatrix}
0 & s_{j} \\ -s_{j}^{-1} & 0
\end{pmatrix} \begin{pmatrix}
1 & 0 \\ 
s_{j}^{-1}e^{-2\sqrt{r}f_{+}(z)} & 1
\end{pmatrix}.
\end{multline}
\subsection{Second transformation: $T \mapsto S$}\label{subsection: S with s1 neq 0}
Here we proceed with the opening of the lenses, which is a standard step of the steepest descent method \cite{Deift}. Let $\Omega_{j,+}$ and $\Omega_{j,-}$ be open regions located above and below $\mathcal{B}_{j}$, respectively. We let $\gamma_{j,+}$ and $\gamma_{j,-}$ denote the parts of $\partial\Omega_{j,+} \cup \partial\Omega_{j,-}$ lying strictly in the upper and lower half-plane, respectively. The contours $\gamma_{j,+}, \gamma_{j,-}$, $j=n,\ldots,m$, are represented in Figure \ref{fig:contour for S s1 neq 0} in a situation where $m=3$ and $n=2$. The second transformation is defined by 
\begin{equation}\label{def of S s1 neq 0}
S(z) = T(z) \left\{ \begin{array}{l l}
\begin{pmatrix}
1 & 0 \\
-s_{j}^{-1}e^{-2\sqrt{r}f(z)} & 1
\end{pmatrix}, & \mbox{if } z \in \Omega_{j,+}, \, j \in \{n,\ldots,m\} \\
\begin{pmatrix}
1 & 0 \\
s_{j}^{-1}e^{-2\sqrt{r}f(z)} & 1
\end{pmatrix}, & \mbox{if } z \in \Omega_{j,-}, \, j \in \{n,\ldots,m\} \\
I, & \mbox{if } z \in \mathbb{C}\setminus  \cup_{j=n}^{m}(\Omega_{j,+}\cup \Omega_{j,-}).
\end{array} \right.
\end{equation}
Since $T$ is analytic in $\mathbb{C}\setminus \Sigma_{\Phi}$, we conclude that $S$ is analytic in $\C \backslash \Gamma_{S}$, where 
\begin{align*}
\Gamma_{S}:=(-\infty,0)\cup \gamma_{+}\cup \gamma_{-} \qquad \mbox{with} \qquad \gamma_{\pm} := \bigcup_{j=n}^{m+1} \gamma_{j,\pm}, \quad \gamma_{m+1,\pm} := -x_{m} + (0,e^{\pm \frac{2\pi i}{3}}\infty).
\end{align*}
The contour $\Gamma_{S}$ is oriented as shown in Figure \ref{fig:contour for S s1 neq 0}. 
Using the factorization \eqref{factorization of the jump} and the jumps \eqref{jumps for T 1}--\eqref{jumps for T 4}, we verify that $S$ satisfies the following jumps:
\begin{align*}
& S_{+}(z) = S_{-}(z)\begin{pmatrix}
0 & s_{j} \\ -s_{j}^{-1} & 0
\end{pmatrix}, & & z \in \mathcal{B}_{j}, \; j \in \{n,\ldots,m+1\} \\
& S_{+}(z) = S_{-}(z)\begin{pmatrix}
1 & 0 \\
s_{j}^{-1} e^{-2\sqrt{r}f(z)} & 1
\end{pmatrix}, & & z \in \gamma_{j,\pm}, \; j \in \{n,\ldots,m+1\}, \\
& S_{+}(z) = S_{-}(z) \begin{pmatrix}
1 & s_{j}e^{2\sqrt{r}f(z)} \\ 0 & 1
\end{pmatrix}, & & z \in \mathcal{C}_{j}, \; j \in \{1,\ldots,n\},
\end{align*}
where $x_{0} = 0$, $x_{m+1} = +\infty$ and $s_{m+1} := 1$.

\medskip Deforming the contour if necessary, we can (and do) assume without loss of generality that $\gamma_{\pm} \subset \mathcal{V}$, where $\mathcal{V}$ is as described in the statement of Lemma \ref{lemma: g function}. We infer from \eqref{re f}--\eqref{re f 2} that the jumps $S_{-}(z)^{-1}S_{+}(z)$ are exponentially close to the identity matrix as $r \to + \infty$ for $z \in \gamma_{+}\cup \gamma_{-} \cup (-a^{2},0)$. This convergence is uniform for $z$ bounded away from $(-\infty,-a^{2})$, but only pointwise for $z$ close to $(-\infty,-a^{2})$.

\begin{figure}
\centering
\begin{tikzpicture}
\draw[fill] (0,0) circle (0.05);
\draw (0,0) -- (8,0);
\draw (0,0) -- (120:3);
\draw (0,0) -- (-120:3);
\draw (0,0) -- (-3,0);

\draw (0,0) .. controls (1,1.3) and (2,1.3) .. (3,0);
\draw (0,0) .. controls (1,-1.3) and (2,-1.3) .. (3,0);
\draw (3,0) .. controls (3.5,1) and (4.5,1) .. (5,0);
\draw (3,0) .. controls (3.5,-1) and (4.5,-1) .. (5,0);

\draw[fill] (3,0) circle (0.05);
\draw[fill] (5,0) circle (0.05);
\draw[fill] (6.5,0) circle (0.05);
\draw[fill] (8,0) circle (0.05);

\node at (0.15,-0.3) {$-x_{m}$};
\node at (3,-0.3) {$-x_{2}$};
\node at (5,-0.3) {$-a^{2}$};
\node at (6.5,-0.3) {$-x_{1}$};
\node at (8.4,-0.3) {$0=x_{0}$};
\node at (-3,-0.3) {$-\infty=-x_{m+1}$};

\node at (1.5,1.2) {$\gamma_{m,+}$};
\node at (1.5,-1.25) {$\gamma_{m,-}$};
\node at (4,1) {$\gamma_{2,+}$};
\node at (4,-1.05) {$\gamma_{2,-}$};

\node at (1.6,0.5) {\small $\Omega_{m,+}$};
\node at (1.6,-0.55) {\small $\Omega_{m,-}$};
\node at (4.1,0.35) {\small $\Omega_{2,+}$};
\node at (4.1,-0.4) {\small $\Omega_{2,-}$};

\draw[black,arrows={-Triangle[length=0.18cm,width=0.12cm]}]
(-120:1.5) --  ++(60:0.001);
\draw[black,arrows={-Triangle[length=0.18cm,width=0.12cm]}]
(120:1.3) --  ++(-60:0.001);
\draw[black,arrows={-Triangle[length=0.18cm,width=0.12cm]}]
(180:1.5) --  ++(0:0.001);

\draw[black,arrows={-Triangle[length=0.18cm,width=0.12cm]}]
(0:1.5) --  ++(0:0.001);
\draw[black,arrows={-Triangle[length=0.18cm,width=0.12cm]}]
(0:4.1) --  ++(0:0.001);
\draw[black,arrows={-Triangle[length=0.18cm,width=0.12cm]}]
(0:5.8) --  ++(0:0.001);
\draw[black,arrows={-Triangle[length=0.18cm,width=0.12cm]}]
(0:7.3) --  ++(0:0.001);

\draw[black,arrows={-Triangle[length=0.18cm,width=0.12cm]}]
(1.55,0.97) --  ++(0:0.001);
\draw[black,arrows={-Triangle[length=0.18cm,width=0.12cm]}]
(1.55,-0.97) --  ++(0:0.001);

\draw[black,arrows={-Triangle[length=0.18cm,width=0.12cm]}]
(4.05,0.76) --  ++(0:0.001);
\draw[black,arrows={-Triangle[length=0.18cm,width=0.12cm]}]
(4.05,-0.76) --  ++(0:0.001);

\end{tikzpicture}
\caption{The contour $\Gamma_{S}$ with $m=3$ and $n = 2$.}
\label{fig:contour for S s1 neq 0}
\end{figure}
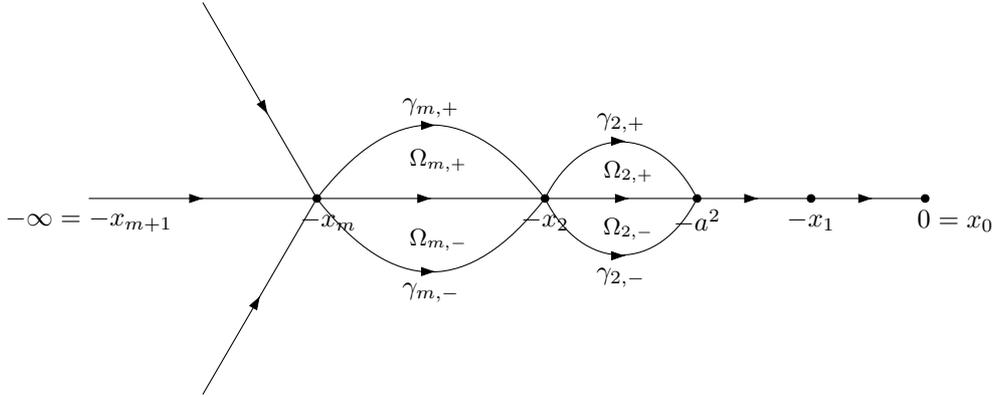
\subsection{Global parametrix}\label{subsection: Global param s1 neq 0}
The global parametrix is denoted $P^{(\infty)}$ and is defined as the solution to a RH problem whose jumps are obtained by ignoring the (pointwise) exponentially small jumps of $S$. We will show in Subsection \ref{subsection Small norm s1 neq 0} that $P^{(\infty)}$ is a good approximation to $S$ outside small neighborhoods of $-a^{2},-x_{1}$, $\ldots,-x_{m}$. 
\subsubsection*{RH problem for $P^{(\infty)}$}
\begin{enumerate}[label={(\alph*)}]
\item[(a)] $P^{(\infty)} : \C \backslash (-\infty,-a^{2}] \rightarrow \C^{2\times 2}$ is analytic.
\item[(b)] The jumps for $P^{(\infty)}$ are given by
\begin{align*}
& P^{(\infty)}_{+}(z) = P^{(\infty)}_{-}(z)\begin{pmatrix}
0 & s_{j} \\ -s_{j}^{-1} & 0
\end{pmatrix}, & & z \in \mathcal{B}_{j}, \, j \in \{n,\dots,m+1\}.
\end{align*}
\item[(c)] As $z \rightarrow \infty$, we have
\begin{equation}
\label{eq:Pinf asympinf s1 neq 0}
P^{(\infty)}(z) = \left( I + \frac{P^{(\infty)}_{1}}{z} + \bigO\left( z^{-2} \right) \right) z^{-\frac{\sigma_3}{4}} N,
\end{equation}
where $P_{1}^{(\infty)}$ is a matrix independent of $z$.

\item[(d)] As $z \to -x_{j}$, $j \in \{n,\dots,m\}$, we have $P^{(\infty)}(z) = \bigO(1)$.

As $z \to -a^{2}$, we have $P^{(\infty)}(z) = \bigO((z+a^{2})^{-1/4})$.
\end{enumerate}
A slightly more complicated version of this RH problem has actually been solved in \cite{CharlierBessel}. The solution $P^{(\infty)}$ is obtained by setting the parameters $\alpha$ and $-x_{1}$ of \cite[Section 5.3]{CharlierBessel} to $0$ and $-a^{2}$, respectively. The construction is as follows. Define
\begin{equation}\label{def of beta_j s1 neq 0}
\beta_{j} := \frac{1}{2\pi i}\log \frac{s_{j+1}}{s_{j}}, \qquad j = 1,\dots,m, \qquad \mbox{ with } s_{m+1} := 1,
\end{equation}
and for $\ell \in \{0,1,2,\dots\}$, define
\begin{equation}\label{d_ell in terms of beta_j s1 neq 0}
d_{\ell} = \frac{2i(-1)^{\ell}}{2\ell-1}\sum_{j=n}^{m} \beta_{j} (x_{j}-a^{2})^{\ell-\frac{1}{2}}.
\end{equation}
The unique solution to the RH problem for $P^{(\infty)}$ is given by
\begin{equation}\label{def of Pinf s1 neq 0}
P^{(\infty)}(z) = \begin{pmatrix}
1 & 0 \\ id_{1} & 1
\end{pmatrix}(z+a^{2})^{-\frac{\sigma_{3}}{4}}ND(z)^{-\sigma_{3}},
\end{equation}
where
\begin{align*}
D(z) = \exp \bigg(- \sqrt{z+a^{2}}\sum_{j=n}^{m} i \beta_{j} \int_{a^{2}}^{x_{j}}\frac{du}{\sqrt{u-a^{2}}(z+u)} \bigg).
\end{align*}
Furthermore, $P^{(\infty)}$ satisfies
\begin{equation}\label{Pinf 1 12 s1 neq 0}
P_{1,12}^{(\infty)} = i d_{1},
\end{equation}
and, for any $p \in \mathbb{N}_{>0}= \{1,2,\ldots\}$ and any $j \in \{n,\dots,m\}$, we have
\begin{align}
& D(z) =  \exp \left( \sum_{\ell = 1}^{p} \frac{d_{\ell}}{(z+a^{2})^{\ell-\frac{1}{2}}} + \bigO(z^{-p-\frac{1}{2}}) \right), & & \mbox{as } z \to \infty, \nonumber \\
& D(z) = \sqrt{s_{j}} (4(x_{j}-a^{2}))^{-\beta_{j}} \bigg(\prod_{\substack{k=n \\ k \neq j}}^{m} T_{k,j}^{-\beta_{k}}\bigg) (z+x_{j})^{\beta_{j}}(1+\bigO(z+x_{j})), & & \mbox{as } z \to -x_{j}, \, \im z > 0, \label{asymp of D near -xj} \\
& D(z) = \sqrt{s_{n}}\Big(1-d_{0}\sqrt{z+a^{2}}+\bigO(z+a^{2})\Big), & & \mbox{as } z \to -a^{2}, \label{asymp of D near -a2} \\
& \log D(0) = - a \sum_{j=n}^{m}i\beta_{j} \mathcal{I}_{j}, \label{explicit value of D0}
\end{align}
where 
\begin{align}\label{T and T tilde s1 neq 0}
& T_{k,j} = \frac{\sqrt{x_{j}-a^{2}}+\sqrt{x_{k}-a^{2}}}{|\sqrt{x_{j}-a^{2}}-\sqrt{x_{k}-a^{2}}|}, & & \mathcal{I}_{j} = \int_{a^{2}}^{x_{j}}\frac{du}{u\sqrt{u-a^{2}}}.
\end{align}

\subsection{Local parametrices}
We consider small open disks $\mathcal{D}_{p}$ centered at $p \in \{-a^{2},-x_{1},\ldots,-x_{m}\}$. The radii of the disks $\mathcal{D}_{-x_{j}}$, $j=1,\ldots,m$ are all chosen to be equal to $\delta>0$, where 
\begin{align}\label{def of delta regime 1}
\delta = \frac{1}{M}\min_{0\leq j <k \leq m} \{x_{k}-x_{j}, x_{n}-a^{2},a^{2}-x_{n-1} \}, \qquad M \geq 3, \qquad M \mbox{ fixed},
\end{align}
where we recall that $x_{0}:=0$. The condition $M \geq 3$ ensures that the disks do not intersect each other, and that $0 \notin \mathcal{D}_{-x_{1}} \cup \mathcal{D}_{-a^{2}}$. Since $(x_{1},\ldots,x_{n-1},a^{2},x_{n},\ldots,x_{m})$ lies in a compact subset of $\mathbb{R}_{\mathrm{ord}}^{+,m+1}$, $\delta$ remains bounded away from $0$. In this regime, there is no need to consider a disk around the origin.

\medskip In each of the disks, we will build a so-called local parametrix. We will show in Subsection \ref{subsection Small norm s1 neq 0} that the local parametrices are good approximations to $S$ inside the disks.

\medskip The local parametrix $P^{(p)}$ is defined in $\mathcal{D}_{p} \cup \partial \mathcal{D}_{p}$ as the solution to a RH problem whose jumps are identical to those of $S$. On the boundary of the disk, we require $P^{(p)}$ to ``match" with $P^{(\infty)}$, in the sense that
\begin{equation}\label{matching weak s1 neq 0}
P^{(p)}(z) = (I+o(1))P^{(\infty)}(z), \qquad \mbox{ as } r \to +\infty,
\end{equation}
uniformly for $z \in \partial \mathcal{D}_{p}$. We need to distinguish three different types of local parametrices:
\begin{itemize}
\item \vspace{-0.1cm} in $\mathcal{D}_{-x_{j}}$, $j=n,\ldots,m$, $P^{(-x_{j})}$ is built in terms of hypergeometric functions,
\item \vspace{-0.1cm} in $\mathcal{D}_{-x_{j}}$, $j=1,\ldots,n-1$, $P^{(-x_{j})}$ can be solved explicitly using elementary functions,
\item \vspace{-0.1cm} in $\mathcal{D}_{-a^{2}}$, $P^{(-a^{2})}$ is built in terms of Airy functions.
\end{itemize}
We will construct the local parametrices with the help of three model RH problems that have already been studied in the literature and which we recall in the appendix (Section \ref{Section:Appendix}). 
\subsubsection{Local parametrices around $-x_{j}$, $j = n,\dots,m$}\label{subsection: local param HG s1 neq 0}
The contruction of $P^{(-x_{j})}$ for $j \in \{n,\dots,m\}$ relies on a model RH problem whose solution $\Phi_{\mathrm{HG}}$ is built in terms of confluent hypergeometric functions. This model RH problem has been studied in \cite{ItsKrasovsky, FoulquieMartinezSousa} and we recall the properties of $\Phi_{\mathrm{HG}}$ in Section \ref{subsection: model RHP with HG functions} for the convenience of the reader. The function
\begin{equation}\label{conformal map for FH}
f_{-x_{j}}(z) := -2 \left\{ \begin{array}{l l}
f(z)-f_{+}(-x_{j}), & \mbox{if } \im z > 0 \\
-(f(z)-f_{-}(-x_{j})), & \mbox{if } \im z < 0
\end{array} \right. = -2i \int_{-x_{j}}^{z} \frac{\sqrt{-s-a^{2}}}{2s}ds
\end{equation}
has the following expansion as $z \to -x_{j}$
\begin{equation}\label{expansion conformal map s1 neq 0}
f_{-x_{j}}(z) = i c_{-x_{j}} (z+x_{j})\big(1+\bigO(z+x_{j})\big), \quad \mbox{ with } \quad c_{-x_{j}} = \frac{\sqrt{x_{j}-a^{2}}}{x_{j}} > 0.
\end{equation}
This shows that $f_{-x_{j}}$ is a conformal map in $\mathcal{D}_{-x_{j}}$, provided that $M$ in \eqref{def of delta regime 1} is chosen sufficiently large. In order to use the model RH problem for $\Phi_{\mathrm{HG}}$, we need $f_{-x_{j}}$ to map the contour $\Gamma_{S}\cap \mathcal{D}_{-x_{j}}$ to a subset of the contour $\Sigma_{\mathrm{HG}}$, where $\Sigma_{\mathrm{HG}}$ is shown in Figure \ref{Fig:HG}. Note that the function $f_{-x_{j}}$ automatically satisfies $f_{-x_{j}}(\mathbb{R}\cap \mathcal{D}_{-x_{j}})\subset i \mathbb{R}$. In Subsection \ref{subsection: S with s1 neq 0}, we had some freedom in the choice of the lenses $\gamma_{\pm}$. Now, we use this freedom to ensure that the lenses in a neighborhood of $-x_{j}$ are such that
\begin{equation}\label{deformation of the lenses local param xj s1 neq 0}
f_{-x_{j}}((\gamma_{j+1,+}\cup \gamma_{j,+})\cap \mathcal{D}_{-x_{j}}) \subset \Gamma_{3} \cup \Gamma_{2}, \qquad f_{-x_{j}}((\gamma_{j+1,-}\cup \gamma_{j,-})\cap \mathcal{D}_{-x_{j}}) \subset \Gamma_{5} \cup \Gamma_{6},
\end{equation}
where $\Gamma_{j}$, $j=2,3,5,6$ are the contours displayed in Figure \ref{Fig:HG}. Let us define
\begin{equation}\label{lol10}
P^{(-x_{j})}(z) = E_{-x_{j}}(z) \Phi_{\mathrm{HG}}(\sqrt{r}f_{-x_{j}}(z);\beta_{j})(s_{j}s_{j+1})^{-\frac{\sigma_{3}}{4}}e^{-\sqrt{r}f(z)\sigma_{3}},
\end{equation}
where $E_{-x_{j}}$ is analytic in $\mathcal{D}_{-x_{j}}$ and given by
\begin{align}\label{def of Ej s1 neq 0}
E_{-x_{j}}(z) = P^{(\infty)}(z) (s_{j} s_{j+1})^{\frac{\sigma_{3}}{4}} \left\{ \begin{array}{l l}
\ds \sqrt{\frac{s_{j}}{s_{j+1}}}^{\sigma_{3}}, & \im z > 0 \\
\begin{pmatrix}
0 & 1 \\ -1 & 0
\end{pmatrix}, & \im z < 0
\end{array} \right\} e^{\sqrt{r}f_{+}(-x_{j})\sigma_{3}}(\sqrt{r}f_{-x_{j}}(z))^{\beta_{j}\sigma_{3}}.
\end{align}
One can verify from the RH problem for $\Phi_{\mathrm{HG}}$ that $S(z)P^{(-x_{j})}(z)^{-1}$ has no jumps in $\mathcal{D}_{-x_{j}}$ and has a removable singularity at $z=-x_{j}$. Furthermore, using \eqref{Asymptotics HG}, we obtain
\begin{equation}\label{matching strong -x_j s1 neq 0}
P^{(-x_{j})}(z)P^{(\infty)}(z)^{-1} = I + \frac{1}{\sqrt{r}f_{-x_{j}}(z)}E_{-x_{j}}(z) \Phi_{\mathrm{HG},1}(\beta_{j})E_{-x_{j}}(z)^{-1} + \bigO\big(r^{-1}\big), 
\end{equation}
as $r \to + \infty$, uniformly for $z \in \partial \mathcal{D}_{-x_{j}}$. In particular $P^{(-x_{j})}$ satisfies \eqref{matching weak s1 neq 0}. Finally, using \eqref{def of Pinf s1 neq 0}, \eqref{asymp of D near -xj} and \eqref{expansion conformal map s1 neq 0}, we get
\begin{equation}\label{E_j at -x_j s1 neq 0}
E_{-x_{j}}(-x_{j}) = \begin{pmatrix}
1 & 0 \\ id_{1} & 1
\end{pmatrix} e^{-\frac{\pi i}{4}\sigma_{3}} (x_{j}-a^{2})^{-\frac{\sigma_{3}}{4}}N\Lambda_{j}^{\sigma_{3}},
\end{equation}
where
\begin{equation}\label{def Lambda_j s1 neq 0}
\Lambda_{j} = (4(x_{j}-a^{2}))^{\beta_{j}} \bigg( \prod_{\substack{k=n \\ k \neq j}}^{m} T_{k,j}^{\beta_{k}} \bigg)e^{\sqrt{r}f_{+}(-x_{j})}r^{\frac{\beta_{j}}{2}}c_{-x_{j}}^{\beta_{j}}.
\end{equation}
The quantities \eqref{E_j at -x_j s1 neq 0} and \eqref{def Lambda_j s1 neq 0} will be useful in Section \ref{Section: integration s1 >0}.
\subsubsection{Local parametrices around $-x_{j}$, $j = 1,\dots,n-1$}\label{subsection: local param explicit}
We note from \eqref{behavior of T near -xj} and \eqref{def of S s1 neq 0} that $S$ has a logarithmic singularity at $-x_{j}$, and from \eqref{def of Pinf s1 neq 0} that $P^{(\infty)}(z)$ remains bounded as $z \to -x_{j}$, $j = 1,\dots,n-1$. Therefore, even though the jumps for $S$ are exponentially small as $r \to + \infty$ uniformly for $z \in \mathcal{D}_{-x_{j}}\cap \mathbb{R}\setminus \{-x_{j}\}$, $P^{(\infty)}$ cannot be a good approximation to $S$ in $\mathcal{D}_{-x_{j}}$, and we need to construct a local parametrix. We define
\begin{align}\label{def of trivial local param -xj}
P^{(-x_{j})}(z) = P^{(\infty)}(z) \begin{pmatrix}
1 & h_{-x_{j}}(z) \\
0 & 1
\end{pmatrix},
\end{align}
where
\begin{align*}
& h_{-x_{j}}(z) = \frac{s_{j+1}}{2\pi i}\int_{-x_{j}-2\delta}^{-x_{j}} \frac{e^{2\sqrt{r}f(s)}}{s-z}ds + \frac{s_{j}}{2\pi i}\int_{-x_{j}}^{-x_{j}+2\delta} \frac{e^{2\sqrt{r}f(s)}}{s-z}ds.
\end{align*}
We easily check from the definition of $h_{-x_{j}}$ that $P^{(-x_{j})}$ has the same jumps as $S$ inside $\mathcal{D}_{-x_{j}}$. Furthermore, we infer from \eqref{def of g s1 neq 0} and \eqref{re f 2} that
\begin{align}\label{matching of trivial local param}
P^{(-x_{j})}(z)P^{(\infty)}(z)^{-1} = I + \bigO(e^{-c \sqrt{r}}), \qquad \mbox{as } r \to + \infty,
\end{align}
uniformly for $z \in \partial \mathcal{D}_{-x_{j}}$, for a certain $c>0$. In particular, $P^{(-x_{j})}$ satisfies \eqref{matching weak s1 neq 0}. Finally, using the expansion
\begin{align}\label{asymp of h -xj}
h_{-x_{j}}(z) = \frac{s_{j+1}-s_{j}}{2\pi i}e^{2\sqrt{r}f(-x_{j})}\log(z+x_{j})+\bigO(1), \qquad \mbox{as } z \to -x_{j},
\end{align}
one verifies that $S(z)P^{(-x_{j})}(z)^{-1}$ has a removable singularity at $z=-x_{j}$.
\subsubsection{Local parametrix around $-a^{2}$}\label{subsection: local param bessel s1 neq 0}
Recall that in this section $a$ remains bounded away from $0$ as $r \to + \infty$. The construction of $P^{(-a^{2})}$ is standard and relies on the model RH problem from \cite{DKMVZ1}, whose solution is denoted $\Phi_{\mathrm{Ai}}$. For the reader's convenience, we recall the properties of $\Phi_{\mathrm{Ai}}$ in Section \ref{subsec:Airy}. The function
\begin{equation}\label{conformal map near 0 s1 neq 0}
f_{-a^{2}}(z) = \bigg( -\frac{3}{2}f(z) \bigg)^{2/3} = \bigg( -\frac{3}{2}\int_{-a^{2}}^{z} \frac{\sqrt{s+a^{2}}}{2s}ds \bigg)^{2/3}
\end{equation}
has the following expansion
\begin{align}\label{expansion of fma2 at ma2}
f_{-a^{2}}(z) = \frac{z+a^{2}}{(2a^{2})^{2/3}} \bigg( 1+\frac{2}{5a^{2}}(z+a^{2}) + \bigO\big((z+a^{2})^{2}\big) \bigg), \qquad \mbox{as } z \to -a^{2}.
\end{align}
From \eqref{expansion of fma2 at ma2}, we infer that the map $f_{-a^{2}}$ is conformal in $\mathcal{D}_{-a^{2}}$, provided that $M$ in \eqref{def of delta regime 1} is chosen sufficiently large. 
In view of the jump contour for $\Phi_{\mathrm{Ai}}$ displayed in Figure \ref{figAiry} (left), we deform the lenses in a neighborhood of $-a^{2}$ such that
\begin{equation}\label{deformation of the lenses local param 0 s1 neq 0}
f_{-a^{2}}(\gamma_{n,+}) \subset e^{\frac{2\pi i}{3}}\mathbb{R}^{+}, \qquad f_{-a^{2}}(\gamma_{n,-}) \subset e^{-\frac{2\pi i}{3}}\mathbb{R}^{+}.
\end{equation} 
We define
\begin{equation}\label{def of P^-x1 s1 neq 0}
P^{(-a^{2})}(z) = E_{-a^{2}}(z)\Phi_{\mathrm{Ai}}(r^{\frac{1}{3}}f_{-a^{2}}(z))s_{n}^{-\frac{\sigma_{3}}{2}}e^{-\sqrt{r}f(z)\sigma_{3}},
\end{equation}
where $E_{-a^{2}}$ is analytic in $\mathcal{D}_{-a^{2}}$ and given by
\begin{equation}
E_{-a^{2}}(z) = P^{(\infty)}(z)s_{n}^{\frac{\sigma_{3}}{2}}N^{-1} f_{-a^{2}}(z)^{\frac{\sigma_{3}}{4}} r^{\frac{\sigma_{3}}{12}}.
\end{equation}
From \eqref{deformation of the lenses local param 0 s1 neq 0}, \eqref{def of P^-x1 s1 neq 0} and the RH problem for $\Phi_{\mathrm{Ai}}$, we verify that $S(z)P^{(-a^2)}(z)^{-1}$ has no jumps in $\mathcal{D}_{-a^{2}}$ and remains bounded as $z \to -a^{2}$. Also, using \eqref{Asymptotics Airy}, we obtain
\begin{equation}\label{matching strong 0 s1 neq 0}
P^{(-a^{2})}(z)P^{(\infty)}(z)^{-1} = I + \frac{1}{\sqrt{r}f_{-a^{2}}(z)^{3/2}}P^{(\infty)}(z)s_{n}^{\frac{\sigma_{3}}{2}}\Phi_{\mathrm{Ai},1}s_{n}^{-\frac{\sigma_{3}}{2}}P^{(\infty)}(z)^{-1} + \bigO(r^{-1}),
\end{equation}
as $r \to +\infty$ uniformly for $z \in \partial \mathcal{D}_{-a^{2}}$. In particular, $P^{(-a^{2})}$ satisfies \eqref{matching weak s1 neq 0}. Finally, after a long computation using \eqref{def of Pinf s1 neq 0}, \eqref{asymp of D near -a2} and \eqref{expansion of fma2 at ma2}, we obtain
\begin{equation}\label{E0 at 0 s1 neq 0}
E_{-a^{2}}(-a^{2}) = \begin{pmatrix}
1 & 0 \\ id_{1} & 1
\end{pmatrix} \begin{pmatrix}
1 & -id_{0} \\
0 & 1
	\end{pmatrix} (r^{\frac{1}{12}} 2^{-\frac{1}{6}}a^{-\frac{1}{3}})^{\sigma_{3}}.
\end{equation}
The above identity will be used in Section \ref{Section: integration s1 >0}.

\subsection{Small norm problem}\label{subsection Small norm s1 neq 0}

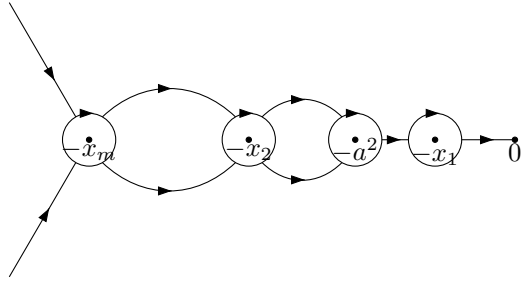
\begin{figure}
\centering
\begin{tikzpicture}[scale=0.7]
\draw[fill] (0,0) circle (0.05);
\draw (0,0) circle (0.5);

\draw (120:0.5) -- (120:3);
\draw (-120:0.5) -- (-120:3);

\draw ($(0,0)+(60:0.5)$) .. controls (1,1.15) and (2,1.15) .. ($(3,0)+(120:0.5)$);
\draw ($(0,0)+(-60:0.5)$) .. controls (1,-1.15) and (2,-1.15) .. ($(3,0)+(-120:0.5)$);
\draw ($(3,0)+(60:0.5)$) .. controls (3.7,0.88) and (4.3,0.88) .. ($(5,0)+(120:0.5)$);
\draw ($(3,0)+(-60:0.5)$) .. controls (3.7,-0.88) and (4.3,-0.88) .. ($(5,0)+(-120:0.5)$);
\draw (5.5,0)--(6,0);
\draw (7,0)--(8,0);

\draw[fill] (3,0) circle (0.05);
\draw (3,0) circle (0.5);
\draw[fill] (5,0) circle (0.05);
\draw (5,0) circle (0.5);
\draw[fill] (6.5,0) circle (0.05);
\draw (6.5,0) circle (0.5);
\draw[fill] (8,0) circle (0.05);

\node at (0.,-0.2) {$-x_{m}$};
\node at (3,-0.2) {$-x_{2}$};
\node at (5,-0.2) {$-a^{2}$};
\node at (6.5,-0.3) {$-x_{1}$};
\node at (8,-0.25) {$0$};

\draw[black,arrows={-Triangle[length=0.18cm,width=0.12cm]}]
(-120:1.5) --  ++(60:0.001);
\draw[black,arrows={-Triangle[length=0.18cm,width=0.12cm]}]
(120:1.3) --  ++(-60:0.001);
\draw[black,arrows={-Triangle[length=0.18cm,width=0.12cm]}]
($(0.08,0)+(90:0.5)$) --  ++(0:0.001);

\draw[black,arrows={-Triangle[length=0.18cm,width=0.12cm]}]
($(3.08,0)+(90:0.5)$) --  ++(0:0.001);
\draw[black,arrows={-Triangle[length=0.18cm,width=0.12cm]}]
($(5.08,0)+(90:0.5)$) --  ++(0:0.001);
\draw[black,arrows={-Triangle[length=0.18cm,width=0.12cm]}]
($(6.58,0)+(90:0.5)$) --  ++(0:0.001);

\draw[black,arrows={-Triangle[length=0.18cm,width=0.12cm]}]
(1.55,0.97) --  ++(0:0.001);
\draw[black,arrows={-Triangle[length=0.18cm,width=0.12cm]}]
(1.55,-0.97) --  ++(0:0.001);

\draw[black,arrows={-Triangle[length=0.18cm,width=0.12cm]}]
(4.05,0.76) --  ++(0:0.001);
\draw[black,arrows={-Triangle[length=0.18cm,width=0.12cm]}]
(4.05,-0.76) --  ++(0:0.001);

\draw[black,arrows={-Triangle[length=0.18cm,width=0.12cm]}]
(5.85,0) --  ++(0:0.001);
\draw[black,arrows={-Triangle[length=0.18cm,width=0.12cm]}]
(7.6,0) --  ++(0:0.001);

\end{tikzpicture} 
\caption{The contour $\Gamma_{R}$ with $m=3$ and $n=2$.}
\label{fig:contour for R s1 neq 0}
\end{figure}
In this section, $c$ and $C$ denote generic positive constants that may change within a computation. We will think of $c$ as being sufficiently small (but fixed), and $C$ as being sufficiently large (but fixed). Define
\begin{equation}\label{def of R s1 neq 0}
R(z) := \left\{ \begin{array}{l l}
S(z)P^{(\infty)}(z)^{-1}, & \mbox{for } z \in \mathbb{C}\setminus \big(\bigcup_{j=1}^{m}\mathcal{D}_{-x_{j}} \cup \mathcal{D}_{-a^2}\big), \\
S(z)P^{(-x_{j})}(z)^{-1}, & \mbox{for } z \in \mathcal{D}_{-x_{j}}, \, j \in \{1,\dots,m\}, \\
S(z)P^{(-a^{2})}(z)^{-1}, & \mbox{for } z \in \mathcal{D}_{-a^{2}}.
\end{array} \right.
\end{equation}
From the analysis of Subsections \ref{subsection: local param HG s1 neq 0}--\ref{subsection: local param bessel s1 neq 0}, $R$ has no jumps inside the $m+1$ disks, and the singularities of $R$ at $-x_{1},\ldots,-x_{m},-a^{2}$ are removable. Therefore, $R$ is analytic inside the $m+1$ disks and $R$ has jumps only on the contour 
\begin{align*}
\Gamma_{R} = \bigg( \partial \mathcal{D}_{-a^{2}} \cup \bigcup_{j=1}^{m} \partial \mathcal{D}_{-x_{j}} \cup \gamma_{+} \cup \gamma_{-} \cup (-a^{2},0) \bigg) \setminus \bigg(\mathcal{D}_{-a^{2}} \cup \bigcup_{j=1}^{m} \mathcal{D}_{-x_{j}} \bigg),
\end{align*}
see also Figure \ref{fig:contour for R s1 neq 0}, and for convenience we orient the boundaries of the disks in the clockwise direction. The jumps $J_{R}(z) := R_{-}(z)^{-1}R_{+}(z)$ are given by
\begin{align*}
J_{R}(z) = \begin{cases}
P^{(\infty)}(z)J_{S}(z)P^{(\infty)}(z)^{-1}, & \mbox{for } z \in (\gamma_{+} \cup  \gamma_{-} \cup (-a^{2},0)) \setminus \big(\mathcal{D}_{-a^{2}} \cup \bigcup_{j=1}^{m} \mathcal{D}_{-x_{j}} \big), \\
P^{(-x_{\star})}(z)P^{(\infty)}(z)^{-1}, & \mbox{for } z \in \partial \mathcal{D}_{-x_{\star}}, \; x_{\star} \in \{x_{1},\ldots,x_{m},a^{2}\}.
\end{cases}
\end{align*}
where $J_{S}(z) := S_{-}^{-1}(z)S_{+}(z)$. Using Lemma \ref{lemma: g function} and \eqref{def of Pinf s1 neq 0}, \eqref{matching strong -x_j s1 neq 0}, \eqref{matching of trivial local param}, \eqref{matching strong 0 s1 neq 0}, we infer that there exists $c>0$ such that, as $r \to + \infty$ we have
\begin{align}\label{asymp for JR with a bounded away from 0}
J_{R}(z)-I = \begin{cases}
\bigO(e^{-c \, \sqrt{r}\sqrt{z}}), & \mbox{unif. for } z \in (\gamma_{+}\cup \gamma_{-})\cap \Gamma_{R}, \\
\bigO(e^{-c \, \sqrt{r}}), & \mbox{unif. for } z \in \bigcup_{j=1}^{n-1} \partial \mathcal{D}_{-x_{j}} \cup (-a^{2},0) \setminus (\mathcal{D}_{-a^{2}} \cup \bigcup_{j=1}^{n-1} \mathcal{D}_{-x_{j}}), \\
\frac{J_{R}^{(1)}(z)}{\sqrt{r}}+\bigO(\tfrac{1}{r}), & \mbox{unif. for } z \in \partial \mathcal{D}_{-a^{2}} \cup \bigcup_{j=n}^{m} \partial \mathcal{D}_{-x_{j}},
\end{cases}
\end{align} \\[-0.2cm]
where $J_{R}^{(1)}$ is given by
\begin{align} \label{def of JRp1p}
J_{R}^{(1)}(z) := \begin{cases} 
\frac{1}{f_{-x_{j}}(z)}E_{-x_{j}}(z) \Phi_{\mathrm{HG},1}(\beta_{j})E_{-x_{j}}(z)^{-1}, & z \in \partial \mathcal{D}_{-x_{j}}, \, j=n,\ldots,m, \\
\frac{1}{f_{-a^{2}}(z)^{3/2}}P^{(\infty)}(z)s_{n}^{\frac{\sigma_{3}}{2}}\Phi_{\mathrm{Ai},1}s_{n}^{-\frac{\sigma_{3}}{2}}P^{(\infty)}(z)^{-1}, & z \in \partial\mathcal{D}_{-a^{2}}.
\end{cases}
\end{align}
In particular, for each $1 \leq p \leq \infty$ and any choice of $k_{1},\ldots,k_{m}\in \mathbb{N}_{\geq 0}$, there exist constants $c > 0$ and $C >0$ such that for all $r \geq 2$ we have
\begin{subequations}\label{estimates Lp norms}
\begin{align}
& \| \partial_{\beta}^{k}(J_{R}-I) \|_{L^{p}(\Gamma_{R}\setminus (\partial \mathcal{D}_{-a^{2}} \cup \bigcup_{l=n}^{m} \partial \mathcal{D}_{-x_{l}}))} \leq e^{-c \sqrt{r}}, \label{estimate Lp norm 1} \\
&\| \partial_{\beta}^{k}(J_{R}-I-\tfrac{J_{R}^{(1)}(z)}{\sqrt{r}}) \|_{L^{p}(\partial \mathcal{D}_{-a^{2}} \cup \bigcup_{l=n}^{m} \partial \mathcal{D}_{-x_{l}})} \leq C \, \frac{\log^{k} r}{r}, \label{estimate Lp norm 2}
	\\
&\| \partial_{\beta}^{k}(J_{R}-I) \|_{L^{p}(\Gamma_R)} \leq C \, \frac{\log^{k} r}{\sqrt{r}}, \label{estimate Lp norm 3}
\end{align}
\end{subequations}
where $k=k_{1}+\ldots+k_{m}$ and $\partial_{\beta}^{k}=\partial_{\beta_{1}}^{k_{1}}\ldots \partial_{\beta_{m}}^{k_{m}}$. The factors $\log^{k}r$ in \eqref{estimates Lp norms} is due to the factors $r^{\pm \beta_{j}}$ appearing in the entries of $J_{R}$, see \eqref{lol10}--\eqref{matching strong -x_j s1 neq 0}. It is also easy to see from \eqref{def of Pinf s1 neq 0}, \eqref{matching strong -x_j s1 neq 0}, \eqref{matching strong 0 s1 neq 0}, \eqref{def of trivial local param -xj} that the the estimates \eqref{asymp for JR with a bounded away from 0} and \eqref{estimates Lp norms} hold uniformly for $s_{1},\ldots,s_{m}$ in compact subsets of $(0,+\infty)$, and uniformly for $(x_{1},\ldots,x_{n-1},a^{2},x_{n},\ldots,x_{m})$ in compact subsets of $\mathbb{R}_{\mathrm{ord}}^{+,m+1}$. Finally, we also note that 
\begin{itemize}
\item \vspace{-0.1cm} since $S(z)$ and $P^{(\infty)}(z)$ are $\bigO(1)$ as $z \to 0$, $R(z)$ remains bounded for $z$ near the endpoint $0$,
\item \vspace{-0.1cm} \eqref{def of Pinf s1 neq 0}, \eqref{lol10}, \eqref{def of P^-x1 s1 neq 0} and \eqref{def of trivial local param -xj} imply that $R(z)$ remains bounded as $z$ approaches an intersection point of $\Gamma_{R}$,
\item \vspace{-0.1cm} \eqref{eq:Tasympinf s1 neq 0}, \eqref{eq:Pinf asympinf s1 neq 0} and \eqref{def of R s1 neq 0} imply that $R(z) = I+\bigO(z^{-1})$ as $z \to \infty$.
\end{itemize}
Therefore, $R$ satisfies a so-called small norm RH problem, and one can prove existence of $R$ for sufficiently large $r$ and compute its asymptotics following the method of Deift and Zhou \cite{DeiftZhou}. Let $\mathcal{C}:L^{2}(\Gamma_{R}) \to L^{2}(\Gamma_{R})$ be the operator defined by
\begin{align*}
\mathcal{C}f(z) = \frac{1}{2\pi i}\int_{\Gamma_{R}} \frac{f(s)}{s-z}dz, \qquad f \in L^{2}(\Gamma_{R}),
\end{align*}
and let $\mathcal{C}_{+}f$ and $\mathcal{C}_{-}f$ denote the left and right non-tangential limits of $\mathcal{C}f$. Since $J_{R}-I \in L^2(\Gamma_R) \cap L^\infty(\Gamma_R)$, we can define the Cauchy operator 
\begin{align*}
\mathcal{C}_{J_{R}} : L^{2}(\Gamma_{R}) + L^{\infty}(\Gamma_{R}) \to L^{2}(\Gamma_{R}), \qquad \mathcal{C}_{J_{R}}f=\mathcal{C}_{-}((J_{R}-I)f), \qquad f \in L^{2}(\Gamma_{R})+ L^{\infty}(\Gamma_{R}).
\end{align*}
Since 
\begin{align}\label{calCJRestimate}
  \|\mathcal{C}_{J_R} \|_{L^2(\Gamma_R) \to L^2(\Gamma_R)} \leq C \|J_R - I\|_{L^\infty(\Gamma_R)},
\end{align}
the estimate \eqref{estimate Lp norm 3} with $p=+\infty$ and $k_{1}=\dots=k_{m}=0$ implies that for all sufficiently large $r$ the operator $I-\mathcal{C}_{J_{R}}$ can be inverted as a Neumann series:
\begin{align}\label{neumann series}
(I-\mathcal{C}_{J_{R}})^{-1} = \sum_{\ell=0}^{+\infty}\mathcal{C}_{J_{R}}^{\ell}.
\end{align}
Hence, for all sufficiently large $r$, we have
\begin{align}\label{R in terms of muR}
R = I + \mathcal{C}(\mu_{R}(J_{R}-I)), \qquad \mbox{where} \qquad \mu_{R} := I + (I-\mathcal{C}_{J_{R}})^{-1}\mathcal{C}_{J_{R}}(I).
\end{align}
Furthermore, using \eqref{estimate Lp norm 3}, one deduces that, for $k_{1},\ldots,k_{m}\in \mathbb{N}_{\geq 0}$, 
\begin{align}\label{muminusIestimate}
\|\partial_{\beta}^{k}(\mu_{R}-I)\|_{L^{2}(\Gamma_{R})} = \|\partial_{\beta}^{k}((I-\mathcal{C}_{J_{R}})^{-1}\mathcal{C}_{J_{R}}(I))\|_{L^{2}(\Gamma_{R})}  \leq C \, \frac{\log^{k} r}{\sqrt{r}},
\end{align} 
where as before $k=k_{1}+\dots+k_{m}$ and $\partial_{\beta}^{k}=\partial_{\beta_{1}}^{k_{1}}\ldots \partial_{\beta_{m}}^{k_{m}}$. Using \eqref{estimates Lp norms}, \eqref{muminusIestimate}, and the representation
\begin{align*}
R = I + \mathcal{C}\Big(\tfrac{J_{R}^{(1)}(z)}{\sqrt{r}}\Big) + \mathcal{C}\Big(J_{R}-I-\tfrac{J_{R}^{(1)}(z)}{\sqrt{r}}\Big) + \mathcal{C}((\mu_{R}-I)(J_{R}-I)),
\end{align*}
we get
\begin{align}
& \partial_{\beta}^{k}R(z) = \partial_{\beta}^{k}\bigg(I + \frac{R^{(1)}(z)}{\sqrt{r}}\bigg) + \bigO\Big(\frac{\log^{k}r}{r}\Big), \qquad \partial_{\beta}^{k}R^{(1)}(z) = \bigO\big(\log^{k}r\big), & & \mbox{ as } r \to  +\infty, \label{eq: asymp R inf s1 neq 0} 
\end{align}
uniformly for $z \in \mathbb{C}\setminus \Gamma_{R}$, for $(x_{1},\ldots,x_{n-1},a^{2},x_{n},\ldots,x_{m})$ in compact subsets of $\mathbb{R}_{\mathrm{ord}}^{+,m+1}$ and for $\beta_{1},\ldots,\beta_{m}$ in compact subsets of $i \mathbb{R}$,  where
\begin{equation}\label{def of R1}
R^{(1)}(z) = \frac{1}{2\pi i}\int_{\partial \mathcal{D}_{-a^{2}}} \frac{J_{R}^{(1)}(s)}{s-z}ds + \sum_{j=n}^{m}\frac{1}{2\pi i}\int_{\partial\mathcal{D}_{-x_{j}}} \frac{J_{R}^{(1)}(s)}{s-z}ds.
\end{equation}
The fact that \eqref{eq: asymp R inf s1 neq 0} holds uniformly for $z$ close to $\Gamma_R$ can be seen, for example, by deforming the contour $\Gamma_R$ slightly.
From \eqref{def of JRp1p}, we infer that $J_{R}^{(1)}$ can be analytically continued to 
\begin{align*}
\bigg(\mathcal{D}_{-a^{2}}\cup\bigcup_{j=n}^{m}\mathcal{D}_{-x_{j}}\bigg)\setminus \{-a^{2},-x_{n},\ldots,-x_{m}\},
\end{align*}
and has a double pole at $-a^{2}$ and simple poles at $-x_{n},\ldots,-x_{m}$. Hence,
\begin{align}\label{expression for R^1 s1 neq 0}
R^{(1)}(z) = & \; \frac{1}{z+a^{2}}\mbox{Res}\big(J_{R}^{(1)}(s),s = -a^{2}\big)+\frac{1}{(z+a^{2})^{2}}\mbox{Res}\big((s+a^{2})J_{R}^{(1)}(s),s = -a^{2}\big) \nonumber \\
& +\sum_{j=n}^{m} \frac{1}{z+x_{j}}\mbox{Res}(J_{R}^{(1)}(s),s = -x_{j}), \qquad \mbox{ for } z \in \mathbb{C}\setminus \Bigg(\mathcal{D}_{-a^{2}}\cup\bigcup_{j=n}^{m}\mathcal{D}_{-x_{j}}\Bigg).
\end{align}
A long but straightforward computation using (\ref{def of Pinf s1 neq 0}), \eqref{asymp of D near -a2}, \eqref{expansion of fma2 at ma2} and (\ref{def of JRp1p}) gives
\begin{subequations}\label{residues at minus a2}
\begin{align}
& \mbox{Res}\big((s+a^{2})J_{R}^{(1)}(s),s = -a^{2}\big) = \frac{5a^{2}d_{1}}{24}\begin{pmatrix}
1 & id_{1}^{-1} \\ id_{1} & -1
\end{pmatrix}, \\ 
& \mbox{Res}\big(J_{R}^{(1)}(s),s = -a^{2}\big) = \begin{pmatrix}
\frac{1}{8}\Big( -d_{1}+4a^{2}d_{0}(1+d_{0}d_{1}) \Big) & \frac{i}{8}(-1+4a^{2}d_{0}^{2}) \\
\frac{i}{24}\Big( -3d_{1}^{2}+a^{2}(7+24d_{0}d_{1}+12d_{0}^{2}d_{1}^{2}) \Big) & \frac{1}{8}\Big( d_{1}-4a^{2}d_{0}(1+d_{0}d_{1}) \Big)
\end{pmatrix}.
\end{align}
\end{subequations}
Also, using \eqref{expansion conformal map s1 neq 0}, \eqref{E_j at -x_j s1 neq 0}, \eqref{def Lambda_j s1 neq 0} and \eqref{def of JRp1p}, for $j \in \{n,\dots,m\}$, we obtain
\begin{align}\nonumber
\mbox{Res}\left( J_{R}^{(1)}(s),s= -x_{j} \right) = &\; \frac{\beta_{j}^{2}}{ic_{-x_{j}}}\begin{pmatrix}
1 & 0 \\ id_{1} & 1
\end{pmatrix}e^{-\frac{\pi i}{4}\sigma_{3}}(x_{j}-a^{2})^{-\frac{\sigma_{3}}{4}} N \begin{pmatrix}
-1 & \widetilde{\Lambda}_{j,1} \\ -\widetilde{\Lambda}_{j,2} & 1
\end{pmatrix} 
	\\ \label{residue at minus xj}
& \times N^{-1} (x_{j}-a^{2})^{\frac{\sigma_{3}}{4}}e^{\frac{\pi i}{4}\sigma_{3}}\begin{pmatrix}
1 & 0 \\ -id_{1} & 1
\end{pmatrix},
\end{align}
where
\begin{equation}
\widetilde{\Lambda}_{j,1} = \tau(\beta_{j})\Lambda_{j}^{2} \qquad \mbox{ and } \qquad \widetilde{\Lambda}_{j,2} = \tau(-\beta_{j})\Lambda_{j}^{-2}.
\end{equation}

\subsection{Large $r$ asymptotics for $\Phi$ with $\alpha \to + \infty$: regime \ref{item 3 in thm}}\label{subsection: regime 3}
In this subsection we consider the regime where $r \to + \infty$, $a \in (0,+\infty)$ is fixed and where $\vec{x}\in \mathbb{R}_{\mathrm{ord}}^{+,m}$ satisfies $0 < a^{2} < x_{1} < \cdots < x_{m}$ and \eqref{Airy regime in the main first thm}. In particular, this means that $n = 1$. 
The asymptotic analysis of $\Phi$ in this regime is essentially the same as the one carried out in Subsections \ref{subsection: g function s1 neq 0}--\ref{subsection Small norm s1 neq 0} (with $n=1$), but some extra care is needed since the $x_{j}$'s converge to $a^{2}$. In particular, the radius $\delta$ of the disks in \eqref{def of delta regime 1} must now depend on $r$ and decrease as $r \to + \infty$, which means that the disks shrink as $r \to + \infty$. We next discuss the consequences this has for the construction of the local parametrices and of the solution $R$ to the small norm RH problem. Note that \eqref{Airy regime in the main first thm b} implies that $\delta$ is of the same order as $|x_m - a^2|$.

\subsubsection{Modification of the error terms in the construction of $P^{(-x_{j})}$, $j=1,\ldots,m$}
By \eqref{def of Pinf s1 neq 0} and \eqref{lol10}, we have
\begin{align}\nonumber
P^{(-x_{j})}(z)P^{(\infty)}(z)^{-1} 
= &\; E_{-x_{j}}(z) \Phi_{\mathrm{HG}}(\sqrt{r}f_{-x_{j}}(z);\beta_{j})(s_{j}s_{j+1})^{-\frac{\sigma_{3}}{4}}e^{-\sqrt{r}f(z)\sigma_{3}}
	\\\label{PmxjPinftymodification}
& \times D(z)^{\sigma_{3}}N^{-1}(z+a^{2})^{\frac{\sigma_{3}}{4}}\begin{pmatrix}
1 & 0 \\ -id_{1} & 1
\end{pmatrix}.
\end{align}
The map $f_{-x_{j}}$ defined in \eqref{conformal map for FH} satisfies
\begin{equation}\label{expansion conformal map s1 neq 0 regime 3}
f_{-x_{j}}(z) = i c_{-x_{j}} (z+x_{j})\bigg(1+\bigO\Big( \frac{z+x_{j}}{x_{j}-a^{2}} \Big)\bigg), \quad c_{-x_{j}} = \frac{\sqrt{x_{j}-a^{2}}}{x_{j}} > 0, \quad \mbox{ as } z \to -x_{j}.
\end{equation}
As can be seen from the above error term, $f_{-x_{j}}$ is a conformal map in $\mathcal{D}_{-x_{j}}$ provided that $\delta \leq c |x_{j}-a^{2}|$ where $c>0$ is a sufficiently small constant. Note also that $c_{-x_{j}} = \bigO(\sqrt{x_{j}-a^{2}}) = \bigO(\delta^{1/2}) \to 0$ as $x_{j} \to a^{2}$. 
Hence, $|f_{-x_{j}}(z)| \asymp \delta^{3/2}$ uniformly for $z \in \partial \mathcal{D}_{-x_{j}}$.
In particular, the argument $\sqrt{r}f_{-x_{j}}(z)$ of $\Phi_{\mathrm{HG}}$ in \eqref{PmxjPinftymodification} satisfies $|\sqrt{r}f_{-x_{j}}(z)| \asymp \delta^{3/2}\sqrt{r}$.
Moreover, for $z \in \partial \mathcal{D}_{-x_{j}}$, we have $E_{-x_{j}}(z)  = \bigO(\delta^{-1/4})$, $D(z) = \bigO(1)$, and $(z+a^2)^{\sigma_3/4}= \bigO(\delta^{-1/4})$.
Consequently, using the asymptotic formula \eqref{Asymptotics HG} for $\Phi_{\mathrm{HG}}$ in \eqref{PmxjPinftymodification}, we conclude that
\begin{subequations}\label{matching strong -x_j s1 neq 0 Regime 3}
\begin{equation}\label{matching strong -x_j s1 neq 0 Regime 3a}
P^{(-x_{j})}(z)P^{(\infty)}(z)^{-1} = I + \frac{1}{\sqrt{r}f_{-x_{j}}(z)}E_{-x_{j}}(z) \Phi_{\mathrm{HG},1}(\beta_{j})E_{-x_{j}}(z)^{-1} + \bigO\big( \delta^{-7/2} r^{-1}\big)
\end{equation}
and
\begin{align}
\frac{1}{\sqrt{r}f_{-x_{j}}(z)}E_{-x_{j}}(z) \Phi_{\mathrm{HG},1}(\beta_{j})E_{-x_{j}}(z)^{-1} = \bigO(\delta^{-2} r^{-1/2})
\end{align}
\end{subequations}
uniformly for $z \in \partial \mathcal{D}_{-x_{j}}$ as $r \to +\infty$ and $\delta \to 0$.
Equation \eqref{matching strong -x_j s1 neq 0 Regime 3a} should be compared to the expansion \eqref{matching strong -x_j s1 neq 0} of $P^{(-x_{j})}(z)P^{(\infty)}(z)^{-1}$ obtained earlier.
It follows from \eqref{matching strong -x_j s1 neq 0 Regime 3} that $P^{(-x_{j})}$ satisfies \eqref{matching weak s1 neq 0}, provided that $\sqrt{r}\delta^{2}$ tends to infinity as $r \to + \infty$.

\subsubsection{Modification of the error terms in the construction of $P^{(-a^{2})}$}
By \eqref{expansion of fma2 at ma2}, we have $|f_{-a^{2}}(z)| \asymp \delta$, and hence
$\sqrt{r}f_{-a^{2}}(z)^{3/2} \asymp \delta^{3/2} \sqrt{r}$, uniformly for $z \in \partial \mathcal{D}_{-a^{2}}$. Moreover, $P^{(\infty)}(z) = \bigO(\delta^{-1/4})$ uniformly for $z \in \partial \mathcal{D}_{-a^{2}}$.
Using these observations, arguments similar to those used to obtain \eqref{matching strong -x_j s1 neq 0 Regime 3} show that the matching condition \eqref{matching strong 0 s1 neq 0} becomes
\begin{subequations}\label{matching strong 0 s1 neq 0 regime 3}
\begin{equation}
P^{(-a^{2})}(z)P^{(\infty)}(z)^{-1} = I + \frac{1}{\sqrt{r}f_{-a^{2}}(z)^{3/2}}P^{(\infty)}(z)s_{n}^{\frac{\sigma_{3}}{2}}\Phi_{\mathrm{Ai},1}s_{n}^{-\frac{\sigma_{3}}{2}}P^{(\infty)}(z)^{-1} + \bigO(\delta^{-7/2}r^{-1})
\end{equation}
and that
\begin{align}
\frac{1}{\sqrt{r}f_{-a^{2}}(z)^{3/2}}P^{(\infty)}(z)s_{n}^{\frac{\sigma_{3}}{2}}\Phi_{\mathrm{Ai},1}s_{n}^{-\frac{\sigma_{3}}{2}}P^{(\infty)}(z)^{-1} = \bigO(\delta^{-2} r^{-1/2})
\end{align}
\end{subequations}
uniformly for $z \in \partial \mathcal{D}_{-a^{2}}$ as $r \to +\infty$ and $\delta \to 0$. 
In particular, $P^{(-a^{2})}$ satisfies \eqref{matching weak s1 neq 0}, provided that $\sqrt{r} \delta^{2}$ tends to infinity as $r \to + \infty$.

\subsubsection{Modification of the error terms in the construction of $R$}
Let $R$ be defined as in \eqref{def of R s1 neq 0} with $n=1$. Using Lemma \ref{lemma: g function}, \eqref{def of Pinf s1 neq 0}, \eqref{matching strong -x_j s1 neq 0 Regime 3} and \eqref{matching strong 0 s1 neq 0 regime 3}, we infer that there exists a $c>0$ such that, as $r \to + \infty$,
\begin{align}\label{asymp for JR with a bounded away from 0 regime 3}
J_{R}(z)-I = \begin{cases}
\bigO(e^{-c \, \delta^{3/2}\sqrt{r}\sqrt{z}}), & \mbox{uniformly for } z \in (\gamma_{+}\cup \gamma_{-})\cap \Gamma_{R}, \\
\bigO(e^{-c \, \delta^{3/2}\sqrt{r}}), & \mbox{uniformly for } z \in (-a^{2},0) \setminus \mathcal{D}_{-a^{2}}, \\
\frac{J_{R}^{(1)}(z)}{\sqrt{r}}+\bigO(\tfrac{1}{\delta^{7/2}r}), & \mbox{uniformly for } z \in \partial \mathcal{D}_{-a^{2}} \cup \bigcup_{j=n}^{m} \partial \mathcal{D}_{-x_{j}},
	\\
\bigO(\tfrac{1}{\delta^2 \sqrt{r}}), & \mbox{uniformly for } z \in \partial \mathcal{D}_{-a^{2}} \cup \bigcup_{j=n}^{m} \partial \mathcal{D}_{-x_{j}},
\end{cases}
\end{align} \\[-0.5cm]
where $J_{R}^{(1)}$ is given by \eqref{def of JRp1p}. 
In particular, for each $1 \leq p \leq \infty$ and any $k_{1},\ldots,k_{m}\in \mathbb{N}_{\geq 0}$, there exist constants $c>0$ and $C > 0$ such that for all $r \geq 2$ we have
\begin{subequations}\label{estimate Lp norm regime 3} 
\begin{align}
& \| \partial_{\beta}^{k}(J_{R}-I) \|_{L^{p}(\Gamma_{R}\setminus (\partial \mathcal{D}_{-a^{2}} \cup \bigcup_{l=1}^{m} \partial \mathcal{D}_{-x_{l}}))} \leq e^{-c \, \delta^{3/2} \sqrt{r}}, \label{estimate Lp norm 1 regime 3} 
	\\
& \| \partial_{\beta}^{k}(J_{R}-I-\tfrac{J_{R}^{(1)}(z)}{\sqrt{r}}) \|_{L^{p}(\partial \mathcal{D}_{-a^{2}} \cup \bigcup_{l=1}^{m} \partial \mathcal{D}_{-x_{l}})} \leq C \, \delta^{1/p}\frac{\log^{k} r}{\delta^{7/2}r} \label{estimate Lp norm 2 regime 3}
	\\
& \| \partial_{\beta}^{k}(J_{R}-I) \|_{L^{p}(\Gamma_R)} \leq C \, \delta^{1/p}\frac{\log^{k} r}{\delta^{2}\sqrt{r}}, \label{estimate Lp norm 3 regime 3}
\end{align}
\end{subequations}
where $k=k_{1}+\dots+k_{m}$ and $\partial_{\beta}^{k}=\partial_{\beta_{1}}^{k_{1}}\ldots \partial_{\beta_{m}}^{k_{m}}$. By (\ref{calCJRestimate}) and (\ref{estimate Lp norm 3 regime 3}), we see that $I-\mathcal{C}_{J_{R}}$ is invertible and $R$ is given by \eqref{R in terms of muR} provided that $\delta^{2}\sqrt{r}$ is large enough.
In a similar way as in \eqref{eq: asymp R inf s1 neq 0}, we infer from 
\eqref{estimate Lp norm regime 3} that $R$ satisfies
\begin{align}
& \partial_{\beta}^{k}R(z) = \partial_{\beta}^{k}\bigg(I + \frac{R^{(1)}(z)}{\sqrt{r}}\bigg) + \bigO\Big(\frac{\log^{k}r}{\delta^{4}r}\Big), \qquad \partial_{\beta}^{k}R^{(1)}(z) = \bigO\Big(\frac{\log^{k}r}{\delta^{2}}\Big), & & \mbox{ as } r \to  +\infty, \label{eq: asymp R inf s1 neq 0 regime 3} 
\end{align}
uniformly for $z \in \mathbb{C}\setminus \Gamma_{R}$ and with $\vec{x} \in \mathbb{R}_{\mathrm{ord}}^{+,m}$ satisfying $a^{2}<x_{1}$ and \eqref{Airy regime in the main first thm}, where $R^{(1)}$ is given by \eqref{def of R1}.

\section{Large $r$ asymptotics for $\Phi$ with $\alpha \to + \infty$: regime 2}\label{section: regime2}
In this section we analyze $\Phi(rz;r\vec{x},\vec{s},\alpha)$ as $r \to +\infty$ in the regime where $a \to 0$, $\alpha = a\sqrt{r}\to + \infty$ and simultaneously $\vec{x}$ lies in a compact subset of $\mathbb{R}_{\mathrm{ord}}^{+,m}$ and $\vec{s}$ lies in a compact subset of $(0,+\infty)^{m}$. This regime is relevant for part \ref{item 2 in thm} of Theorem \ref{thm:s1 neq 0}.

\medskip In this regime, if one tries to repeat the same construction of $P^{(-a^{2})}$ as in Subsection \ref{subsection: local param bessel s1 neq 0}, one faces several issues. As can be seen from \eqref{expansion of fma2 at ma2}, for $f_{-a^{2}}$ to be a conformal map in $\mathcal{D}_{-a^{2}}$, one needs to choose the radius $\delta$ small compared to $a^{2}$. As a consequence, the expansion \eqref{matching strong 0 s1 neq 0} converges more slowly to $I$ if $a \to 0$ at a slow rate, and the construction completely breaks down if $a \to 0$ at a fast rate. This is a serious obstacle to the study of the regime $a \to 0$, which we circumvent by constructing a local parametrix $P^{(0)}$ in a disk $\mathcal{D}_{0}$ centered at $0$.

\medskip We take the radii of $\mathcal{D}_{-x_{j}}$, $j=1,\ldots,m$ as 
\begin{align}\label{def 2 of mathfrak r}
\delta=\frac{1}{M}\min_{0\leq j <k \leq m} \{x_{k}-x_{j}, x_{1}-a^{2}\}, \qquad M \geq 3, \qquad M \mbox{ fixed},
\end{align}
and we take $\delta_{0}=a^{2}+\frac{x_{1}-a^{2}}{M}$ for the radius of $\mathcal{D}_{0}$. Since the $x_{j}$'s remain bounded away from each other and from $0$, both $\delta$ and $\delta_{0}$ are bounded away from $0$. We choose $M \geq 3$ so that the disks do not intersect each other, and by definition of $\delta_{0}$ we have $-a^{2} \in \mathcal{D}_{0}$. 

%

\medskip The local parametrices around $-x_{j}$, $j=1,\ldots,m$, are identical to those constructed in Section \ref{Section: Steepest descent with s1>0} (with $n=1$), so we do repeat their construction here. 

\subsection{Local parametrix around $0$}\label{subsection: local param bessel s1 neq 0 true}
Our construction is based on the Bessel model RH problem from \cite{KMcLVAV}, whose solution is denoted $\Phi_{\mathrm{Be}}$ and given by \eqref{Phi explicit}. To the best of our knowledge, this is the first time that $\Phi_{\mathrm{Be}}$ is used for large values of $\alpha$. A main difference compared to the case where $\alpha$ is bounded lies in the rather complicated asymptotics \eqref{asymp for Phi Be with large alpha} for $\Phi_{\mathrm{Be}}(z,\alpha)$ as $z \to \infty$ and simultaneously $\alpha \to +\infty$. An important observation is that the function $\xi$ defined in \eqref{def of p and xi} is related to the $g$-function \eqref{def of g s1 neq 0} by the relation
\begin{align}\label{gftheta}
\sqrt{r}g(z) = \sqrt{r}f(z)+ \frac{\pi i \alpha}{2}\theta(z) = \sqrt{rz+\alpha^{2}} + \alpha \log \frac{\sqrt{z}}{a+\sqrt{a^{2}+z}} = \alpha  \xi\Big(\frac{\sqrt{rz}}{\alpha}\Big).
\end{align}
Here, we will simply use $z \mapsto \frac{rz}{4}$ as the conformal map, and we deform the lenses in a neighborhood of $0$ such that
\begin{equation}\label{deformation of the lenses local param 0 Bessel}
(\gamma_{1,+} \cap \mathcal{D}_{0}) \subset -a^{2}+e^{\frac{2\pi i}{3}}(0,+\infty), \qquad (\gamma_{1,-} \cap \mathcal{D}_{0}) \subset -a^{2}+e^{-\frac{2\pi i}{3}}(0,+\infty).
\end{equation}
We seek a local parametrix $P^{(0)}$ of the form (see Figure \ref{bessel0figure})
\begin{align}\label{P0 def local param}
P^{(0)}(z) = E_{0}(z)\Phi_{\mathrm{Be}}\Big(\frac{rz}{4};\alpha\Big)\left\{ \begin{array}{l l}
\begin{pmatrix}
1 & 0 \\
e^{\pi i \alpha} & 1
\end{pmatrix}, & z \in \mathfrak{I}_{+} \\[0.3cm]
\begin{pmatrix}
1 & 0 \\
-e^{-\pi i \alpha} & 1
\end{pmatrix}, & z \in \mathfrak{I}_{-} \\
I, & \text{else}
\end{array} \right\} s_{1}^{-\frac{\sigma_{3}}{2}}e^{-\sqrt{r}f(z)\sigma_{3}}e^{-\frac{\pi i \alpha}{2}\theta(z)\sigma_{3}},
\end{align} 
where $\mathfrak{I}_{\pm} = \{z: \pm \arg z \in (\frac{2\pi}{3},\pi) \mbox{ and } \pm \arg(z+a^{2}) \in (0,\frac{2\pi}{3})\}$, and $E_{0}$ is an analytic matrix-valued function in $\mathcal{D}_{0}$ that will be determined below. It can be verified from the jumps \eqref{Jump for P_Be} of $\Phi_{\mathrm{Be}}$ that $P^{(0)}$ has the same jumps as $S$ inside $\mathcal{D}_{0}$, as desired.

\begin{figure}[t]
\begin{center}
\begin{tikzpicture}
\node at (0.15,-0.18) {$0$};
\node at (-1.5,-0.24) {$-a^2$};

\node at (-1.3,1) {$\mathfrak{I}_{+}$};
\node at (-1.3,-1) {$\mathfrak{I}_{-}$};


\draw[fill] (0,0) circle (0.05);
\draw[fill] (-1.5,0) circle (0.05);
\draw[fill] (-4,0) circle (0.05);
\draw (-4.8,0)--(0,0);
\draw (0,0)--(120:2.2); \draw (0,0)--(-120:2.2);
\draw (-1.5,0)--++(120:2.2); \draw (-1.5,0)--++(-120:2.2);
\end{tikzpicture}
    \caption{\label{bessel0figure} The subsets $\mathfrak{I}_\pm$ of the complex $z$-plane.}
\end{center}
\end{figure}
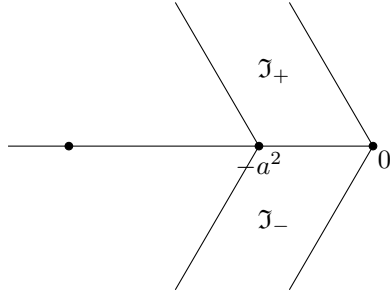

We now choose $E_{0}$ such that the matching condition \eqref{matching weak s1 neq 0} holds. Since $\delta_{0}$ remains bounded away from $0$ as $r \to + \infty$, $\sqrt{rz}\alpha^{-1}\to \infty$ for all $z \in \partial \mathcal{D}_{0}$. Therefore, we can use \eqref{asymp for Phi Be with large alpha} with $z$ replaced by $\frac{rz}{4}$, and we obtain
\begin{align}
P^{(0)}(z)P^{(\infty)}(z)^{-1} = & E_{0}(z)\big( \sqrt{\pi}(\alpha^{2}+rz)^{\frac{1}{4}} \big)^{-\sigma_{3}}N \bigg( I+ \frac{2\Phi_{\mathrm{Be},1}(\frac{rz}{4};\alpha)}{\sqrt{rz}} + \bigO(r^{-1}) \bigg) \nonumber \\
& \times \left\{ \begin{array}{l l}
\begin{pmatrix}
1 & 0 \\
e^{-2\sqrt{r}f(z)} & 1
\end{pmatrix}, & z \in \mathfrak{I}_{+} \\[0.3cm]
\begin{pmatrix}
1 & 0 \\
-e^{-2\sqrt{r}f(z)} & 1
\end{pmatrix}, & z \in \mathfrak{I}_{-} \\
I, & \text{else}
\end{array} \right\} s_{1}^{-\frac{\sigma_{3}}{2}}P^{(\infty)}(z)^{-1} \quad \mbox{ as } r \to + \infty, \label{lol15}
\end{align}
uniformly for $z \in \partial \mathcal{D}_{0}$. On the other hand, we deduce from \eqref{def of g s1 neq 0} that
\begin{align*}
f(z) = \sqrt{z} + o(1), \qquad \mbox{as } a \to 0,
\end{align*}
uniformly for $z$ in compacts subsets of $\mathbb{C}\setminus \{0\}$, which implies
\begin{align}\label{lol16}
\begin{pmatrix}
1 & 0 \\
\pm e^{-2\sqrt{r}f(z)} & 1
\end{pmatrix} = I + \bigO(e^{-c\sqrt{r}}), \qquad \mbox{as } r \to +\infty, a \to 0,
\end{align}
uniformly for $z \in \partial \mathcal{D}_{0}\cap (\mathfrak{I}_{+} \cup \mathfrak{I}_{-})$, for a certain $c>0$. In view of \eqref{lol15} and \eqref{lol16}, we define
\begin{align}\label{expression for E0}
E_{0}(z) = P^{(\infty)}(z) s_{1}^{\frac{\sigma_{3}}{2}}N^{-1} \big( \sqrt{\pi}(\alpha^{2}+rz)^{\frac{1}{4}} \big)^{\sigma_{3}},
\end{align}
and we verify that $E_{0}$ is indeed analytic inside $\mathcal{D}_{0}$, as desired. Furthermore, we have
\begin{equation}\label{matching strong 0 true}
P^{(0)}(z)P^{(\infty)}(z)^{-1} = I + \frac{2}{\sqrt{rz}}P^{(\infty)}(z)s_{1}^{\frac{\sigma_{3}}{2}}\Phi_{\mathrm{Be},1}\Big( \frac{rz}{4};\alpha \Big)s_{1}^{-\frac{\sigma_{3}}{2}}P^{(\infty)}(z)^{-1} + \bigO(r^{-1}),
\end{equation}
as $r \to +\infty$ uniformly for $z \in \partial \mathcal{D}_{0}$. Finally, it is directly seen from \eqref{expression for E0} that
\begin{align}\label{expression for E0 at 0}
E_{0}(0) = P^{(\infty)}(0) s_{1}^{\frac{\sigma_{3}}{2}}N^{-1} \big( \sqrt{\pi a} \, r^{\frac{1}{4}} \big)^{\sigma_{3}}.
\end{align}
\subsection{Small norm problem}\label{subsubsection: a to 0 in small norm}
The function
\begin{equation}\label{def of R s1 neq 0 a to 0}
R(z) := \left\{ \begin{array}{l l}
S(z)P^{(\infty)}(z)^{-1}, & \mbox{for } z \in \mathbb{C}\setminus \big( \mathcal{D}_{0} \cup \bigcup_{j=1}^{m}\mathcal{D}_{-x_{j}} \big), \\
S(z)P^{(-x_{j})}(z)^{-1}, & \mbox{for } z \in \mathcal{D}_{-x_{j}}, \, j \in \{1,\dots,m\}, \\
S(z)P^{(0)}(z)^{-1}, & \mbox{for } z \in \mathcal{D}_{0},
\end{array} \right.
\end{equation}
is analytic in $\mathbb{C}\setminus \Gamma_{R}$, where
\begin{align*}
\Gamma_{R} = \bigg( \partial \mathcal{D}_{0} \cup \bigcup_{j=1}^{m} \partial \mathcal{D}_{-x_{j}} \cup \gamma_{+} \cup \gamma_{-} \bigg) \setminus \bigg(\mathcal{D}_{0} \cup \bigcup_{j=1}^{m} \mathcal{D}_{-x_{j}} \bigg),
\end{align*}
see Figure \ref{figure4}. From \eqref{matching strong -x_j s1 neq 0} and \eqref{matching strong 0 true}, as $r \to + \infty$ the jumps $J_{R}=R_{-}^{-1}R_{+}$ satisfy
\begin{align}\label{asymp for JR with a to 0}
J_{R}(z)-I = \begin{cases}
\bigO(e^{-c \, \sqrt{r}\sqrt{z}}), & \mbox{unif. for } z \in (\gamma_{+}\cup \gamma_{-})\cap \Gamma_{R}, \\
\frac{J_{R}^{(1)}(z)}{\sqrt{r}}+\bigO(\tfrac{1}{r}), & \mbox{unif. for } z \in \partial \mathcal{D}_{0} \cup \bigcup_{j=1}^{m} \partial \mathcal{D}_{-x_{j}},
\end{cases}
\end{align} 
where $J_{R}^{(1)}$ is given by
\begin{align} \label{def of JRp1p a to 0}
J_{R}^{(1)}(z) := \begin{cases} 
\frac{1}{f_{-x_{j}}(z)}E_{-x_{j}}(z) \Phi_{\mathrm{HG},1}(\beta_{j})E_{-x_{j}}(z)^{-1}, & z \in \partial \mathcal{D}_{-x_{j}}, \, j=1,\ldots,m, \\
\frac{2}{\sqrt{z}}P^{(\infty)}(z)s_{1}^{\frac{\sigma_{3}}{2}}\Phi_{\mathrm{Be},1}\big( \frac{rz}{4};\alpha \big)s_{1}^{-\frac{\sigma_{3}}{2}}P^{(\infty)}(z)^{-1}, & z \in \partial\mathcal{D}_{0}.
\end{cases}
\end{align}
As in Subsection \ref{subsection Small norm s1 neq 0}, we obtain that, for $k_{1},\ldots,k_{m}\in \mathbb{N}_{\geq 0}$,
\begin{align}\label{expansion of R as r to inf and a to 0}
& \partial_{\beta}^{k}R(z) = \partial_{\beta}^{k}\bigg(I + \frac{R^{(1)}(z)}{\sqrt{r}}\bigg) + \bigO\Big(\frac{\log^{k}r}{r}\Big), \qquad \partial_{\beta}^{k}R^{(1)}(z) = \bigO(\log^{k}r), & & \mbox{ as } r \to  +\infty, 
\end{align}
uniformly for $z \in \mathbb{C}\setminus \Gamma_{R}$, for $a \to 0$, for $\vec{x}$ in compact subsets of $\mathbb{R}_{\mathrm{ord}}^{+,m}$ and for $\beta_{1},\ldots,\beta_{m}$ in compact subsets of $i \mathbb{R}$, where $k=k_{1}+\ldots+k_{m}$, $\partial_{\beta}^{k}=\partial_{\beta_{1}}^{k_{1}}\ldots \partial_{\beta_{m}}^{k_{m}}$ and
\begin{equation}\label{integral representation of Rp1p as a to 0}
R^{(1)}(z) = \frac{1}{2\pi i}\int_{\partial \mathcal{D}_{0}} \frac{J_{R}^{(1)}(s)}{s-z}ds + \sum_{j=1}^{m}\frac{1}{2\pi i}\int_{\partial\mathcal{D}_{-x_{j}}} \frac{J_{R}^{(1)}(s)}{s-z}ds.
\end{equation}
We deduce from \eqref{def of JRp1p a to 0} and \eqref{def of PhiBe1} that $J_{R}^{(1)}(z)$ can be analytically continued to 
\begin{align*}
\bigg(\mathcal{D}_0 \cup\bigcup_{j=n}^{m}\mathcal{D}_{-x_{j}}\bigg)\setminus \{-a^{2},-x_{1},\ldots,-x_{m}\},
\end{align*}
and has a double pole at $-a^{2}$ and simple poles at $-x_{1},\ldots,-x_{m}$.

In Section \ref{Section: integration s1 >0}, we will need the explicit values of $R^{(1)}(z)$ for $z$ outside the disks, and also for $z=0$. Although the computations involved in the evaluation of the residues at $-a^2$ are quite different from the ones in Subsection \ref{subsection Small norm s1 neq 0}, we find, somewhat remarkably, that the exact same formula \eqref{expression for R^1 s1 neq 0} (with $n = 1$ and the residues given by \eqref{residues at minus a2} and \eqref{residue at minus xj}) holds for $R^{(1)}(z)$ for $z$ outside the disks.
To obtain an explicit expression for $R(0)$, we evaluate the integrals in \eqref{integral representation of Rp1p as a to 0} by means of residue calculations, and we get
\begin{align}\label{expression for R^1 at 0 as a to 0}
R^{(1)}(0) = & \; \frac{1}{a^{2}}\mbox{Res}\big(J_{R}^{(1)}(s),s = -a^{2}\big)+\frac{1}{a^{4}}\mbox{Res}\big((s+a^{2})J_{R}^{(1)}(s),s = -a^{2}\big) \nonumber \\
& +\sum_{j=1}^{m} \frac{1}{x_{j}}\mbox{Res}(J_{R}^{(1)}(s),s = -x_{j}) - J_{R}^{(1)}(0),
\end{align}
with the residues given by \eqref{residues at minus a2}--\eqref{residue at minus xj}) and
\begin{align*}
J_{R}^{(1)}(0) & = \sqrt{r} P^{(\infty)}(0)s_{1}^{\frac{\sigma_{3}}{2}} \begin{pmatrix}
\frac{-3p(0)+p(0)^{3}}{24\alpha} & \frac{i(-p(0)+p(0)^{3})}{4\alpha} \\
\frac{i(-p(0)+p(0)^{3})}{4\alpha} & \frac{3p(0)-p(0)^{3}}{24\alpha}
\end{pmatrix}s_{1}^{-\frac{\sigma_{3}}{2}}P^{(\infty)}(0)^{-1} \\
& = \sqrt{r} P^{(\infty)}(0)s_{1}^{\frac{\sigma_{3}}{2}} \begin{pmatrix}
-\frac{1}{12\alpha} & 0 \\
0 & \frac{1}{12\alpha}
\end{pmatrix}s_{1}^{-\frac{\sigma_{3}}{2}}P^{(\infty)}(0)^{-1} = \frac{1}{12 a^2}\begin{pmatrix} d_1 & i \\ -i(a^2 - d_1^2) & -d_1 \end{pmatrix},
\end{align*}
and $P^{(\infty)}$ is given by \eqref{def of Pinf s1 neq 0} and the function $p$ is given by \eqref{def of p and xi}.
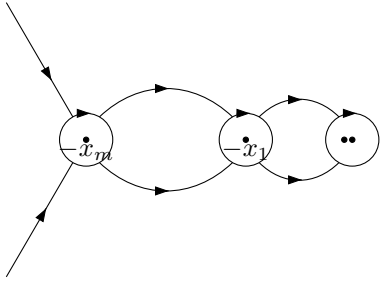
\begin{figure}
\begin{center}
\begin{tikzpicture}[scale=0.7]
\draw[fill] (0,0) circle (0.05);
\draw (0,0) circle (0.5);

\draw (120:0.5) -- (120:3);
\draw (-120:0.5) -- (-120:3);

\draw ($(0,0)+(60:0.5)$) .. controls (1,1.15) and (2,1.15) .. ($(3,0)+(120:0.5)$);
\draw ($(0,0)+(-60:0.5)$) .. controls (1,-1.15) and (2,-1.15) .. ($(3,0)+(-120:0.5)$);
\draw ($(3,0)+(60:0.5)$) .. controls (3.7,0.88) and (4.3,0.88) .. ($(5,0)+(120:0.5)$);
\draw ($(3,0)+(-60:0.5)$) .. controls (3.7,-0.88) and (4.3,-0.88) .. ($(5,0)+(-120:0.5)$);

\draw[fill] (3,0) circle (0.05);
\draw (3,0) circle (0.5);
\draw[fill] (5,0) circle (0.05);
\draw (5,0) circle (0.5);
\draw[fill] (4.85,0) circle (0.05);

\node at (0.,-0.2) {$-x_{m}$};
\node at (3,-0.2) {$-x_{1}$};

\draw[black,arrows={-Triangle[length=0.18cm,width=0.12cm]}]
(-120:1.5) --  ++(60:0.001);
\draw[black,arrows={-Triangle[length=0.18cm,width=0.12cm]}]
(120:1.3) --  ++(-60:0.001);
\draw[black,arrows={-Triangle[length=0.18cm,width=0.12cm]}]
($(0.08,0)+(90:0.5)$) --  ++(0:0.001);

\draw[black,arrows={-Triangle[length=0.18cm,width=0.12cm]}]
($(3.08,0)+(90:0.5)$) --  ++(0:0.001);
\draw[black,arrows={-Triangle[length=0.18cm,width=0.12cm]}]
($(5.08,0)+(90:0.5)$) --  ++(0:0.001);

\draw[black,arrows={-Triangle[length=0.18cm,width=0.12cm]}]
(1.55,0.97) --  ++(0:0.001);
\draw[black,arrows={-Triangle[length=0.18cm,width=0.12cm]}]
(1.55,-0.97) --  ++(0:0.001);

\draw[black,arrows={-Triangle[length=0.18cm,width=0.12cm]}]
(4.05,0.76) --  ++(0:0.001);
\draw[black,arrows={-Triangle[length=0.18cm,width=0.12cm]}]
(4.05,-0.76) --  ++(0:0.001);
\end{tikzpicture}
\end{center}
\caption{The contour $\Gamma_{R}$ with $m=2$ and $a \to 0$. The two unlabeled dots are $-a^{2}$ and $0$.} 
\label{figure4}
\end{figure}
\section{Proof of Theorem \ref{thm:s1 neq 0}}\label{Section: integration s1 >0}
In this section, we prove Theorem \ref{thm:s1 neq 0} in three steps using the following differential identity stated already in \eqref{DIFF identity final form general case}:
\begin{equation}\label{lol17}
\partial_{s_{k}} \log F_{\alpha}(r\vec{x},\vec{s})= K_{\infty} + \sum_{j=1}^{m}K_{-x_{j}} + K_{0}.
\end{equation}
First, we obtain large $r$ asymptotics for $K_{\infty},K_{-x_{1}},$ $\ldots,K_{-x_{m}},K_{0}$ using the analysis of Sections \ref{Section: Steepest descent with s1>0} and \ref{section: regime2}. Second, we substitute these asymptotics into the differential identity \eqref{lol17} and use the change of variables \eqref{def of beta_j s1 neq 0} to get large $r$ asymptotics for $\partial_{\beta_{k}}\log F_{\alpha}(r\vec{x},\vec{s})$, $k=1,\ldots,m$. Third, we integrate these asymptotics in $\beta_{1},\ldots,\beta_{m}$ and obtain large $r$ asymptotics for $\log F_{\alpha}(r\vec{x},\vec{s})$. 

\medskip The three regimes considered in Sections \ref{Section: Steepest descent with s1>0} and \ref{section: regime2} can all be treated at once, save for the analysis of $K_{0}$. A main difference between these regimes is that the terms that are of order $\bigO(r^{-1/2})$ and $\bigO(r^{-1})$ in regimes $1$ and $2$ become of order $\bigO(\delta^{-2}r^{-1/2})$ and $\bigO(\delta^{-4}r^{-1})$ in regime 3. Since $\delta \asymp |x_{m}-a^{2}|$ is of order $1$ in regimes $1$ and $2$, we can include all the factors of $\delta$ in the error terms, and then simply replace $\delta$ by $1$ in regimes $1$ and $2$ at the end. Moreover, we will present all calculations for any value of $n \in \{1,\dots,m+1\}$, but we recall that $n=1$ in regimes $2$ and $3$.

\paragraph{Asymptotics for $K_{\infty}$.} By \eqref{eq:Tasympinf s1 neq 0}, \eqref{eq:Pinf asympinf s1 neq 0} and \eqref{def of R s1 neq 0}, we have $T_{1} = R_{1} + P_{1}^{(\infty)}$, where $R_{1}$ is defined via the expansion $R(z) = I + \frac{R_{1}}{z} + \bigO(z^{-2})$ as $z \to \infty$. Therefore, by \eqref{eq: asymp R inf s1 neq 0}, \eqref{eq: asymp R inf s1 neq 0 regime 3} and \eqref{def of beta_j s1 neq 0}, for any $k \in \{1,\ldots,m\}$, we have 
\begin{align}\label{asymptotics of T1}
& \partial_{s_{k}}T_{1} = \partial_{s_{k}}P_{1}^{(\infty)} + \frac{\partial_{s_{k}}R_{1}^{(1)}}{\sqrt{r}} + \bigO\Big(\frac{\log r}{\delta^{4}r}\Big), \qquad \mbox{as } r \to + \infty,
\end{align}
where $R_{1}^{(1)}$ is such that $R^{(1)}(z) = \frac{R_{1}^{(1)}}{z} + \bigO(z^{-2})$ as $z \to \infty$. Here we have used that the expansions in $r$ and $z$ can be computed in any order, which is a consequence of \eqref{R in terms of muR}.
The matrix $\partial_{s_{k}}R_{1}^{(1)}$ can be calculated explicitly using \eqref{expression for R^1 s1 neq 0}. Recalling \eqref{K inf}, \eqref{eq:Tasympinf s1 neq 0} and \eqref{Pinf 1 12 s1 neq 0}, a computation yields
\begin{align}\nonumber
K_{\infty} & = - \frac{i}{2}\sqrt{r}\partial_{s_{k}}T_{1,12} = -\frac{i}{2}\bigg( \partial_{s_{k}} P_{1,12}^{(\infty)}\sqrt{r} + \partial_{s_{k}}R_{1,12}^{(1)} + \bigO \bigg( \frac{\log r}{\delta^{4}\sqrt{r}} \bigg) \bigg) 
	\\ \label{K inf asymp s1 neq 0}
& = \frac{1}{2}\partial_{s_{k}}d_{1} \sqrt{r} - \sum_{j=n}^{m} \frac{ \partial_{s_{k}}\big( \beta_{j}^{2}(\widetilde{\Lambda}_{j,1}-\widetilde{\Lambda}_{j,2}+2i) \big)}{4i c_{-x_{j}}\sqrt{x_{j}-a^{2}}} + \frac{a^{2}}{4}\partial_{s_{k}}(d_{0}^{2})+ \bigO \bigg( \frac{\log r}{\delta^{4}\sqrt{r}} \bigg).
\end{align}

\paragraph{Asymptotics for $K_{-x_{j}}$ with $j \in \{1,\dots,n-1\}$.}
Suppose $j \in \{1,\dots,n-1\}$. For $z$ outside the lenses such that $z \in \mathcal{D}_{-x_{j}}$ and $\im z > 0$, using \eqref{def of S s1 neq 0}, \eqref{def of R s1 neq 0}, \eqref{def of trivial local param -xj}, \eqref{def of T s1 neq 0}, \eqref{def of Gj} and \eqref{gftheta}, we find
\begin{align*}
T(z) & = R(z)P^{(x_{j})}(z) = R(z)P^{(\infty)}(z)\begin{pmatrix}
1 & h_{-x_{j}}(z) \\
0 & 1
\end{pmatrix} \\
& = \begin{pmatrix}
1 & 0 \\
- \frac{i a^{2}}{2}\sqrt{r} & 1
\end{pmatrix} r^{\frac{\sigma_{3}}{4}}G_{j}(rz;r\vec{x},\vec{s},\alpha) \begin{pmatrix}
1 & \frac{s_{j+1}-s_{j}}{2\pi i} \log (r(z+x_{j})) \\ 0 & 1
\end{pmatrix} e^{-\sqrt{r}f(z)\sigma_{3}}.
\end{align*}
Therefore, using also \eqref{asymp of h -xj}, 
\begin{align*}
G_{j}(-rx_{j};r\vec{x},\vec{s},\alpha) = r^{-\frac{\sigma_{3}}{4}}\begin{pmatrix}
1 & 0 \\
\frac{i a^{2}}{2}\sqrt{r} & 1
\end{pmatrix} R(-x_{j})P^{(\infty)}(-x_{j})e^{\sqrt{r}f(-x_{j})\sigma_{3}}\begin{pmatrix}
1 & \widehat{c}\\ 0 & 1
\end{pmatrix},
\end{align*}
where the exact expression for $\widehat{c}= \widehat{c}(r)$ is unimportant for us. Substituting the above expression into \eqref{K -xj} and using \eqref{eq: asymp R inf s1 neq 0}, we find that
\begin{align}\label{K xj part 2 asymp s1 neq 0}
K_{-x_{j}} = \bigO(e^{-c \sqrt{r}}), \qquad \mbox{as }r \to + \infty,
\end{align}
for some constant $c>0$.
\paragraph{Asymptotics for $K_{-x_{j}}$ with $j \in \{n,\dots,m\}$.} 
Suppose $j \in \{n,\dots,m\}$. For $z$ outside the lenses such that $z \in \mathcal{D}_{-x_{j}}$ and $\im z > 0$, using now \eqref{def of S s1 neq 0}, \eqref{def of R s1 neq 0} and \eqref{lol10}, we have
\begin{equation}\label{lol11}
T(z) = R(z)E_{-x_{j}}(z)\Phi_{\mathrm{HG}}(\sqrt{r}f_{-x_{j}}(z);\beta_{j})(s_{j}s_{j+1})^{-\frac{\sigma_{3}}{4}}e^{-\sqrt{r}f(z)\sigma_{3}}.
\end{equation}
By \eqref{def of beta_j s1 neq 0}, \eqref{expansion conformal map s1 neq 0} and \eqref{precise asymptotics of Phi HG near 0}, as $z \to -x_{j}$ from outside the lenses, $\im z > 0$, we find
\begin{equation}\label{lol 2 s1 neq 0}
\Phi_{\mathrm{HG}}(\sqrt{r}f_{-x_{j}}(z);\beta_{j})(s_{j}s_{j+1})^{-\frac{\sigma_{3}}{4}} = \begin{pmatrix}
\Psi_{j,11} & \Psi_{j,12} \\
\Psi_{j,21} & \Psi_{j,22}
\end{pmatrix} (I + \bigO(z+x_{j})) \begin{pmatrix}
1 & \frac{s_{j+1}-s_{j}}{2\pi i}\log(r(z + x_{j})) \\
0 & 1
\end{pmatrix} ,
\end{equation}
where
\begin{align}
& \Psi_{j,11} = \frac{\Gamma(1-\beta_{j})}{(s_{j}s_{j+1})^{\frac{1}{4}}}, & & \Psi_{j,21} = \frac{\Gamma(1+\beta_{j})}{(s_{j}s_{j+1})^{\frac{1}{4}}}, & & \Psi_{j,11}\Psi_{j,21} = \beta_{j} \frac{2\pi i}{s_{j+1}-s_{j}}. \label{Psi j entries}
\end{align}
The exact values of $\Psi_{j,12}$ and $\Psi_{j,22}$ can also be computed explicitly, but they are unimportant for us. By combining \eqref{lol11}--\eqref{lol 2 s1 neq 0} with \eqref{def of Gj}, \eqref{def of T s1 neq 0} and \eqref{gftheta}, we find
\begin{equation}\label{lol12}
G_{j}(-rx_{j};r\vec{x},\vec{s},\alpha) = r^{-\frac{\sigma_{3}}{4}}\begin{pmatrix}
1 & 0 \\
i \frac{a^{2}}{2}\sqrt{r} & 1
\end{pmatrix} R(-x_{j})E_{-x_{j}}(-x_{j})\begin{pmatrix}
\Psi_{j,11} & \Psi_{j,12} \\ \Psi_{j,21} & \Psi_{j,22}
\end{pmatrix}.
\end{equation}
The change of variables \eqref{def of beta_j s1 neq 0} shows that $\partial_{s_{1}} = -(2\pi i s_1)^{-1} \partial_{\beta_1}$ and $\partial_{s_{k}} = (2\pi i s_k)^{-1}(\partial_{\beta_{k-1}} - \partial_{\beta_k})$ for $k = 2, \dots, m$.
Hence \eqref{eq: asymp R inf s1 neq 0} and \eqref{eq: asymp R inf s1 neq 0 regime 3} imply that, for each $k = 1, \dots, m$,
\begin{align}\label{Ratminusxj}
& R(-x_{j}) = I + \bigO\Big(\frac{1}{\delta^{2} \sqrt{r}}\Big) \quad
 \mbox{and} \quad \partial_{s_k}R(-x_{j}) = \bigO\Big(\frac{\log{r}}{\delta^{2} \sqrt{r}}\Big), & & \mbox{ as } r \to  +\infty.
\end{align}
Substituting \eqref{lol12} into \eqref{K -xj} and using \eqref{E_j at -x_j s1 neq 0}, \eqref{Ratminusxj} as well as the identities $\Gamma(z+1) = z\Gamma(z)$ and $\Gamma(z)\Gamma(1-z) = \pi/\sin(\pi z)$, we find after a long computation that
\begin{align}\nonumber
K_{-x_{j}} = &\; \beta_{j} \partial_{s_{k}} \log \frac{\Gamma(1+\beta_{j})}{\Gamma(1-\beta_{j})} -  2\beta_{j} \partial_{s_{k}} \log \Lambda_{j} 
	\\  
& +  \frac{\partial_{s_{k}}d_{1}}{2\sqrt{x_{j}-a^{2}}}\Big( \beta_{j}^{2}(\widetilde{\Lambda}_{j,1}+\widetilde{\Lambda}_{j,2})+2i\beta_{j} \Big) + \bigO \bigg( \frac{\log r}{\delta^{5/2}\sqrt{r}} \bigg), \qquad \mbox{as } r \to + \infty. \label{K xj part 1 asymp s1 neq 0}
\end{align}

\paragraph{Asymptotics for $K_{0}$ in the regimes where $a$ is bounded away from $0$.} 
For $z$ near $0$, we use \eqref{def of S s1 neq 0}, \eqref{def of R s1 neq 0}, \eqref{def of T s1 neq 0} and \eqref{def of G_0} to obtain
\begin{align*}
T(z) & = R(z)P^{(\infty)}(z) = \begin{pmatrix}
1 & 0 \\
- \frac{i a^{2}}{2}\sqrt{r} & 1
\end{pmatrix} r^{\frac{\sigma_{3}}{4}}G_{0}(rz;r\vec{x},\vec{s},\alpha)(rz)^{\frac{a\sqrt{r}}{2}\sigma_{3}} \begin{pmatrix}
1 & s_{1}h(rz) \\ 0 & 1
\end{pmatrix} e^{-\sqrt{r}g(z)\sigma_{3}},
\end{align*}
from which we get
\begin{align*}
G_{0}(rz;r\vec{x},\vec{s},\alpha) & = r^{-\frac{\sigma_{3}}{4}} \begin{pmatrix}
1 & 0 \\
\frac{i a^{2}}{2}\sqrt{r} & 1
\end{pmatrix} R(z)P^{(\infty)}(z) e^{\sqrt{r}g(z)\sigma_{3}} \begin{pmatrix}
1 & -s_{1}h(rz) \\ 0 & 1
\end{pmatrix}(rz)^{-\frac{a\sqrt{r}}{2}\sigma_{3}} \\
& = r^{-\frac{\sigma_{3}}{4}} \begin{pmatrix}
1 & 0 \\
\frac{i a^{2}}{2}\sqrt{r} & 1
\end{pmatrix} R(z)P^{(\infty)}(z) e^{\sqrt{r}g(z)\sigma_{3}} (rz)^{-\frac{a\sqrt{r}}{2}\sigma_{3}}\begin{pmatrix}
1 & -s_{1}h(rz)(rz)^{a\sqrt{r}} \\ 0 & 1
\end{pmatrix} \\
& = r^{-\frac{\sigma_{3}}{4}} \begin{pmatrix}
1 & 0 \\
\frac{i a^{2}}{2}\sqrt{r} & 1
\end{pmatrix} R(z)P^{(\infty)}(z) \widehat{c}_{2}^{\sigma_{3}}(I+o(1)), \qquad \mbox{as } z \to 0,
\end{align*}
for a certain $\widehat{c}_{2}$ that is independent of $s_{1},\ldots,s_{m}$ and whose exact expression is unimportant for us. Therefore, we have
\begin{align}\label{exact expression for G0}
G_{0}(0;r\vec{x},\vec{s},\alpha) = r^{-\frac{\sigma_{3}}{4}} \begin{pmatrix}
1 & 0 \\
\frac{i a^{2}}{2}\sqrt{r} & 1
\end{pmatrix} R(0)P^{(\infty)}(0) \widehat{c}_{2}^{\sigma_{3}}.
\end{align}
Substituting this expression into \eqref{K 0} and simplifying, we find
\begin{align*}
K_{0} = a\sqrt{r} \Big( H_{0,21} \partial_{s_{k}}H_{0,12} - H_{0,11} \partial_{s_{k}}H_{0,22} - \partial_{s_{k}}\log D(0) \Big),
\end{align*}
where 
\begin{align*}
H_{0} := R(0)P^{(\infty)}(0)D(0)^{\sigma_{3}} = R(0)\begin{pmatrix}
1 & 0 \\ id_{1} & 1
\end{pmatrix}a^{-\frac{\sigma_{3}}{2}}N.
\end{align*}
Recall that $D(0)$ is given by \eqref{explicit value of D0} and $R(0)$ by \eqref{eq: asymp R inf s1 neq 0} (see also \eqref{eq: asymp R inf s1 neq 0 regime 3}) and \eqref{expression for R^1 s1 neq 0} with $z=0$. After a direct computation, we obtain
\begin{align}
& K_{0} = \bigg(\frac{\partial_{s_{k}}d_{1}}{2} + ia^{2} \sum_{j=n}^{m}\mathcal{I}_{j}\partial_{s_{k}}\beta_{j}\bigg)\sqrt{r} + \frac{d_{0}}{2}\partial_{s_{k}}\big( d_{1}-a^{2}d_{0} \big) - \sum_{j=n}^{m}\frac{\beta_{j}^{2}(\widetilde{\Lambda}_{j,1} + \widetilde{\Lambda}_{j,2})\partial_{s_{k}}d_{1}}{2c_{-x_{j}} x_{j}} \nonumber \\
& + \sum_{j=n}^{m} \frac{\partial_{s_{k}} \big( \big[4a^{2}-2x_{j} - ix_{j}(\widetilde{\Lambda}_{j,1}-\widetilde{\Lambda}_{j,2})\big]\beta_{j}^{2} \big)}{4c_{-x_{j}}x_{j}\sqrt{x_{j}-a^{2}}} + \bigO\Big(\frac{\log r}{\delta^{4}\sqrt{r}}\Big), \qquad \mbox{as } r \to + \infty. \label{asymptotics for K0}
\end{align}

\paragraph{Asymptotics for $K_{0}$ in the regime where $a \to 0$ and $x_{1}$ remains bounded away from $0$ as $r \to +\infty$.} 
Let us consider regime 2, where $a \to 0$ while the points $x_j$ stay bounded away from $0$ as $r \to +\infty$. In this regime, the analysis of $K_{0}$ involves the local parametrix $P^{(0)}$ rather than $P^{(\infty)}$. For $z \in \mathcal{D}_{0}$, we use \eqref{def of S s1 neq 0}, \eqref{def of R s1 neq 0 a to 0}, \eqref{def of T s1 neq 0} and \eqref{def of G_0} to conclude that
\begin{align*}
T(z) & = R(z)P^{(0)}(z) = \begin{pmatrix}
1 & 0 \\
- \frac{i a^{2}}{2}\sqrt{r} & 1
\end{pmatrix} r^{\frac{\sigma_{3}}{4}}G_{0}(rz;r\vec{x},\vec{s},\alpha)(rz)^{\frac{a\sqrt{r}}{2}\sigma_{3}} \begin{pmatrix}
1 & s_{1}h(rz) \\ 0 & 1
\end{pmatrix} e^{-\sqrt{r}g(z)\sigma_{3}},
\end{align*}
and in the same way as for \eqref{exact expression for G0}, we find
\begin{align}\label{G0 as a to 0}
G_{0}(0;r\vec{x},\vec{s},\alpha) = r^{-\frac{\sigma_{3}}{4}} \begin{pmatrix}
1 & 0 \\
\frac{i a^{2}}{2}\sqrt{r} & 1
\end{pmatrix} R(0)P^{(0)}(0) \widehat{c}_{2}^{\sigma_{3}}
\end{align}
for a certain $\widehat{c}_{2}$ that is independent of $s_{1},\ldots,s_{m}$ and whose exact expression is unimportant for us. We now find an explicit expression for $P^{(0)}(0)$. As $z \to 0$, $\im z > 0$, $z$ outside the regions $\mathfrak{I}_{\pm}$ defined in \eqref{P0 def local param} (for example, $z \to 0$ with $\arg z = \frac{\pi}{2}$), we use \eqref{P0 def local param} and \eqref{precise matrix at 0 in asymptotics of Phi Bessel near 0} to obtain
\begin{align*}
P^{(0)}(z) & = E_{0}(z)\Phi_{\mathrm{Be}}\Big(\frac{rz}{4};\alpha\Big) s_{1}^{-\frac{\sigma_{3}}{2}}e^{-\sqrt{r}f(z)\sigma_{3}}e^{-\frac{\pi i \alpha}{2}\theta(z)\sigma_{3}} \\
& = E_{0}(z)\begin{pmatrix}
\frac{1}{\Gamma(1+\alpha)} & \frac{i \Gamma(\alpha)}{2\pi} \\
\frac{i \pi}{\Gamma(\alpha)} & \frac{\Gamma(1+\alpha)}{2}
\end{pmatrix} \Big(\frac{rz}{4}\Big)^{\frac{\alpha}{2}\sigma_{3}} \begin{pmatrix}
1 & h(\frac{rz}{4}) \\ 0 & 1
\end{pmatrix} s_{1}^{-\frac{\sigma_{3}}{2}}e^{-\sqrt{r}g(z)\sigma_{3}} \\
& = E_{0}(0)\begin{pmatrix}
\frac{1}{\Gamma(1+\alpha)} & \frac{i \Gamma(\alpha)}{2\pi} \\
\frac{i \pi}{\Gamma(\alpha)} & \frac{\Gamma(1+\alpha)}{2}
\end{pmatrix} s_{1}^{-\frac{\sigma_{3}}{2}} \widehat{c}_{3}^{\sigma_{3}}(I+o(1)),
\end{align*}
where here too, $\widehat{c}_{3}$ is a function independent of $s_{1},\ldots,s_{m}$ whose exact expression is unimportant for us. Hence, using \eqref{expression for E0 at 0}, we find
\begin{align}\nonumber
P^{(0)}(0) & = E_{0}(0)\begin{pmatrix}
\frac{1}{\Gamma(1+\alpha)} & \frac{i \Gamma(\alpha)}{2\pi} \\
\frac{i \pi}{\Gamma(\alpha)} & \frac{\Gamma(1+\alpha)}{2}
\end{pmatrix} s_{1}^{-\frac{\sigma_{3}}{2}} \widehat{c}_{3}^{\sigma_{3}} 
	\\ \label{P00E00Pinfty0}
& = P^{(\infty)}(0) s_{1}^{\frac{\sigma_{3}}{2}}N^{-1} \big( \sqrt{\pi a} \, r^{\frac{1}{4}} \big)^{\sigma_{3}}\begin{pmatrix}
\frac{1}{\Gamma(1+\alpha)} & \frac{i \Gamma(\alpha)}{2\pi} \\
\frac{i \pi}{\Gamma(\alpha)} & \frac{\Gamma(1+\alpha)}{2}
\end{pmatrix} s_{1}^{-\frac{\sigma_{3}}{2}} \widehat{c}_{3}^{\sigma_{3}}.
\end{align}
Defining $H_0$ by $H_{0} := R(0)P^{(\infty)}(0)$, equations \eqref{G0 as a to 0} and \eqref{P00E00Pinfty0} imply that
$$G_{0}(0;r\vec{x},\vec{s},\alpha) = r^{-\frac{\sigma_{3}}{4}} \begin{pmatrix}
1 & 0 \\
\frac{i a^{2}}{2}\sqrt{r} & 1
\end{pmatrix} H_{0} s_{1}^{\frac{\sigma_{3}}{2}}N^{-1} \big( \sqrt{\pi a} \, r^{\frac{1}{4}} \big)^{\sigma_{3}}\begin{pmatrix}
\frac{1}{\Gamma(1+\alpha)} & \frac{i \Gamma(\alpha)}{2\pi} \\
\frac{i \pi}{\Gamma(\alpha)} & \frac{\Gamma(1+\alpha)}{2}
\end{pmatrix} s_{1}^{-\frac{\sigma_{3}}{2}} \widehat{c}_{3}^{\sigma_{3}}\widehat{c}_{2}^{\sigma_{3}}.$$
Substituting this expression for $G_0(0;r\vec{x},\vec{s},\alpha)$ into \eqref{K 0} and using the identities $\Gamma(\alpha+1) = \alpha \Gamma(\alpha)$ and $\det H_{0}\equiv 1$, we obtain
\begin{align}\label{K0 in terms of H with a to 0}
K_{0} = a\sqrt{r} \Big( H_{0,21} \partial_{s_{k}}H_{0,12} - H_{0,11} \partial_{s_{k}}H_{0,22} \Big).
\end{align}
The fact that $P^{(\infty)}(0) = \bigO(a^{-\frac{1}{2}})$ as $a \to 0$ could a priori worsen the error in the large $r$ asymptotics of $K_{0}$. However, the following exact formula, which is obtained by substituting the explicit expression \eqref{def of Pinf s1 neq 0} for $P^{(\infty)}(0)$ into \eqref{K0 in terms of H with a to 0} with $H_{0} = R(0)P^{(\infty)}(0)$, shows that this is not the case:
\begin{align*}
K_{0} = &\; \frac{i \sqrt{r}}{2} \bigg( 2i a \partial_{s_{k}} \log D(0) - i \partial_{s_{k}}d_{1} + (a^{2} - d_{1}^{2}) \Big( R(0)_{22}\partial_{s_{k}}R(0)_{12} - R(0)_{12} \partial_{s_{k}}R(0)_{22} \Big) \\
& + \Big( R(0)_{21}\partial_{s_{k}}R(0)_{11} - R(0)_{11} \partial_{s_{k}}R(0)_{21} \Big) + 2id_{1} \Big( R(0)_{21}\partial_{s_{k}}R(0)_{12} - R(0)_{11} \partial_{s_{k}}R(0)_{22} \Big) \bigg).
\end{align*}
The right-hand side of the above equation can now be expanded as $r \to + \infty$ using \eqref{expansion of R as r to inf and a to 0} and the expressions \eqref{explicit value of D0} for $D(0)$ and \eqref{expression for R^1 at 0 as a to 0} for $R^{(1)}(0)$. A calculation shows that the term $ J_{R}^{(1)}(0)$ in \eqref{expression for R^1 at 0 as a to 0} leads to a contribution to $K_0$ of $\bigO(\log(r)/\sqrt{r})$ and in the end we arrive at the following asymptotic formula which is identical to \eqref{asymptotics for K0}:
\begin{align}
 K_{0} = &\; \bigg(\frac{\partial_{s_{k}}d_{1}}{2} + ia^{2} \sum_{j=n}^{m}\mathcal{I}_{j}\partial_{s_{k}}\beta_{j}\bigg)\sqrt{r} + \frac{d_{0}}{2}\partial_{s_{k}}\big( d_{1}-a^{2}d_{0} \big) - \sum_{j=n}^{m}\frac{\beta_{j}^{2}(\widetilde{\Lambda}_{j,1} + \widetilde{\Lambda}_{j,2})\partial_{s_{k}}d_{1}}{2c_{-x_{j}} x_{j}} \nonumber \\
& + \sum_{j=n}^{m} \frac{\partial_{s_{k}} \big( \big[4a^{2}-2x_{j} - ix_{j}(\widetilde{\Lambda}_{j,1}-\widetilde{\Lambda}_{j,2})\big]\beta_{j}^{2} \big)}{4c_{-x_{j}}x_{j}\sqrt{x_{j}-a^{2}}} + \bigO\Big(\frac{\log r}{\delta^{4}\sqrt{r}}\Big), \qquad \mbox{as } r \to + \infty. \label{asymptotics for K0 as a to 0}
\end{align}

\paragraph{Asymptotics for $\partial_{\beta_{k}} \log F_{\alpha}(r\vec{x},\vec{s})$.} Recall that $c_{-x_{j}}$ is given by \eqref{expansion conformal map s1 neq 0}, and $d_{0}$ and $d_{1}$ by \eqref{d_ell in terms of beta_j s1 neq 0}. Hence, by substituting into \eqref{DIFF identity final form general case} the large $r$ asymptotics of $K_{\infty}$, $K_{-x_{j}}$, $j=1,\dots,m$ and $K_{0}$ given by \eqref{K inf asymp s1 neq 0}, \eqref{K xj part 2 asymp s1 neq 0}, \eqref{K xj part 1 asymp s1 neq 0} and \eqref{asymptotics for K0}, we find (after simplifications), for $k = 1, \dots, m$,
\begin{multline}\label{lol5 s1 neq 0}
\partial_{s_k}\log F_{\alpha}(r\vec{x},\vec{s}) = \bigg(\partial_{s_{k}}d_{1}+ ia^{2} \sum_{j=n}^{m}\mathcal{I}_{j}\partial_{s_{k}}\beta_{j}\bigg)\sqrt{r}  - \sum_{j=n}^{m} \Big( 2\beta_{j} \partial_{s_{k}} \log \Lambda_{j} + \partial_{s_{k}} (\beta_{j}^{2}) \Big) \\ + \sum_{j=n}^{m}\beta_{j} \partial_{s_{k}} \log \frac{\Gamma(1+\beta_{j})}{\Gamma(1-\beta_{j})} + \bigO\Big( \frac{\log r}{\delta^{4}\sqrt{r}} \Big), \qquad \mbox{as } r \to + \infty.
\end{multline}
Using \eqref{def Lambda_j s1 neq 0}, we also note that
\begin{equation}\label{lol4 s1 neq 0}
-\sum_{j=n}^{m} 2\beta_{j} \partial_{s_{k}} \log \Lambda_{j} = -2 \sum_{j=n}^{m} \beta_{j} \partial_{s_{k}}(\beta_{j}) \log (4(x_{j}-a^{2})c_{-x_{j}}\sqrt{r}) -2\sum_{j=n}^{m} \beta_{j} \sum_{\substack{\ell = n \\ \ell \neq j}}^{m} \partial_{s_{k}}(\beta_{\ell})\log(T_{\ell,j}).
\end{equation}
Since $\beta_{n},\ldots,\beta_{m}$ are independent of $s_{1},\ldots,s_{n-1}$, see \eqref{def of beta_j s1 neq 0}, $\partial_{s_k}\log F_{\alpha}(r\vec{x},\vec{s}) = \bigO(\log(r)/(\delta^{4}\sqrt{r}))$ as $r \to + \infty$ for each $k = 1,\ldots,n-1$.
Let us define $\widetilde{F}_{\alpha}(r \vec{x}, \vec{\beta}) = F_{\alpha}(r \vec{x},\vec{s})$, where $\vec{\beta} = (\beta_{1},\dots,\beta_{m})$.
By \eqref{def of beta_j s1 neq 0}, we have $s_k= \exp(-2\pi i \sum_{j=k}^m \beta_j)$ and hence $\partial_{\beta_k} = -2\pi i \sum_{j=1}^k s_j \partial_{s_j}$ for $k = 1, \dots, m$.
In particular, for each $k = 1,\ldots,n-1$,
\begin{align}\label{first part trivial diff identity}
\partial_{\beta_k}\log \widetilde{F}_{\alpha}(r\vec{x},\vec{\beta}) = \bigO\Big( \frac{\log r}{\delta^{4}\sqrt{r}} \Big), \qquad \mbox{as } r \to + \infty.
\end{align}
Substituting \eqref{lol4 s1 neq 0} into \eqref{lol5 s1 neq 0}, for $k \in \{n,\ldots,m\}$ we obtain
\begin{multline}\label{lol5 part 2 s1 neq 0}
\partial_{\beta_k}\log\widetilde{F}_{\alpha}(r \vec{x}, \vec{\beta}) = (\partial_{\beta_{k}}d_{1}+ ia^{2} \mathcal{I}_{k}) \sqrt{r}  -2 \sum_{j=n}^{m} \beta_{j} \partial_{\beta_{k}}(\beta_{j}) \log \Big(4 (x_{j}-a^{2})^{3/2} x_{j}^{-1} \sqrt{r}\Big) \\ -2\sum_{j=n}^{m} \beta_{j} \sum_{\substack{\ell = n \\ \ell \neq j}}^{m} \partial_{\beta_{k}}(\beta_{\ell})\log(T_{\ell,j}) - \sum_{j=n}^{m} \partial_{\beta_{k}}(\beta_{j}^{2}) + \sum_{j=n}^{m} \beta_{j} \partial_{\beta_{k}} \log \frac{\Gamma(1+\beta_{j})}{\Gamma(1-\beta_{j})} + \bigO\Big(\frac{\log r}{\delta^{4}\sqrt{r}}\Big)
\end{multline}
as $r \to + \infty$. Recalling that $d_{1}$ is given by \eqref{d_ell in terms of beta_j s1 neq 0}, $\mathcal{I}_{k}$ by \eqref{T and T tilde s1 neq 0}, and $f$ by \eqref{def of g s1 neq 0}, we note that
\begin{align*}
\partial_{\beta_{k}}d_{1}+ ia^{2} \mathcal{I}_{k} = -2i \sqrt{x_{k}-a^{2}} + ia^{2} \int_{a^{2}}^{x_{k}}\frac{du}{u\sqrt{u-a^{2}}} = -i \int_{a^{2}}^{x_{k}}\frac{\sqrt{u-a^{2}}}{u}du = 2 f_{-}(-x_{k}).
\end{align*}
Hence, for $k \in \{n,\ldots,m\}$, the asymptotics \eqref{lol5 part 2 s1 neq 0} can be rewritten as
\begin{align}\nonumber
\partial_{\beta_k}\log \widetilde{F}_{\alpha}(r \vec{x}, \vec{\beta}) =&\; 2f_{-}(-x_{k})\sqrt{r} -2 \beta_{k} \log \Big(4 (x_{k}-a^{2})^{3/2} x_{k}^{-1} \sqrt{r}\Big) 
	\\ 
& - 2 \sum_{\substack{j=n \\ j \neq k}}^{m}\beta_{j} \log(T_{k,j}) - 2\beta_{k} + \beta_{k} \partial_{\beta_{k}} \log \frac{\Gamma(1+\beta_{k})}{\Gamma(1-\beta_{k})} + \bigO\Big( \frac{\log r}{\delta^{4}\sqrt{r}} \Big). \label{DIFF IDENTITY s1 neq 0}
\end{align}
\paragraph{Asymptotics for $\log F_{\alpha}(r\vec{x},\vec{s})$.} Integrating \eqref{first part trivial diff identity} in $\beta_{1},\ldots,\beta_{n-1}$, we get
\begin{align}\label{first trivial ratio asympt}
\frac{\log \widetilde{F}_{\alpha}(r\vec{x},(\beta_{1},\ldots,\beta_{n-1},0,\ldots,0))}{\log \widetilde{F}_{\alpha}(r\vec{x},(0,\ldots,0))} = \bigO\Big( \frac{\log r}{\delta^{4}\sqrt{r}} \Big), \qquad \mbox{as } r \to + \infty.
\end{align}
We will now proceed with the successive integrations of \eqref{DIFF IDENTITY s1 neq 0} in $\beta_{n},\ldots,\beta_{m}$. We will use the notation $\vec{\beta}_{j} := (\beta_{1},\dots,\beta_{j}, 0,\ldots,0)$, $j \geq n-1$. First, we set $\beta_{n+1}=\ldots=\beta_{m}=0$ and use \eqref{DIFF IDENTITY s1 neq 0} with $k=n$. After integration in $\beta_{n}$, we get
\begin{align*}
\log \frac{\widetilde{F}_{\alpha}(r\vec{x},\vec{\beta}_{n})}{\widetilde{F}_{\alpha}(r\vec{x},\vec{\beta}_{n-1})} & = 2\beta_{n} f_{-}(-x_{n}) \sqrt{r} - \beta_{n}^{2} \log \big(4 (x_{n}-a^{2})^{3/2} x_{n}^{-1} \sqrt{r}\big) \\
&  +\log\big(G(1+\beta_{n})G(1-\beta_{n})\big) + \bigO \bigg( \frac{\log r}{\delta^{4}\sqrt{r}} \bigg), \qquad \mbox{as } r \to + \infty,
\end{align*}
where $G$ is Barnes' $G$-function, and where we have used that
\begin{equation}\label{integral of Gamma with Barnes}
\int_{0}^{\beta} x \partial_{x} \log \frac{\Gamma(1+x)}{\Gamma(1-x)}dx = \beta^{2} + \log \big(G(1+\beta)G(1-\beta)\big).
\end{equation}
Now, we set $\beta_{n+2}=\dots=\beta_{m} = 0$ and $k = n+1$ in \eqref{DIFF IDENTITY s1 neq 0}, and integrate in $\beta_{n+1}$. This gives
\begin{multline}
\log  \frac{\widetilde{F}_{\alpha}(r\vec{x},\vec{\beta}_{n+1})}{\widetilde{F}_{\alpha}(r\vec{x},\vec{\beta}_{n})} = 2\beta_{n+1} f_{-}(-x_{n+1}) \sqrt{r} - \beta_{n+1}^{2} \log \big(4 (x_{n+1}-a^{2})^{3/2} x_{n+1}^{-1} \sqrt{r}\big) \\ - 2 \beta_{n}\beta_{n+1} \log(T_{n+1,n}) +\log\big(G(1+\beta_{n+1})G(1-\beta_{n+1})\big) + \bigO \bigg( \frac{\log r}{\delta^{4}\sqrt{r}} \bigg)
\end{multline}
as $r \to + \infty$. The integrations in $\beta_{n+2},\dots,\beta_{m}$ are similar. By adding the asymptotics of $\log  \frac{\widetilde{F}_{\alpha}(r\vec{x},\vec{\beta}_{j})}{\widetilde{F}_{\alpha}(r\vec{x},\vec{\beta}_{j-1})}$ for $j=n,\ldots,m$ to \eqref{first trivial ratio asympt}, we obtain
\begin{multline}\label{asymp Ftilde ratio final}
\log \frac{\widetilde{F}_{\alpha}(r\vec{x},\vec{\beta})}{\widetilde{F}_{\alpha}(r\vec{x},\vec{0})} = \sum_{j=n}^{m} 2 \beta_{j} f_{-}(-x_{j}) \sqrt{r} - \sum_{j=n}^{m} \beta_{j}^{2} \log \big(4 (x_{j}-a^{2})^{3/2} x_{j}^{-1} \sqrt{r}\big) \\ - 2 \sum_{n \leq j < k \leq m} \beta_{j}\beta_{k} \log(T_{k,j}) + \sum_{j=n}^{m} \log\big(G(1+\beta_{j})G(1-\beta_{j})\big) + \bigO \bigg( \frac{\log r}{\delta^{4}\sqrt{r}} \bigg).
\end{multline}
Since $u_{j}=-2\pi i \beta_{j}$, $\widetilde{F}_{\alpha}(r\vec{x},\vec{\beta}) = F_{\alpha}(r \vec{x},\vec{s}) = E_{\alpha}(r\vec{x},\vec{u})$, and
\begin{align*}
\widetilde{F}_{\alpha}(r\vec{x},\vec{0}) = E_{\alpha}(r\vec{x},\vec{0}) = 1,
\end{align*}
we can rewrite \eqref{asymp Ftilde ratio final} as
\begin{multline*}
\log E_{\alpha}(r\vec{x},\vec{u}) =  \sum_{j=n}^{m} \bigg(\frac{\sqrt{r}}{\pi} \int_{a^{2}}^{x_{j}}\frac{\sqrt{u-a^{2}}}{2u}du\bigg)u_{j} + \sum_{j=n}^{m} \frac{u_{j}^{2}}{4\pi^{2}} \log \big(4 (x_{j}-a^{2})^{3/2} x_{j}^{-1} \sqrt{r}\big) \\ + \sum_{n \leq j < k \leq m} \frac{u_{j}u_{k}}{2\pi^{2}} \log\left(\frac{\sqrt{x_{j}-a^{2}}+\sqrt{x_{k}-a^{2}}}{|\sqrt{x_{j}-a^{2}}-\sqrt{x_{k}-a^{2}}|}\right) + \sum_{j=n}^{m} \log \big(G(1+\beta_{j})G(1-\beta_{j})\big) + \bigO \bigg( \frac{\log r}{\delta^{4}\sqrt{r}} \bigg),
\end{multline*}
which finishes the proof of Theorem \ref{thm:s1 neq 0}.
\section{Proof of Theorem \ref{thm:rigidity}}\label{Section: rigidity}
The proof is an adaptation of \cite[Theorem 1.2]{ChCl4} to handle varying point processes. Let $\{X_{r}\}_{r \geq 0}$ be a family of point processes satisfying Assumptions \ref{assumptions}, and let $\mathrm{N}_{r}(x)$ denote the random variable that counts the number of points of $X_{r}$ that are $\leq x$. In what follows, we let $\mathfrak{a}, \eta_{r,1},\eta_{r,2},\delta_{r}$ be the constants appearing in Assumptions \ref{assumptions}. We also write $\upxi_{r,k}$ for the $k$-th smallest point of $X_{r}$, and let $\kappa_{r,k}:= \upmu_{r}^{-1}(k)$. We divide the proof into two lemmas.

\begin{lemma}\label{lemma: A r eps}
There exist $c>0$ and $r_{0}>0$ such that for any $\epsilon>0$ sufficiently small and $r>r_{0}$,
\begin{align}\label{prob statement lemma 2.1}
\mathbb P\left(\sup_{x\in ((\eta_{r,1}+\delta_{r}) r,(\eta_{r,2}-\delta_{r}) r)} \left|\frac{\mathrm{N}_{r}(x)-\upmu_{r}(x)}{\upsigma_{r}^2(x)}\right| >  \sqrt{\frac{2}{\mathfrak{a}}(1+\epsilon)} \right)\leq \frac{c\, \upmu_{r}((\eta_{r,1}+\delta_{r}) r)^{-\epsilon}}{2\epsilon}.
\end{align}
In particular, for any $\epsilon > 0$, 
\begin{align*}
\lim_{r \to + \infty} \mathbb{P}\left(\sup_{x\in ((\eta_{r,1}+\delta_{r}) r,(\eta_{r,2}-\delta_{r}) r)} \left|\frac{\mathrm{N}_{r}(x)-\upmu_{r}(x)}{\upsigma_{r}^2(x)}\right| \leq   \sqrt{\frac{2}{\mathfrak{a}}(1+\epsilon)}\right) = 1.
\end{align*}
\end{lemma}
\begin{proof}
For each large enough $r$, $\upmu_{r}$ and $\upsigma_{r}$ are increasing by part (2) of Assumptions \ref{assumptions}, and therefore
\begin{align}\label{fracNrxmur}
\frac{\mathrm{N}_{r}(x)-\upmu_{r}(x)}{\upsigma_{r}^2(x)} \leq \frac{\mathrm{N}_{r}(\kappa_{r,k})-\upmu_{r}(\kappa_{r,k-1})}{\upsigma_{r}^2(\kappa_{r,k-1})} = \frac{\mathrm{N}_{r}(\kappa_{r,k})-\upmu_{r}(\kappa_{r,k})+1}{\upsigma_{r}^2(\kappa_{r,k-1})} , \qquad \mbox{for all } x \in [\kappa_{r,k-1},\kappa_{r,k}]
\end{align}
and for all $k\in \mathcal{K}_{r}:=\{k \in \mathbb{N}_{> 0}: \kappa_{r,k}>(\eta_{r,1}+\delta_{r})r \mbox{ and } \kappa_{r,k-1}<(\eta_{r,2}-\delta_{r}) r\}$. The definition of $\mathcal{K}_{r}$ implies in particular that
\begin{align*}
((\eta_{r,1}+\delta_{r}) r,(\eta_{r,2}-\delta_{r}) r) \subset \bigcup_{k\in \mathcal{K}_{r}} [\kappa_{r,k-1},\kappa_{r,k}].
\end{align*}
Note that $\mathcal{K}_{r}$ is finite and $\#\mathcal{K}_{r}\to + \infty$ as $r \to + \infty$ (this follows directly from Remark \ref{remark:big lebesgue measure}). Using first \eqref{fracNrxmur}, then a union bound and then Markov's inequality, we find
\begin{align}
& \mathbb P\left(\sup_{x\in ((\eta_{r,1}+\delta_{r}) r,(\eta_{r,2}-\delta_{r}) r)}\frac{\mathrm{N}_{r}(x)-\upmu_{r}(x)}{\upsigma_{r}^2(x)}>\gamma\right) \leq \mathbb P\left(\max_{k\in \mathcal{K}_{r}}\frac{\mathrm{N}_{r}(\kappa_{r,k})-\upmu_{r}(\kappa_{r,k})+1}{\upsigma_{r}^2(\kappa_{r,k-1})}>\gamma\right) \nonumber \\
& \leq \sum_{k \in \mathcal{K}_{r}} \mathbb P\left(\frac{\mathrm{N}_{r}(\kappa_{r,k})-\upmu_{r}(\kappa_{r,k})+1}{\upsigma_{r}^2(\kappa_{r,k-1})}>\gamma\right) \leq \sum_{k \in \mathcal{K}_{r}}\mathbb{E}\left[e^{\gamma \mathrm{N}_{r}(\kappa_{r,k})}\right]e^{-\gamma\upmu_{r}(\kappa_{r,k})+\gamma-\gamma^2 \upsigma_{r}^2(\kappa_{r,k-1})}, \label{Markov inequality}
\end{align}
for any $\gamma>0$. Let $k':=\max\{k:k\in\mathcal{K}_{r}\}$ and $k'':=\min\{k:k\in\mathcal{K}_{r}\}$ (note that $k'$ and $k''$ depends on $r$, even though this is not indicated in the notation). For all $k$ such that $k''\leq k<k'$, by definition of $\mathcal{K}_{r}$ we can directly use \eqref{expmomentbound} to obtain an upper bound for $\mathbb{E}[e^{\gamma \mathrm{N}_{r}(\kappa_{r,k})}]$. To also obtain an upper bound for $\mathbb{E}[e^{\gamma \mathrm{N}_{r}(\kappa_{r,k'})}]$ using \eqref{expmomentbound}, we need to show that the inequality $\kappa_{r,k'}<(\eta_{r,2}-\frac{\delta_{r}}{2})r$ holds for all sufficiently large $r$. Let us write $\kappa_{r,k'} = \kappa_{r,k'-1}+m$. We have
\begin{align}
k' & = \upmu_{r}(\upmu_{r}^{-1}(k'-1)+m) \geq k'-1 + m \inf_{\xi \in [\kappa_{r,k'-1},\kappa_{r,k'}]}  \upmu_{r}'(\xi) \nonumber \\
& \geq k'-1 + \frac{m}{\kappa_{r,k'-1}+m}\inf_{\xi \in [\kappa_{r,k'-1},\kappa_{r,k'}]} \xi\upmu_{r}'(\xi). \label{kprimeupmu}
\end{align}
Note that $\kappa_{r,k'-1} \geq (\eta_{r,1} +\delta_r)r$ for all sufficiently large $r$, because $\#\mathcal{K}_r \to +\infty$. Thus, since $\xi \mapsto \xi\upmu_{r}'(\xi)$ is non-decreasing by Assumptions \ref{assumptions}, 
\begin{align}\label{infxikappa}
\inf_{\xi \in [\kappa_{r,k'-1},\kappa_{r,k'}]} \xi\upmu_{r}'(\xi) = \kappa_{r,k'-1}\upmu_{r}'(\kappa_{r,k'-1}) \geq (\eta_{r,1}+\delta_{r})r \upmu_{r}'((\eta_{r,1}+\delta_{r})r)
\end{align}
for all large enough $r$. Also, by definition of $\mathcal{K}_{r}$, $\kappa_{r,k'-1} \leq (\eta_{r,2}-\delta_{r})r$. Hence, \eqref{kprimeupmu}--\eqref{infxikappa} imply
\begin{align}\label{lol13}
\frac{m}{\delta_{r}r} \leq \frac{1}{\delta_{r}} \frac{\eta_{r,2}-\delta_{r}}{(\eta_{r,1}+\delta_{r})r\upmu_{r}'((\eta_{r,1}+\delta_{r})r)-1} \leq \frac{2(\eta_{r,2}-\delta_{r})}{(\eta_{r,1}+\delta_{r})\delta_{r}r\upmu_{r}'((\eta_{r,1}+\delta_{r})r)}
\end{align}
for all sufficiently large $r$, and the right-hand side tends to $0$ as $r \to +\infty$ by parts (2) and (4) of Assumptions \ref{assumptions}. This shows that $\kappa_{r,k'}<(\eta_{r,2}-\frac{\delta_{r}}{2})r$ for all sufficiently large $r$. Therefore, using \eqref{expmomentbound} in \eqref{Markov inequality}, we obtain
\begin{align*}
\mathbb P\left(\sup_{x\in ((\eta_{r,1}+\delta_{r}) r,(\eta_{r,2}-\delta_{r}) r)}\frac{\mathrm{N}_{r}(x)-\upmu_{r}(x)}{\upsigma_{r}^2(x)}>\gamma\right)\leq 
\mathrm{C}
e^\gamma\sum_{k \in \mathcal{K}_{r}} e^{-\frac{\gamma^2}{2} \upsigma_{r}^2(\kappa_{r,k})}e^{\gamma^2(\upsigma_{r}^{2}(\kappa_{r,k})-\upsigma_{r}^{2}(\kappa_{r,k-1}))}.
\end{align*}
Using the assumption that $\upsigma_{r}^{2} \circ \upmu_{r}^{-1}$ is concave and increasing, we have
\begin{align}\label{lol14}
\sup_{k \in \mathcal{K}_{r}} e^{\gamma^2(\upsigma_{r}^{2}(\kappa_{r,k})-\upsigma_{r}^{2}(\kappa_{r,k-1}))} = e^{\gamma^2(\upsigma_{r}^{2}\circ \upmu_{r}^{-1}(k'')-\upsigma_{r}^{2}\circ \upmu_{r}^{-1}(k''-1))} \leq e^{\gamma^2(\upsigma_{r}^{2}\circ \upmu_{r}^{-1})'(k''-1)}.
\end{align}
On the other hand, for any $k_{1},k_{2} \in [\upmu_{r}(\eta_{r,1}r),\upmu_{r}(\eta_{r,2}r)]$ with $k_{1} \leq k_{2}$, 
\begin{align}\label{lol5}
(\upsigma_{r}^{2} \circ \upmu_{r}^{-1})'(k_{2}) \leq  \frac{(\upsigma_{r}^{2} \circ \upmu_{r}^{-1})(k_{2})- (\upsigma_{r}^{2} \circ \upmu_{r}^{-1})(k_{1})}{k_{2}-k_{1}}  \leq (\upsigma_{r}^{2} \circ \upmu_{r}^{-1})'(k_{1}).
\end{align}
Moreover, for any $b_{1},b_{2} \in [\eta_{r,1},\eta_{r,2}]$ with $b_{2}>b_{1}$, 
\begin{align}\label{b2b1 bounds}
\upmu_{r}(b_{2}r)-\upmu_{r}(b_{1}r) \geq (b_{2}-b_{1})r \inf_{x \in [b_{1}r,b_{2}r]} \upmu_{r}'(x) \geq \Big( 1-\frac{b_{1}}{b_{2}} \Big)\inf_{x\in[b_{1}r,b_{2}r]}x\upmu_{r}'(x) = \Big( 1-\frac{b_{1}}{b_{2}} \Big) b_{1}r \upmu_{r}'(b_{1}r).
\end{align}
If $b_{1},b_{2} \in (\eta_{r,1}+\frac{\delta_{r}}{2},\eta_{r,2}-\frac{\delta_{r}}{2})$ are chosen such that $1-\frac{b_{1}}{b_{2}} \geq c_{1} \delta_{r}$ for a certain $c_{1}>0$ independent of $r$, then the right-hand side of \eqref{b2b1 bounds} converges to $+\infty$ as $r\to + \infty$ by part (2) of Assumptions \ref{assumptions}.
The definition of $\mathcal{K}_{r}$ implies that $k''-1 \in (\upmu_{r}((\eta_{r,1}+\delta_{r})r)-1,\upmu_{r}((\eta_{r,1}+\delta_{r})r)]$, and therefore \eqref{b2b1 bounds} with $b_1 = (\eta_{r,1}+\frac{3\delta_{r}}{4})r$ and $b_2 = (\eta_{r,1}+\delta_{r})r$ implies $k''-1 \in (\upmu_{r}((\eta_{r,1}+\frac{3\delta_{r}}{4})r),\upmu_{r}((\eta_{r,1}+\delta_{r})r)]$ for all sufficiently large $r$. Hence, applying the first inequality in \eqref{lol5} with $k_{2}=k''-1$ and $k_{1}= \upmu_{r}((\eta_{r,1}+\frac{\delta_{r}}{2})r)$, and using parts (2) and (4) of Assumptions \ref{assumptions}, we get
\begin{align*}
& (\upsigma_{r}^2 \circ \upmu_{r}^{-1})'(k''-1) \leq  \frac{(\upsigma_{r}^{2} \circ \upmu_{r}^{-1}) (k''-1)- \upsigma_{r}^{2}((\eta_{r,1}+\frac{\delta_{r}}{2})r)}{k''-1-\upmu_{r}((\eta_{r,1}+\frac{\delta_{r}}{2})r)} \leq \frac{\upsigma_{r}^{2} ((\eta_{r,1}+\delta_{r})r)}{\upmu_{r}((\eta_{r,1}+\frac{3\delta_{r}}{4})r)-\upmu_{r}((\eta_{r,1}+\frac{\delta_{r}}{2})r)} \\
& \leq \frac{4\upsigma_{r}^{2} ((\eta_{r,1}+\delta_{r})r)}{\delta_{r}\inf_{x \in (\eta_{r,1}+\frac{\delta_{r}}{2},\eta_{r,1}+\frac{3\delta_{r}}{4})}r\upmu_{r}'(rx)} \leq \frac{\eta_{r,1}+\frac{3\delta_{r}}{4}}{\eta_{r,1}+\frac{\delta_{r}}{2}} \frac{4\upsigma_{r}^{2} ((\eta_{r,1}+\delta_{r})r)}{\delta_{r}r \upmu_{r}'((\eta_{r,1}+\frac{\delta_{r}}{2})r)} \leq C'',
\end{align*}
for all sufficiently large $r$, where $C''$ is independent of $r$. Therefore, the right-hand side of \eqref{lol14} is bounded by a constant $C'=C'(M)$ for all $r$ sufficiently large and all $\gamma\in[0,M]$, where $M>\sqrt{2/\mathfrak{a}}$ is arbitrary but fixed. Using also the fact that $\upsigma_{r}^2$ and $\upmu_{r}$ are increasing, we obtain
\begin{align}
& \mathbb P\left(\sup_{x\in ((\eta_{r,1}+\delta_{r}) r,(\eta_{r,2}-\delta_{r})r)}\frac{\mathrm{N}_{r}(x)-\upmu_{r}(x)}{\upsigma_{r}^2(x)}>\gamma\right)\leq 
\mathrm{C}C' e^\gamma \sum_{k \in \mathcal{K}_{r}} e^{-\frac{\gamma^2}{2} \upsigma_{r}^2(\upmu_{r}^{-1}(k))} \nonumber \\
&\leq \mathrm{C}C'e^\gamma\left(e^{-\frac{\gamma^2}{2} \upsigma_{r}^2((\eta_{r,1}+\delta_{r})r)}+ \int_{\upmu_{r}((\eta_{r,1}+\delta_{r})r)}^{\upmu_{r}((\eta_{r,2}-\delta_{r})r)+1} e^{-\frac{\gamma^2}{2}(\upsigma_{r}^2\circ\upmu_{r}^{-1})(k)}dk\right). \label{bound1 sup}
\end{align}
The proof for the other bound is similar (and simpler): for any $\gamma > 0$,
\begin{align}
& \mathbb P\left(\sup_{x\in ((\eta_{r,1}+\delta_{r}) r,(\eta_{r,2}-\delta_{r})r)}\frac{\upmu_{r}(x)-\mathrm{N}_{r}(x)}{\upsigma_{r}^2(x)}>\gamma\right) \leq \sum_{k \in \mathcal{K}_{r}} \mathbb P\left( \frac{\upmu_{r}(\kappa_{r,k-1})-\mathrm{N}_{r}(\kappa_{r,k-1})+1}{\upsigma_{r}^2(\kappa_{r,k-1})}>\gamma\right) \nonumber \\
& \leq \sum_{k: k+1 \in \mathcal{K}_{r}}\mathbb{E}\left[e^{-\gamma \mathrm{N}_{r}(\kappa_{r,k})}\right]e^{\gamma\upmu_{r}(\kappa_{r,k})+\gamma-\gamma^2 \upsigma_{r}^2(\kappa_{r,k})} \leq \mathrm{C} \, e^\gamma \sum_{k: k+1 \in \mathcal{K}_{r}} e^{-\frac{\gamma^2}{2} \upsigma_{r}^2(\upmu_{r}^{-1}(k))} \nonumber \\
& \leq \mathrm{C}e^\gamma\left(2e^{-\frac{\gamma^2}{2} (\upsigma_{r}^2(\upmu_{r}^{-1}(\upmu_{r}((\eta_{r,1}+\delta_{r})r)-1))}+ \int_{\upmu_{r}((\eta_{r,1}+\delta_{r})r)}^{\upmu_{r}((\eta_{r,2}-\delta_{r})r)} e^{-\frac{\gamma^2}{2}(\upsigma_{r}^2\circ\upmu_{r}^{-1})(k)}dk\right). \label{bound2 sup}
\end{align}
The inequalities \eqref{bound1 sup} and \eqref{bound2 sup} imply that
\begin{align*}
& \mathbb P\left(\sup_{x\in ((\eta_{r,1}+\delta_{r}) r,(\eta_{r,2}-\delta_{r})r)}\left|\frac{\mathrm{N}_{r}(x)-\upmu_{r}(x)}{\upsigma_{r}^2(x)}\right|>\gamma\right) \\
& \leq \mathrm{C}(C'+2)e^\gamma\left(e^{-\frac{\gamma^2}{2} (\upsigma_{r}^2\circ\upmu_{r}^{-1})(\upmu_{r}((\eta_{r,1}+\delta_{r})r)-1)}+ \int_{\upmu_{r}((\eta_{r,1}+\delta_{r})r)}^{\upmu_{r}((\eta_{r,2}-\delta_{r})r)+1} e^{-\frac{\gamma^2}{2}(\upsigma_{r}^2\circ\upmu_{r}^{-1})(k)}dk\right).
\end{align*}
The above inequality is valid for any $\gamma > 0$ but is useful only for $\gamma > \sqrt{2/\mathfrak{a}}$. Indeed, using part 3 of Assumptions \ref{assumptions}, we find
\begin{align*}
\mathbb P\left(\sup_{x\in ((\eta_{r,1}+\delta_{r}) r,(\eta_{r,2}-\delta_{r})r)}\left|\frac{\mathrm{N}_{r}(x)-\upmu_{r}(x)}{\upsigma_{r}^2(x)}\right|>\gamma\right) \leq 3\mathrm{C}(C'+2)e^M\frac{\upmu_{r}((\eta_{r,1}+\delta_{r}) r)^{1-\frac{\mathfrak{a} \gamma^2}{2}}}{\frac{\mathfrak{a} \gamma^2}{2}-1}
\end{align*}
for all sufficiently large $r$ and for all $\gamma \in (\sqrt{2/\mathfrak{a}},M]$. We obtain the claim with $c = 6\mathrm{C}(C'+2)e^M$ after choosing $\gamma = \sqrt{\frac{2}{\mathfrak{a}}(1+\epsilon)}$.
\end{proof}
\begin{lemma}\label{lemma: xk not far away from kappa k}
Let $\epsilon \in (0,1)$ and $\lambda > 1$. For all sufficiently large $r$, if the event
\begin{align}\label{event holds true}
\sup_{x\in ((\eta_{r,1}+\delta_{r}) r,(\eta_{r,2}-\delta_{r})r)}\left|\frac{\mathrm{N}_{r}(x)-\upmu_{r}(x)}{\upsigma_{r}^2(x)}\right| \leq   \sqrt{\frac{2}{\mathfrak{a}}(1+\epsilon)}
\end{align}
holds true, then
\begin{align}\label{upper and lower bound for mu xk}
\max_{k \in (\upmu_{r}((\eta_{r,1}+2\delta_{r}) r),\upmu_{r}((\eta_{r,2}-2\delta_{r})r)) \cap \mathbb{N}_{\geq 0}} \frac{|\upmu_{r}(\upxi_{r,k}) - k|}{(\upsigma_{r}^2\circ\upmu_{r}^{-1})(k)} \leq \sqrt{\frac{2}{\mathfrak{a}}(1+\lambda\epsilon)}.
\end{align}
\end{lemma}
\begin{proof}
First, we observe that if $\upxi_{r,k} \leq (\eta_{r,1}+\delta_{r}) r<(\eta_{r,1}+2\delta_{r}) r \leq \kappa_{r,k}$, then
\begin{align*}
\upmu_{r}((\eta_{r,1}+2\delta_{r}) r) \leq  \upmu_{r}(\kappa_{r,k})=k=\mathrm{N}_{r}(\upxi_{r,k}) \leq \mathrm{N}_{r}((\eta_{r,1}+\delta_{r}) r),
\end{align*}
and hence, using Assumptions \ref{assumptions}, 
\begin{align*}
& \frac{\mathrm{N}_{r}((\eta_{r,1}+\delta_{r}) r)-\upmu_{r}((\eta_{r,1}+\delta_{r}) r)}{\upsigma_{r}^2((\eta_{r,1}+\delta_{r}) r)}\geq \frac{\upmu_{r}((\eta_{r,1}+2\delta_{r}) r)-\upmu_{r}((\eta_{r,1}+\delta_{r}) r)}{\upsigma_{r}^2((\eta_{r,1}+\delta_{r}) r)} \\
& \geq \frac{\delta_{r} \inf_{x \in (\eta_{r,1}+\delta_{r},\eta_{r,1}+2\delta_{r})}r\upmu_{r}'(xr)}{\upsigma_{r}^2((\eta_{r,1}+\delta_{r}) r)}  \geq \frac{\eta_{r,1}+\delta_{r}}{\eta_{r,1}+2\delta_{r}}\frac{\delta_{r} r \upmu_{r}'((\eta_{r,1}+\delta_{r}) r)}{\upsigma_{r}^2((\eta_{r,1}+\delta_{r}) r)},
\end{align*}
and since the right-hand side tends to $+\infty$ as $r\to +\infty$, this contradicts \eqref{event holds true} if $r$ is large enough. Similarly, if $\upxi_{r,k}\geq (\eta_{r,2}-\delta_{r}) r >(\eta_{r,2}-2\delta_{r}) r \geq \kappa_{r,k}$, then 
\begin{align*}
\upmu_{r}((\eta_{r,2}-2\delta_{r})r) \geq  \upmu_{r}(\kappa_{r,k})=k=\mathrm{N}_{r}(\upxi_{r,k}) \geq \mathrm{N}_{r}((\eta_{r,2}-\delta_{r})r),
\end{align*}
and we find
\begin{align*}
& \frac{\upmu_{r}((\eta_{r,2}-\delta_{r})r)-\mathrm{N}_{r}((\eta_{r,2}-\delta_{r})r)}{\upsigma_{r}^2((\eta_{r,2}-\delta_{r})r)}\geq \frac{\upmu_{r}((\eta_{r,2}-\delta_{r})r)-\upmu_{r}((\eta_{r,2}-2\delta_{r})r)}{\upsigma_{r}^2((\eta_{r,2}-\delta_{r})r)}  \nonumber \\
& \geq \frac{\delta_{r}  \inf_{\xi \in (\eta_{r,2}-2\delta_{r},\eta_{r,2}-\delta_{r})}r\upmu_{r}'(r\xi)}{\sigma^2((\eta_{r,2}-\delta_{r})r)} \geq \frac{\eta_{r,2}-2\delta_{r}}{\eta_{r,2}-\delta_{r}}\frac{\delta_{r} r \upmu_{r}'((\eta_{r,2}-2\delta_{r})r)}{\upsigma_{r}^2((\eta_{r,2}-\delta_{r})r)}.
\end{align*}
The right-hand side tends to $+\infty$ as $r \to + \infty$ by Assumptions \ref{assumptions}, so this also contradicts \eqref{event holds true} if $r$ is large enough. Therefore, we have shown that $\upxi_{r,k} \in ((\eta_{r,1}+\delta_{r}) r, (\eta_{r,2}-\delta_{r})r)$ for all $k \in (\upmu_{r}((\eta_{r,1}+2\delta_{r}) r),\upmu_{r}((\eta_{r,2}-2\delta_{r})r))$, provided that $r$ is large enough. The rest of the proof follows as in \cite[proof of Lemma 2.2, eq (2.9) and below]{ChCl4}. In \cite{ChCl4}, the parameter $\lambda$ is equal to $2$. We consider here the sharper situation $\lambda>1$, but this improvement can be obtained without modifying the idea of the proof of \cite{ChCl4}, so we omit the details here.
\end{proof}
\begin{proof}[Proof of Theorem \ref{thm:rigidity}]
Let $\lambda> 1$. By Lemma \ref{lemma: A r eps}, there exist $c>0$ and $r_{0}>0$ such that
\begin{align*}
\mathbb{P}(A) \geq 1- \frac{c\, \upmu_{r}((\eta_{r,1}+\delta_{r}) r)^{-\frac{\epsilon}{\lambda}}}{\epsilon}, \qquad \mbox{for any } r>r_{0} \mbox{ and any small } \epsilon >0,
\end{align*}
and Lemma \ref{lemma: xk not far away from kappa k} implies that $A \subset B$, where $A$ and $B$ are the events
\begin{align*}
& A = \bigg\{\sup_{x\in ((\eta_{r,1}+\delta_{r}) r,(\eta_{r,2}-\delta_{r})r)}\left|\frac{\mathrm{N}_{r}(x)-\upmu_{r}(x)}{\upsigma_{r}^2(x)}\right| \leq   \sqrt{\frac{2}{\mathfrak{a}}\Big(1+\frac{\epsilon}{\lambda}\Big)}\bigg\}, \\
& B = \bigg\{\max_{k \in (\upmu_{r}((\eta_{r,1}+2\delta_{r}) r),\upmu_{r}((\eta_{r,2}-2\delta_{r})r))\cap \mathbb{N}_{\geq 0}} \frac{|\upmu_{r}(\upxi_{r,k}) - k|}{(\upsigma_{r}^2\circ\upmu_{r}^{-1})(k)} \leq \sqrt{\frac{2}{\mathfrak{a}}(1+\epsilon)}\bigg\}.
\end{align*} 
Hence $\mathbb{P}(B) \geq \mathbb{P}(A) \geq  1- \frac{c\, \upmu_{r}((\eta_{r,1}+\delta_{r}) r)^{-\frac{\epsilon}{\lambda}}}{\epsilon}$, which finishes the proof.
\end{proof}

\appendix

\section{Model RH problems}\label{Section:Appendix}
\subsection{Airy model RH problem}\label{subsec:Airy}
The following RH problem was introduced in \cite{DKMVZ1}, and its unique solution can be explicitly written in terms of Airy functions. 
\begin{itemize}
\item[(a)] $\Phi_{\mathrm{Ai}} : \mathbb{C} \setminus \Sigma_{\mathrm{Ai}} \rightarrow \mathbb{C}^{2 \times 2}$ is analytic, where $\Sigma_{\mathrm{Ai}}$ is shown in Figure \ref{figAiry} (left).
\item[(b)] $\Phi_{\mathrm{Ai}}$ has the jump relations
\begin{equation}\label{jumps P3}
\begin{array}{l l}
\Phi_{\mathrm{Ai},+}(z) = \Phi_{\mathrm{Ai},-}(z) \begin{pmatrix}
0 & 1 \\ -1 & 0
\end{pmatrix}, & \mbox{ on } \mathbb{R}^{-}, \\

\Phi_{\mathrm{Ai},+}(z) = \Phi_{\mathrm{Ai},-}(z) \begin{pmatrix}
 1 & 1 \\
 0 & 1
\end{pmatrix}, & \mbox{ on } \mathbb{R}^{+}, \\

\Phi_{\mathrm{Ai},+}(z) = \Phi_{\mathrm{Ai},-}(z) \begin{pmatrix}
 1 & 0  \\ 1 & 1
\end{pmatrix}, & \mbox{ on } e^{ \frac{2\pi i}{3} }  \mathbb{R}^{+} , \\

\Phi_{\mathrm{Ai},+}(z) = \Phi_{\mathrm{Ai},-}(z) \begin{pmatrix}
 1 & 0  \\ 1 & 1
\end{pmatrix}, & \mbox{ on }e^{ -\frac{2\pi i}{3} }\mathbb{R}^{+} . \\
\end{array}
\end{equation}
\item[(c)] As $z \to \infty$, $z \notin \Sigma_{\mathrm{Ai}}$, we have
\begin{equation}\label{Asymptotics Airy}
\Phi_{\mathrm{Ai}}(z) = z^{-\frac{\sigma_{3}}{4}}N \left( I + \frac{\Phi_{\mathrm{Ai,1}}}{z^{3/2}} + \bigO(z^{-3}) \right) e^{-\frac{2}{3}z^{3/2}\sigma_{3}},
\end{equation}
where
$$N = \frac{1}{\sqrt{2}} \begin{pmatrix} 1 & i \\ i & 1 \end{pmatrix}, \qquad \Phi_{\mathrm{Ai,1}} = \frac{1}{8}\begin{pmatrix}
\frac{1}{6} & i \\ i & -\frac{1}{6}
\end{pmatrix}.$$
As $z \to 0$, we have
\begin{equation}
\Phi_{\mathrm{Ai}}(z) = \bigO(1).
\end{equation} 
\end{itemize}

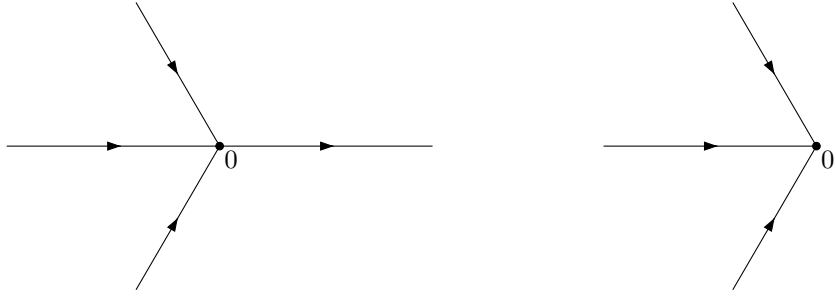
\begin{figure}[t]
\begin{center}
\begin{tikzpicture}
\node at (0.15,-0.18) {$0$};
\draw[fill] (0,0) circle (0.05);
\draw (-2.8,0)--(0,0);
\draw (0,0)--(120:2.2); \draw (0,0)--(-120:2.2);
\draw (0,0)--(2.8,0);
\draw[black,arrows={-Triangle[length=0.18cm,width=0.12cm]}]
(-1.3,0) --  ++(0:0.001);
\draw[black,arrows={-Triangle[length=0.18cm,width=0.12cm]}]
(120:1.1) --  ++(-60:0.001);
\draw[black,arrows={-Triangle[length=0.18cm,width=0.12cm]}]
(-120:1.1) --  ++(60:0.001);
\draw[black,arrows={-Triangle[length=0.18cm,width=0.12cm]}]
(0:1.5) --  ++(0:0.001);
\end{tikzpicture} \hspace{2cm}
\begin{tikzpicture}
\node at (0.15,-0.18) {$0$};
\draw[fill] (0,0) circle (0.05);
\draw (-2.8,0)--(0,0);
\draw (0,0)--(120:2.2); \draw (0,0)--(-120:2.2);
\draw[black,arrows={-Triangle[length=0.18cm,width=0.12cm]}]
(-1.3,0) --  ++(0:0.001);
\draw[black,arrows={-Triangle[length=0.18cm,width=0.12cm]}]
(120:1.1) --  ++(-60:0.001);
\draw[black,arrows={-Triangle[length=0.18cm,width=0.12cm]}]
(-120:1.1) --  ++(60:0.001);
\end{tikzpicture}
    \caption{\label{figAiry}Left: the jump contour $\Sigma_{\mathrm{Ai}}$ for $\Phi_{\mathrm{Ai}}$. Right: the jump contour $\Sigma_{\mathrm{Be}}$ for $\Phi_{\mathrm{Be}}$.}
\end{center}
\end{figure}

\subsection{Bessel model RH problem}\label{ApB}
Given $\alpha > -1$, we define
\begin{equation}\label{Phi explicit}
\Phi_{\mathrm{Be}}(z;\alpha)=
\begin{cases}
\begin{pmatrix}
I_{\alpha}(2z^{\frac{1}{2}}) & \frac{ i}{\pi} K_{\alpha}(2z^{\frac{1}{2}}) \\
2\pi i z^{\frac{1}{2}} I_{\alpha}^{\prime}(2 z^{\frac{1}{2}}) & -2 z^{\frac{1}{2}} K_{\alpha}^{\prime}(2 z^{\frac{1}{2}})
\end{pmatrix}, & |\arg z | < \frac{2\pi}{3}, \\

\begin{pmatrix}
\frac{1}{2} H_{\alpha}^{(1)}(2(-z)^{\frac{1}{2}}) & \frac{1}{2} H_{\alpha}^{(2)}(2(-z)^{\frac{1}{2}}) \\
\pi z^{\frac{1}{2}} \big( H_{\alpha}^{(1)} \big)^{\prime} (2(-z)^{\frac{1}{2}}) & \pi z^{\frac{1}{2}} \big( H_{\alpha}^{(2)} \big)^{\prime} (2(-z)^{\frac{1}{2}})
\end{pmatrix}e^{\frac{\pi i \alpha}{2}\sigma_{3}}, & \frac{2\pi}{3} < \arg z < \pi, \\

\begin{pmatrix}
\frac{1}{2} H_{\alpha}^{(2)}(2(-z)^{\frac{1}{2}}) & -\frac{1}{2} H_{\alpha}^{(1)}(2(-z)^{\frac{1}{2}}) \\
-\pi z^{\frac{1}{2}} \big( H_{\alpha}^{(2)} \big)^{\prime} (2(-z)^{\frac{1}{2}}) & \pi z^{\frac{1}{2}} \big( H_{\alpha}^{(1)} \big)^{\prime} (2(-z)^{\frac{1}{2}})
\end{pmatrix}e^{-\frac{\pi i \alpha}{2}\sigma_{3}}, & -\pi < \arg z < -\frac{2\pi}{3},
\end{cases}
\end{equation}
where $H_{\alpha}^{(1)}$ and $H_{\alpha}^{(2)}$ are the Hankel functions of the first and second kind, and $I_\alpha$ and $K_\alpha$ are the modified Bessel functions of the first and second kind.

\medskip \noindent The matrix-valued function $\Phi_{\mathrm{Be}}$ was first considered in \cite{DIZ} for $\alpha=0$ and in \cite{KMcLVAV} for general $\alpha >-1$. For fixed $\alpha$, it is known that $\Phi_{\mathrm{Be}}$ enjoys the following properties
\begin{itemize}
\item[(a)] $\Phi_{\mathrm{Be}} : \mathbb{C} \setminus \Sigma_{\mathrm{Be}} \to \mathbb{C}^{2\times 2}$ is analytic, where
$\Sigma_{\mathrm{Be}}$ is shown in Figure \ref{figAiry} (right).
\item[(b)] $\Phi_{\mathrm{Be}}$ satisfies the jump conditions
\begin{equation}\label{Jump for P_Be}
\begin{array}{l l} 
\Phi_{\mathrm{Be},+}(z) = \Phi_{\mathrm{Be},-}(z) \begin{pmatrix}
0 & 1 \\ -1 & 0
\end{pmatrix}, & z \in \mathbb{R}^{-}, \\

\Phi_{\mathrm{Be},+}(z) = \Phi_{\mathrm{Be},-}(z) \begin{pmatrix}
1 & 0 \\ e^{\pi i \alpha} & 1
\end{pmatrix}, & z \in e^{ \frac{2\pi i}{3} }  \mathbb{R}^{+}, \\

\Phi_{\mathrm{Be},+}(z) = \Phi_{\mathrm{Be},-}(z) \begin{pmatrix}
1 & 0 \\ e^{-\pi i \alpha} & 1
\end{pmatrix}, & z \in e^{ -\frac{2\pi i}{3} }  \mathbb{R}^{+}. \\
\end{array}
\end{equation}
\item[(c)] As $z \to \infty$, $z \notin \Sigma_{\mathrm{Be}}$, we have
\begin{equation}\label{large z asymptotics Bessel}
\Phi_{\mathrm{Be}}(z;\alpha) = ( 2\pi z^{\frac{1}{2}} )^{-\frac{\sigma_{3}}{2}}N
\left(I+\sum_{k=1}^{\infty} \widetilde{\Phi}_{\mathrm{Be},k}(\alpha) z^{-k/2}\right) e^{2z^{\frac{1}{2}}\sigma_{3}},
\end{equation}
for certain matrices $\widetilde{\Phi}_{\mathrm{Be},k}(\alpha)$ independent of $z$.
\item[(d)] As $z$ tends to 0, 
\begin{equation}\label{local behaviour near 0 of P_Be}
\begin{array}{l l}
\displaystyle \Phi_{\mathrm{Be}}(z;\alpha) = \left\{ \begin{array}{l l}
\begin{pmatrix}
\bigO(1) & \bigO(1) \\
\bigO(1) & \bigO(1) 
\end{pmatrix}z^{\frac{\alpha}{2}\sigma_{3}}, & |\arg z | < \frac{2\pi}{3}, \\
\begin{pmatrix}
\bigO(z^{-\frac{\alpha}{2}}) & \bigO(z^{-\frac{\alpha}{2}}) \\
\bigO(z^{-\frac{\alpha}{2}}) & \bigO(z^{-\frac{\alpha}{2}}) 
\end{pmatrix}, & \frac{2\pi}{3}<|\arg z | < \pi,
\end{array} \right.  & \displaystyle \mbox{ if } \alpha > 0.
\end{array}
\end{equation}
\end{itemize}
Also, in \cite[eqs (7.5) and (7.6)]{CharlierBessel} it was shown that, for $z$ in a neighborhood of $0$,
\begin{equation}\label{precise asymptotics of Phi Bessel near 0}
\Phi_{\mathrm{Be}}(z;\alpha) = \Phi_{\mathrm{Be},0}(z;\alpha)z^{\frac{\alpha}{2}\sigma_{3}}\begin{pmatrix}
1 & h(z) \\ 0 & 1
\end{pmatrix} H_{0}(z),
\end{equation}
where $H_{0}$ is given by \eqref{def of H}, $h$ by \eqref{def of h}, and $\Phi_{\mathrm{Be},0}$ is analytic in a neighborhood of $0$ and satisfies
\begin{equation}\label{precise matrix at 0 in asymptotics of Phi Bessel near 0}
\Phi_{\mathrm{Be},0}(0;\alpha) = \left\{ 
\begin{array}{l l}
\begin{pmatrix}
\frac{1}{\Gamma(1+\alpha)} & \frac{i \Gamma(\alpha)}{2\pi} \\
\frac{i \pi}{\Gamma(\alpha)} & \frac{\Gamma(1+\alpha)}{2}
\end{pmatrix}, & \mbox{if } \alpha \neq 0, \\[0.5cm]
\begin{pmatrix}
1 & \frac{\gamma_{\mathrm{E}}}{\pi i} \\ 0 & 1
\end{pmatrix}, & \mbox{if } \alpha = 0,
\end{array} \right.
\end{equation}
where $\gamma_{\mathrm{E}}$ is Euler's gamma constant.

We emphasize that the asymptotics \eqref{large z asymptotics Bessel} are valid only as $z \to \infty$ while $\alpha$ remains in a compact subset of $(-1,+\infty)$. Our task in the rest of this section is to find uniform asymptotics for $\Phi_{\mathrm{Be}}(z,\alpha)$ as $z \to \infty$ and simultaneously $\alpha \to +\infty$. The uniform asymptotics of $I_{\alpha}(\alpha z)$ and $K_{\alpha}(\alpha z)$ as $z\to \infty$, $\re z > 0$, and $\alpha \to + \infty$ are given in \cite[Chapter 10, equations (7.18)--(7.19)]{Olver}. As $z \to \infty$ in the right half-plane $\re z > 0$ and simultaneously $\alpha \to +\infty$, we have 
\begin{align}
& I_{\alpha}(\alpha z) = \frac{e^{\alpha \xi(z)}}{(2\pi \alpha)^{1/2}(1+z^{2})^{1/4}}\bigg( 1 + \frac{3p(z)-5p(z)^{3}}{24\alpha} + \bigO\Big(\frac{1}{(\alpha z)^{2}}\Big) \bigg), \label{Ia large} \\
& K_{\alpha}(\alpha z) = \Big( \frac{\pi}{2\alpha} \Big)^{1/2} \frac{e^{-\alpha \xi(z)}}{(1+z^{2})^{1/4}} \bigg( 1- \frac{3p(z)-5p(z)^{3}}{24\alpha} + \bigO\Big(\frac{1}{(\alpha z)^{2}}\Big) \bigg), \label{Ka large}
\end{align}
where the error terms are uniform for $|\arg z| \leq \frac{\pi}{2} - \delta$ for each fixed $\delta > 0$ and
\begin{align}\label{def of p and xi}
p(z) = \frac{1}{\sqrt{1+z^{2}}}, \qquad \xi(z) = \sqrt{1+z^{2}} + \log \frac{z}{1+\sqrt{1+z^{2}}}.
\end{align}
In \eqref{def of p and xi}, the principal branch is used for the square roots and the logarithm.
To evaluate the asymptotics of $I_{\alpha}'(\alpha z)$ and $K_{\alpha}'(\alpha z)$, we use \eqref{Ia large} and \eqref{Ka large} together with
\begin{align*}
& I_{\alpha}'(\alpha z) = \frac{I_{\alpha+1}(\alpha z)+I_{\alpha-1}(\alpha z)}{2}, \qquad K_{\alpha}'(\alpha z) = -K_{\alpha-1}(\alpha z)- \frac{1}{z}K_{\alpha}(\alpha z), \\
& e^{(\alpha+1)\xi(\frac{\alpha z}{\alpha +1})} = e^{\alpha \xi(z)} \frac{z}{1+\sqrt{1+z^{2}}} \bigg( 1-\frac{p(z)}{2\alpha} + \bigO\Big( \frac{1}{(\alpha z)^{2}} \Big) \bigg), \\
& e^{(\alpha-1)\xi(\frac{\alpha z}{\alpha -1})} = e^{\alpha \xi(z)} \frac{1+\sqrt{1+z^{2}}}{z} \bigg( 1-\frac{p(z)}{2\alpha} + \bigO\Big( \frac{1}{(\alpha z)^{2}} \Big) \bigg).
\end{align*}
As $z \to \infty$, $\re z > 0$, and simultaneously $\alpha \to +\infty$, we obtain
\begin{align}
& I_{\alpha}'(\alpha z) = \frac{e^{\alpha \xi(z)}(1+z^{2})^{1/4}}{(2\pi \alpha)^{1/2} z} \bigg( 1 + \frac{-9p(z)+7p(z)^{3}}{24\alpha} + \bigO\Big(\frac{1}{(\alpha z)^{2}}\Big) \bigg), \label{Ia der large} \\
& K_{\alpha}'(\alpha z) = \Big( \frac{\pi}{2\alpha} \Big)^{1/2} \frac{e^{-\alpha \xi(z)}(1+z^{2})^{1/4}}{-z}\bigg( 1 + \frac{9p(z)-7p(z)^{3}}{24\alpha} + \bigO\Big(\frac{1}{(\alpha z)^{2}}\Big) \bigg). \label{Ka der large}
\end{align}
Setting $z \mapsto 2\sqrt{z}\alpha^{-1}$ in \eqref{Ia large}--\eqref{Ka der large}, we obtain, as $\alpha \to +\infty$ and $\sqrt{z}\alpha^{-1} \to \infty$ uniformly for $|\arg z | < \frac{2\pi}{3}$,
\begin{align}\label{asymp for Phi Be with large alpha}
\Phi_{\mathrm{Be}}(z; \alpha) = \big( \sqrt{\pi}(\alpha^{2}+4z)^{\frac{1}{4}} \big)^{-\sigma_{3}}N \bigg( I+ \Phi_{\mathrm{Be},1}(z;\alpha)z^{-1/2} + \bigO(z^{-1}) \bigg)e^{\alpha \xi(\frac{2\sqrt{z}}{\alpha})\sigma_{3}},
\end{align}
where $\Phi_{\mathrm{Be},1}(z;\alpha)$ is uniformly bounded and given by
\begin{align}\label{def of PhiBe1}
\Phi_{\mathrm{Be},1}(z;\alpha) = \sqrt{z} \begin{pmatrix}
\frac{-3p(\frac{2\sqrt{z}}{\alpha})+p(\frac{2\sqrt{z}}{\alpha})^{3}}{24\alpha} & \frac{i(-p(\frac{2\sqrt{z}}{\alpha})+p(\frac{2\sqrt{z}}{\alpha})^{3})}{4\alpha} \\
\frac{i(-p(\frac{2\sqrt{z}}{\alpha})+p(\frac{2\sqrt{z}}{\alpha})^{3})}{4\alpha} & \frac{3p(\frac{2\sqrt{z}}{\alpha})-p(\frac{2\sqrt{z}}{\alpha})^{3}}{24\alpha}
\end{pmatrix}.
\end{align}
In fact, the asymptotic formula \eqref{asymp for Phi Be with large alpha} is valid as $\alpha \to +\infty$ and $\sqrt{z}\alpha^{-1} \to \infty$ uniformly for all values of $\arg z$ (not just in the sector $|\arg z | < \frac{2\pi}{3}$).
To see this, we employ the identities (see \cite[eq (10.27.8)]{NIST})
\begin{align*}
& H_\alpha^{(1)}(z) = \frac{2}{\pi}e^{-\frac{\pi i}{2}(\alpha+1)} K_\alpha(-iz), & & -\frac{\pi}{2} < \arg(z) \leq \pi,
	\\
& H_\alpha^{(2)}(z) = \frac{2}{\pi} e^{\frac{\pi i}{2}(\alpha+1)}  K_\alpha(iz), & & -\pi < \arg(z) \leq \frac{\pi}{2},
\end{align*}
to write 
\begin{equation}\label{Phi explicit2}
\Phi_{\mathrm{Be}}(z; \alpha)=
\begin{cases}
\begin{pmatrix}
-\frac{i}{\pi} K_{\alpha}(-i2(-z)^{\frac{1}{2}}) & \frac{i}{\pi}K_{\alpha}(i2(-z)^{\frac{1}{2}}) \\
-2 z^{\frac{1}{2}}  K_{\alpha}^{\prime} (-i2(-z)^{\frac{1}{2}}) & -2 z^{\frac{1}{2}} K_{\alpha}^{\prime} (i2(-z)^{\frac{1}{2}})
\end{pmatrix}, & \frac{2\pi}{3} < \arg z < \pi, \\

\begin{pmatrix}
\frac{i}{\pi} K_{\alpha}(i2(-z)^{\frac{1}{2}}) & \frac{i}{\pi} K_{\alpha}(-i2(-z)^{\frac{1}{2}}) \\
2 z^{\frac{1}{2}} K_{\alpha}^{\prime} (i2(-z)^{\frac{1}{2}}) & -2 z^{\frac{1}{2}} K_{\alpha}^{\prime} (-i2(-z)^{\frac{1}{2}})
\end{pmatrix}, & -\pi < \arg z < -\frac{2\pi}{3}.
\end{cases}
\end{equation}
We now observe that the asymptotic formulas \eqref{Ka large} and \eqref{Ka der large} 
for $K_\alpha(\alpha z)$ and $K_\alpha'(\alpha z)$ are in fact valid uniformly for all $|\arg z| \leq \pi$ as $z \to \infty$ and $\alpha \to +\infty$, provided that the functions $(1+z^{2})^{1/4}$, $p(z)$ and $\xi(z)$ are analytically continued in the natural way as $z$ crosses the imaginary axis, see \cite[Chapter 10, Section 8.2]{Olver}. (The formulas \eqref{Ia large} and \eqref{Ia der large} cannot be analytically continued in the same way.)
Substituting the asymptotics \eqref{Ka large} and \eqref{Ka der large} into \eqref{Phi explicit2}, we conclude after a long calculation that \eqref{asymp for Phi Be with large alpha} holds as $\alpha \to +\infty$ and $\sqrt{z}\alpha^{-1} \to \infty$ uniformly also for $\frac{2\pi}{3} < \arg z < \pi$ and $-\pi < \arg z < -\frac{2\pi}{3}$.

\subsection{Confluent hypergeometric model RH problem}\label{subsection: model RHP with HG functions}
The following RH problem depends on a parameter $\beta \in i \mathbb{R}$ and was introduced in \cite{ItsKrasovsky}. Its unique solution can be explicitly written in terms of confluent hypergeometric function.
\begin{itemize}
\item[(a)] $\Phi_{\mathrm{HG}} : \mathbb{C} \setminus \Sigma_{\mathrm{HG}} \rightarrow \mathbb{C}^{2 \times 2}$ is analytic, where $\Sigma_{\mathrm{HG}}$ is shown in Figure \ref{Fig:HG}.
\item[(b)] For $z \in \Gamma_{k}$ (see Figure \ref{Fig:HG}), $k = 1,\dots,6$, $\Phi_{\mathrm{HG}}$ obeys the jump relations
\begin{equation}\label{jumps PHG3}
\Phi_{\mathrm{HG},+}(z) = \Phi_{\mathrm{HG},-}(z)J_{k},
\end{equation}
where
\begin{align*}
& J_{1} = \begin{pmatrix}
0 & e^{-i\pi \beta} \\ -e^{i\pi\beta} & 0
\end{pmatrix}, \quad J_{2} = \begin{pmatrix}
1 & 0 \\ e^{i\pi\beta} & 1
\end{pmatrix}, \quad J_{3} = \begin{pmatrix}
1 & 0 \\ e^{-i\pi\beta} & 1
\end{pmatrix} \quad  \\
& J_{4} = \begin{pmatrix}
0 & e^{i\pi\beta} \\ -e^{-i\pi\beta} & 0
\end{pmatrix}, \quad J_{5} = \begin{pmatrix}
1 & 0 \\ e^{-i\pi\beta} & 1
\end{pmatrix}, \quad J_{6} = \begin{pmatrix}
1 & 0 \\ e^{i\pi\beta} & 1
\end{pmatrix}.
\end{align*}
\item[(c)] As $z \to \infty$, $z \notin \Sigma_{\mathrm{HG}}$, we have
\begin{equation}\label{Asymptotics HG}
\Phi_{\mathrm{HG}}(z) = \left( I +  \frac{\Phi_{\mathrm{HG},1}(\beta)}{z} + \bigO(z^{-2}) \right) z^{-\beta\sigma_{3}}e^{-\frac{z}{2}\sigma_{3}}\left\{ \begin{array}{l l}
\displaystyle e^{i\pi\beta  \sigma_{3}}, & \displaystyle \frac{\pi}{2} < \arg z <  \frac{3\pi}{2}, \\
\begin{pmatrix}
0 & -1 \\ 1 & 0
\end{pmatrix}, & \displaystyle -\frac{\pi}{2} < \arg z < \frac{\pi}{2},
\end{array} \right.
\end{equation}
where 
\begin{equation}\label{def of tau}
\Phi_{\mathrm{HG},1}(\beta) = \beta^{2} \begin{pmatrix}
-1 & \tau(\beta) \\ - \tau(-\beta) & 1
\end{pmatrix}, \qquad \tau(\beta) = \frac{- \Gamma\left( -\beta \right)}{\Gamma\left( \beta + 1 \right)}.
\end{equation}
In \eqref{Asymptotics HG}, the root is defined by $z^{\beta} = |z|^{\beta}e^{i\beta \arg z}$ with $\arg z \in (-\frac{\pi}{2},\frac{3\pi}{2})$.

As $z \to 0$, we have
\begin{equation}\label{lol 35}
\Phi_{\mathrm{HG}}(z) = \left\{ \begin{array}{l l}
\begin{pmatrix}
\bigO(1) & \bigO(\log z) \\
\bigO(1) & \bigO(\log z)
\end{pmatrix}, & \mbox{if } z \in II \cup V, \\
\begin{pmatrix}
\bigO(\log z) & \bigO(\log z) \\
\bigO(\log z) & \bigO(\log z)
\end{pmatrix}, & \mbox{if } z \in I\cup III \cup IV \cup VI.
\end{array} \right.
\end{equation}
\end{itemize}
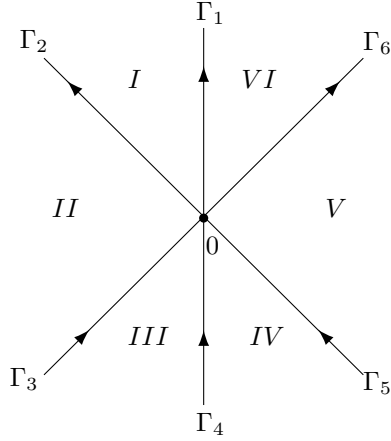
\begin{figure}[t!]
    \begin{center}
    \setlength{\unitlength}{1truemm}
    \begin{picture}(100,55)(-5,10)
             
        \put(50,39.8){\thicklines\circle*{1.2}}
        \put(50,40){\line(-0.5,0.5){21}}
        \put(50,40){\line(-0.5,-0.5){21}}
        \put(50,40){\line(0.5,0.5){21}}
        \put(50,40){\line(0.5,-0.5){21}}
        \put(50,40){\line(0,1){25}}
        \put(50,40){\line(0,-1){25}}
        
        \put(50.3,35){$0$}
        \put(71,62){$\Gamma_6$}
        \put(49,66){$\Gamma_1$}
        \put(25.8,62.3){$\Gamma_2$}        
        \put(24.5,17.5){$\Gamma_3$}
        \put(49,11.5){$\Gamma_4$}
        \put(71,17){$\Gamma_5$}
        
        \put(32,58){\thicklines\vector(-0.5,0.5){.0001}}
        \put(35,25){\thicklines\vector(0.5,0.5){.0001}}
        \put(68,58){\thicklines\vector(0.5,0.5){.0001}}
        \put(65,25){\thicklines\vector(-0.5,0.5){.0001}}
        \put(50,60){\thicklines\vector(0,1){.0001}}
        \put(50,25){\thicklines\vector(0,1){.0001}}
        \put(40,57){$I$}
        \put(30,40){$II$}
        \put(40,23){$III$}
        \put(56,23){$IV$}
        \put(66,40){$V$}
        \put(55,57){$VI$}
    \end{picture}
    \caption{\label{Fig:HG}The jump contour $\Sigma_{\mathrm{HG}}$. For each $k=1,\ldots,6$, the angle formed by $\Gamma_{k}$ and $(0,+\infty)$ is a multiple of $\frac{\pi}{4}$.}
\end{center}
\end{figure}
We will need the following more detailed asymptotics: as $z \to 0$, $z \in II$, we have
\begin{equation}\label{precise asymptotics of Phi HG near 0}
\Phi_{\mathrm{HG}}(z) = \begin{pmatrix}
\Psi_{11} & \Psi_{12} \\ \Psi_{21} & \Psi_{22}
\end{pmatrix} (I + \bigO(z)) \begin{pmatrix}
1 & \frac{\sin (\pi \beta)}{\pi} \log z \\
0 & 1
\end{pmatrix},
\end{equation}
where the argument of $\log z = \log |z| + i \arg z$ is such that $\arg z \in \big(\small{-}\frac{\pi}{2},\frac{3\pi}{2}\big)$, and
\begin{align*}
& \Psi_{11} = \Gamma(1-\beta), \qquad \Psi_{12} = \frac{1}{\Gamma(\beta)} \left( \frac{\Gamma^{\prime}(1-\beta)}{\Gamma(1-\beta)}+2\gamma_{\mathrm{E}} - i \pi \right), \\
& \Psi_{21} = \Gamma(1+\beta), \qquad \Psi_{22} = \frac{-1}{\Gamma(-\beta)} \left( \frac{\Gamma^{\prime}(-\beta)}{\Gamma(-\beta)} + 2\gamma_{\mathrm{E}} - i \pi \right),
\end{align*}
where $\gamma_{\mathrm{E}}$ is Euler's gamma constant.

\end{document}